\RequirePackage[english]{babel}
\documentclass[a4paper,11pt]{amsart} 
\usepackage[latin1]{inputenc}
\usepackage{amsmath}
\usepackage{amssymb}
\usepackage{amsfonts}

\usepackage{enumerate}

\renewcommand{\labelenumi}{\theenumi}

\usepackage[matrix,arrow,cmtip,frame,graph]{xy} 
\SelectTips{cm}{}
\newdir{(}{{}*!/-5pt/@^{(}}
\newdir{(x}{{}*!/-5pt/@_{(}}
\newdir{+}{{}*!/-9pt/{}}
\newdir{>+}{@{>}*!/-9pt/{}}
\entrymodifiers={+!!<0pt,\fontdimen22\textfont2>}

\usepackage{mytheorems2}

\newcommand{\pref}[2][]{%
\ifx\\#1\\%
\eqref{#2}%
\else%
{\upshape(\ref{#2},~\ref{#1})}%
\fi%
}

\newcommand{\upref}[1]{{\upshape\ref{#1}}}

\newcommand\N{\mathbb{N}}
\newcommand\Z{\mathbb{Z}}
\newcommand\Q{\mathbb{Q}}
\newcommand\C{\mathbb{C}}
\newcommand\F{\mathbb{F}}

\newcommand\Hilb{\mathrm{Hilb}}

\newcommand\FHilb{\mathcal{H}\mathit{ilb}}
\DeclareRobustCommand\FGamma{\underline{\Gamma}}
\newcommand\fF{\mathcal{F}} 
\newcommand\fG{\mathcal{G}} 

\newcommand\Sets{\mathbf{Sets}}
\newcommand\Mod{\mathbf{Mod}} 
\newcommand\Alg{\mathbf{Alg}} 
\newcommand\xMod{\text{\nobreakdash--}\Mod} 
\newcommand\xAlg{\text{\nobreakdash--}\Alg} 
\newcommand\AffSch{\mathbf{Aff}} 

\newcommand\sO{\mathcal{O}}

\newcommand\inj{\hookrightarrow}
\newcommand\surj{\twoheadrightarrow}
\newcommand\iso{\cong} 
\newcommand\mapname[1]{#1\,:\,}
\newcommand\map[3]{#1\,:\, #2\rightarrow #3}
\newcommand\injmap[3]{#1\,:\, #2\hookrightarrow #3}

\newcommand\id[1]{\mathrm{id}_{#1}} 

\newcommand{\etale}{\'{e}tale}

\newcommand\A[1]{\mathbb{A}^{#1}}    
\newcommand\Spec{\mathrm{Spec}}
\newcommand\red{\mathrm{red}}
\newcommand\reg{\mathrm{reg}}
\newcommand\hensel[1]{{^{\mathrm{h}}{#1}}}

\newcommand\ceil[1]{\left\lceil#1\right\rceil}
\newcommand\divides{\mid}
\newcommand\I{\mathcal{I}} 

\newcommand\im{\mathfrak{m}} 

\newcommand\SG[1]{{\mathfrak{S}_{#1}}}   
\DeclareMathOperator{\Hom}{Hom}
\DeclareMathOperator{\Sym}{Sym}

\newcommand\mysetminus{\setminus}

\newcommand\bigtimes{\mathop{\times}}

\newcommand\sA{\mathcal{A}} 

\newcommand\fpr{\mathrm{fpr}} 
\newcommand\qreg{\mathrm{qreg}} 
\newcommand\nondeg{\mathrm{nondeg}} 

\newcommand\CH{\mathrm{CH}} 

\newcommand\T{\mathrm{T}}
\newcommand\TS{\mathrm{TS}}
\newcommand\SYM{\mathrm{S}}    
\newcommand\Hompl[1]{\mathrm{Pol}^{#1}}
\newcommand\Hommpl[1]{{\mathrm{Pol}_{\mathrm{mult}}^{#1}}}

\newcommand\FCI{\mathrm{Image}} 
\newcommand\FCS{\mathrm{Supp}} 

\newcommand{\AF}{{\upshape AF}} 

\newcommand{\groupoidIL}[4]{\xymatrix@1@M=0mm@C=7mm{#1%
\ar@<.5ex>@{+->+}[r]^-{#3} \ar@<-.5ex>@{+->+}[r]_-{#4} & #2}}

\begin{document}

\title[Families of zero cycles]
{Families of zero cycles and divided powers: I.\\ Representability}
\author[D. Rydh]{David Rydh}
\address{Department of Mathematics, KTH, 100 44 Stockholm, Sweden}
\email{dary@math.kth.se}
\date{2008-03-04}




\begin{abstract}
Let $X/S$ be a separated algebraic space. We construct an algebraic space
$\Gamma^d(X/S)$, \emph{the space of divided powers}, which parameterizes zero
cycles of degree $d$ on $X$. When $X/S$ is affine, this space is affine and
given by the spectrum of the ring of divided powers.
%
%
In characteristic zero or when $X/S$ is flat, the constructed space coincides
with the symmetric product $\Sym^d(X/S)$.
%
%
We also prove several fundamental results on the kernels of multiplicative
polynomial laws necessary for the construction of $\Gamma^d(X/S)$.
\end{abstract}

\maketitle

\setcounter{secnumdepth}{0}
\begin{section}{Introduction}

\renewcommand\addcontentsline[3]{}

Chow varieties, parameterizing families of cycles of a certain dimension and
degree, are classically constructed using explicit projective
methods\ \cite{chow-vdW,samuel}. Moreover, Chow varieties are defined as
\emph{reduced} schemes and in positive characteristic the classical
construction has the unpleasant property that it depends on a given projective
embedding~\cite{nagata_chownorm}.

Many attempts to give a nice functorial description of Chow varieties have been
made and some successful steps towards this goal have been taken. For families
parameterized by \emph{seminormal} schemes, Koll\'ar, Suslin and Voevodsky,
have given a functorial
description~\cite{kollar_rat_curves_book,voevodsky_cycles}. In characteristic
zero, Barlet~\cite{barlet} has given an analytic description over
\emph{reduced} $\C$-schemes and Ang\'eniol~\cite{angeniol_thesis} has given an
algebraic description over, not necessarily reduced, $\Q$-schemes. The
situation in characteristic zero is simplified by the fact that for a finite
extension $A\inj B$ such that the determinant $B\to A$ is defined, the
determinant is determined by the trace.

In this article we will restrict our attention to Chow varieties of zero
cycles, that is, families of cycles of relative dimension zero. We will
construct an algebraic space $\Gamma^d(X/S)$, parameterizing zero cycles, which
coincides with Ang\'eniol's Chow space in characteristic zero. As with
Ang\'eniol's Chow space, the algebraic space $\Gamma^d(X/S)$ is not always
reduced but its reduction coincides with the classical Chow variety if we use a
\emph{sufficiently good} projective embedding. The relation with the Chow
variety will be discussed in a subsequent
article~\cite{rydh_gammasymchow_inprep}. A good understanding of families of
zero cycles is crucial for the understanding of families of higher-dimensional
cycles. In fact, a family of higher-dimensional cycles is defined by giving
zero-dimensional families on ``smooth
projections''~\cite{barlet,rydh_famofcycles}.

A natural candidate parameterizing zero cycles is the symmetric product
$\Sym^d(X/S)=(X/S)^d/\SG{d}$. This is the correct choice, in the sense
that it coincides with $\Gamma^d(X/S)$, when $X$ is of characteristic zero or
when $X/S$ is flat. In general, however, $\Sym^d(X/S)$ is not functorially
well-behaved and should be replaced with the ``scheme of divided powers''. In
the affine case, this is the spectrum of the algebra of divided power
$\Gamma^d_A(B)$ and it coincides with the symmetric product when $d!$ is
invertible in $A$ or when $B$ is a flat $A$-algebra.

Although the ring of divided powers $\Gamma^d_A(B)$ and multiplicative
polynomial laws have been studied by many
authors~\cite{roby_lois_pol,roby_lois_pol_mult,bergmann_formen_auf_moduln,
ziplies_gamma_hat,ferrand_norme}, there are some important
results missing. We provide these missing parts, giving a full treatment of the
\emph{kernel} of a multiplicative law.
Somewhat surprisingly, the kernel does not commute with flat base change,
except in characteristic zero. We will show that the kernel does commute with
\etale{} base change.

After this preliminary study of $\Gamma^d_A(B)$
we define,
for any separated algebraic space $X/S$, a functor $\FGamma^d_{X/S}$ which
parameterizes families of zero cycles. From the definition of $\FGamma^d_{X/S}$
and the results on the kernel of a multiplicative law,
it will be obvious that $\FGamma^d_{X/S}$ is represented by
$\Spec(\Gamma^d_A(B))$ in the affine case. If $X/S$ is a scheme such that for
every $s\in S$, every finite subset of the fiber $X_s$ is contained in an
\emph{affine} open subset of $X$, then we say that $X/S$ is an \AF{}-scheme,
cf.\ Appendix~\ref{SS:AF}. In particular, this is the case if $X/S$ is
quasi-projective. For an \AF{}-scheme $X/S$ it is easy to show that
$\FGamma^d_{X/S}$ is representable by a scheme.

To treat the general case --- when $X/S$ is any separated scheme or separated
algebraic space --- we use the fact that $\FGamma^d_{X/S}$ is functorial in
$X$: For any morphism $\map{f}{U}{X}$ there is an induced \emph{push-forward}
$\map{f_*}{\FGamma^d_{U/S}}{\FGamma^d_{X/S}}$.
We show that when $f$ is \etale{}, then $f_*$ is \etale{} over a certain open
subset corresponding to families of cycles which are \emph{regular} with
respect to $f$.
We then show that $\FGamma^d_{X/S}$ is represented by an algebraic space
$\Gamma^d(X/S)$ giving an explicit \etale{} covering.

In the last part of the article we introduce ``addition of cycles'' and
investigate the relation between the symmetric product $\Sym^d(X/S)$ and the
algebraic space $\Gamma^d(X/S)$. Intuitively, the universal family of
$\Gamma^d(X/S)$ should be related to the addition of cycles morphism
$\map{\Psi_{X/S}}{\Gamma^{d-1}(X/S)\times_S X}{\Gamma^d(X/S)}$. In the special
case when $\Psi_{X/S}$ is \emph{flat}, e.g., when $X/S$ is a smooth curve,
Iversen has shown that the universal family is given by the \emph{norm} of
$\Psi_{X/S}$~\cite{iversen_lin_det}. In general, there is a similar but more
subtle description. The universal family and some other properties of
$\Gamma^d(X/S)$ are treated in~\cite{rydh_famzerocycles-II}.

We now discuss the results and methods in more detail:

\begin{subsection}{Multiplicative polynomial laws}
In~\S\ref{S:Gamma} we recall the basic properties of the algebra of divided
powers $\Gamma_A(B)$ and the algebra $\Gamma^d_A(B)$. We also mention the
\emph{universal multiplication} of laws which later on will be described
geometrically as \emph{addition of cycles}.
\end{subsection}

\begin{subsection}{Kernel of a multiplicative polynomial law}
Let $B$ be an $A$-algebra. In~\S\ref{S:support} the basic properties of the
\emph{kernel} $\ker(F)$ of a multiplicative law $\map{F}{B}{A}$ is
established. First we show that $B/\ker(F)$ is integral over $A$ using
Cayley-Hamilton's theorem. We then show that the kernel commutes with limits,
localization and smooth base change. As mentioned above, the kernel does not
commute with flat base change in general and showing that the kernel commutes
with smooth base change takes some effort. Finally, we show some topological
properties of the kernel: The radical of the kernel commutes with arbitrary
base change, the fibers of $\Spec(B/\ker(F))\to\Spec(A)$ are finite sets, and
$\Spec(B/\ker(F))\to\Spec(A)$ is universally open.
\end{subsection}

\begin{subsection}{The functor $\FGamma^d_{X/S}$}
Guided by the knowledge that $\Gamma^d_A(B)$ is what we want in the
affine case, we define in~\S\ref{SS:FGamma} a well-behaved functor
$\FGamma^d_{X/S}$ parameterizing families of zero cycles of degree $d$ as
follows. A family over an affine $S$-scheme $T=\Spec(A)$ is given by the
following data
\begin{enumerate}
\item A closed subspace $Z\inj X\times_S T$ such that $Z\to T$
is \emph{integral}.
In particular $Z=\Spec(B)$ is \emph{affine}.
\item A family $\alpha$ on $Z$, i.e., a morphism
$T\to\Gamma^d(Z/T):=\Spec(\Gamma^d_A(B))$.
\end{enumerate}
Moreover, two families are equivalent if they are both induced by a family for
some common smaller subspace $Z$. We often suppress the subspace $Z$ and talk
about the family $\alpha$. The smallest subspace $Z\inj X\times_S T$ in the
equivalence class containing $\alpha$ is the \emph{image} of the family
$\alpha$ and the reduction $Z_\red$ of the image is the \emph{support} of the
family. The image of $\alpha$ is given by the kernel of the multiplicative law
corresponding to $\alpha$. Since the kernel commutes with \etale{} base change,
as shown in~\S\ref{S:support}, so does the image of a family. This is the key
result needed to show that $\FGamma^d_{X/S}$ is a sheaf in the \etale{}
topology.

In contrast to the Hilbert functor, for which families over $T$ are determined
by a subspace $Z\inj X\times_S T$, a family of zero cycles is not determined by
its image $Z$. If $T$ is \emph{reduced}, then the image $Z$ of a family
parameterized by $T$ is reduced and the family is determined by an effective
cycle supported on $Z$. In positive characteristic, over non-perfect fields,
this cycle may have rational coefficients. This is discussed
in~\cite{rydh_famzerocycles-II}.
\end{subsection}

\begin{subsection}{Push-forward of cycles}
A morphism $\map{f}{X}{Y}$ of separated algebraic spaces induces a natural
transformation $\map{f_*}{\FGamma^d_{X/S}}{\FGamma^d_{Y/S}}$ which we call the
\emph{push-forward}. When $Y/S$ is locally of finite type, the existence of
$f_*$ follows from standard results. In general, we need a technical result on
integral morphisms given in Appendix~\ref{SS:tff-result}.

We say that a family $\alpha\in\FGamma^d_{X/S}(T)$ is \emph{regular} if the
restriction of $f_T$ to the image of $\alpha$ is an isomorphism. If
$\map{f}{X}{Y}$ is \etale{} then the regular locus is an open subfunctor of
$\FGamma^d_{X/S}$. A main result is that under certain regularity constraints,
push-forward commutes with products, cf.\ Proposition~\pref{P:pushfwd-diagram}.
Using this fact we show that the push-forward along an \etale{} morphism is
representable and \etale{} over the regular locus. This is
Proposition~\pref{P:pushfwd-over-etale-and-reg-is-etale}.
\end{subsection}

\begin{subsection}{Representability}
The representability of $\FGamma^d_{X/S}$ when $X/S$ is affine or \AF{} is, as
already mentioned, not difficult and given in~\ref{SS:FGamma}. When $X/S$ is
any separated algebraic space, the representability is proven in
Theorem~\pref{T:Gamma-representable:alg-space} using the results on the
push-forward.
\end{subsection}

\begin{subsection}{Addition of cycles}
Using the push-forward we define in~\S\ref{SS:addition-of-cycles} a morphism
$\Gamma^d(X/S)\times_S \Gamma^e(X/S)\to\Gamma^{d+e}(X/S)$ which on points is
addition of cycles. This induces a morphism $(X/S)^d\to\Gamma^d(X/S)$ which
has the topological properties of a quotient of $(X/S)^d$ by the symmetric
group.
\end{subsection}

\begin{subsection}{Relation with the symmetric product}
The addition of cycles morphism $(X/S)^d\to\Gamma^d(X/S)$ factors through
the quotient map $(X/S)^d\to\Sym^d(X/S)$ and it is easily proven that
$\Sym^d(X/S)\to\Gamma^d(X/S)$ is a universal homeomorphism with trivial
residue field extensions, cf.\ Corollary~\pref{C:Sym-Gamma}. It is further
easy to show that $\Sym^d(X/S)\to\Gamma^d(X/S)$ is an isomorphism over the
non-degeneracy locus, cf.\ Proposition~\pref{P:Sym-Gamma:nondeg}.
\end{subsection}

\begin{subsection}{Comparison of representability techniques}
Consider the following inclusions of categories:
$$\xygraph{ !{0;<1.8cm,0cm>:<0cm,2cm>::}
[d]
*++{\text{$X/S$ affine}}="AFF" :@{(->} [rr]
*++{\text{$X/S$ \AF{}-scheme}}="AF" [u]
*++{\parbox{3.6 cm}{$X/S$ quasi-projective\\of finite presentation}}="QP"
 :@{(->} "AF" "QP" :@{(->} [rrr]
*++{\parbox{5.5 cm}{$X/S$ separated algebraic space,\\locally of finite
presentation}} :@{(->} [d]
 *++{\parbox{5.5 cm}{$X/S$ separated algebraic space.}}="SAS"
"AF" :@{(->} "SAS"
}$$
When $X/S$ is affine, it is fairly easy to show the existence of the quotient
$\Sym^d(X/S)$~\cite[Ch.~V, \S2, No.~2, Thm.~2]{bourbaki_alg_comm_5_6}, the
representability of $\FGamma^d_{X/S}$ and the representability of the Hilbert
functor of points $\FHilb^d_{X/S}$~\cite{nori_Hilb_appendix,GLS_Affine_Hilb}.
The existence of $\Sym^d(X/S)$ and the representability of $\FGamma^d_{X/S}$
and $\FHilb^d_{X/S}$ in the category of \AF{}-schemes is then a simple
consequence.

When $X/S$ is (quasi-)projective and $S$ is noetherian, one can also show the
existence and (quasi-)projectivity of $\Sym^d(X/S)$, $\Gamma^d(X/S)$ and
$\Hilb^d(X/S)$ with projective methods, cf.~\cite{rydh_gammasymchow_inprep}
and~\cite[No.~221]{fga}. The representability of the Hilbert scheme in the
category of separated algebraic spaces locally of finite presentation can be
established using Artin's algebraization
theorem~\cite[Cor.~6.2]{artin_alg_formal_moduli_I}. We could likewise have used
Artin's algebraization theorem to prove the representability of
$\FGamma^d_{X/S}$ when $X/S$ is locally of finite presentation. The crucial
criterion, that $\FGamma^d_{X/S}$ is effectively pro-representable, is shown
in~\S\ref{SS:eff-pro-rep}.

Finally, the methods that we have used in this article to show that
$\FGamma^d_{X/S}$ is representable in the category of all separated algebraic
spaces can be applied, mutatis mutandis, to the Hilbert functor of points. The
proofs become significantly simpler as the difficulties encountered for
$\FGamma^d_{X/S}$ are almost trivial for the Hilbert functor. More generally,
these methods apply to the Hilbert stack of
points~\cite{rydh_hilbert}. The existence of $\Sym^d(X/S)$ can also be
proven in the same vein and this is done in~\cite{rydh_finite_quotients}.
\end{subsection}

\begin{subsection}{Notation and conventions}
We denote a \emph{closed} immersion of schemes or algebraic spaces with $X\inj
Y$. When $A$ and $B$ are rings or modules we use $A\inj B$ for an injective
homomorphism. We let $\N$ denote the set of non-negative integers
$0,1,2,\dots$ and use the notation $((a,b))=\binom{a+b}{a}$ for binomial
coefficients.
\end{subsection}

\end{section}
\setcounter{secnumdepth}{3}

\tableofcontents


\begin{section}{The algebra of divided powers}\label{S:Gamma}

We begin this section by briefly recalling the definition of polynomial laws
in~\S\ref{SS:pol_laws}, the algebra of divided powers $\Gamma_A(M)$
in~\S\ref{SS:divided_powers} and the multiplicative structure of
$\Gamma^d_A(B)$ in~\S\ref{SS:mul_div_powers}.

\begin{subsection}{Polynomial laws and symmetric tensors}\label{SS:pol_laws}

We recall the definition of a polynomial
law~\cite{roby_lois_pol, roby_lois_pol_mult}.

\begin{definition}\label{D:pol-law}
Let $M$ and $N$ be $A$-modules. We denote by $\fF_M$ the functor
$$\map{\fF_M}{A\xAlg}{\Sets},\quad\quad A'\mapsto M\otimes_A A'$$
A \emph{polynomial law} from $M$ to $N$ is a natural transformation
$\map{F}{\fF_M}{\fF_N}$. More concretely, a polynomial law is a set of
\emph{maps} $\map{F_{A'}}{{M\otimes_A A'}}{{N\otimes_A A'}}$ for every
$A$-algebra $A'$ such that for any homomorphism of $A$-algebras
${\map{g}{A'}{A''}}$ the diagram
$$\xymatrix{
{M\otimes_A A'}\ar[r]^{F_{A'}}\ar[d]_{\id{M}\otimes g} & 
   {N\otimes_A A'}\ar[d]^{\id{N}\otimes g} \\
{M\otimes_A A''}\ar[r]^{F_{A''}} & {N\otimes_A A''}\ar@{}[ul]|{\circ}
}$$
commutes.
The polynomial law $F$ is
\emph{homogeneous of degree} $d$ if for any $A$-algebra $A'$, the corresponding
map $\map{F_{A'}}{M\otimes_A A'}{N\otimes_A A'}$ is such that $F_{A'}(ax)=a^d
F_{A'}(x)$ for any $a\in A'$ and $x\in M\otimes_A A'$. If $B$ and $C$ are
$A$-algebras then a polynomial law from $B$ to $C$ is \emph{multiplicative} if
for any $A$-algebra $A'$, the corresponding map $\map{F_{A'}}{B\otimes_A
A'}{C\otimes_A A'}$ is such that $F_{A'}(1)=1$ and
$F_{A'}(xy)=F_{A'}(x)F_{A'}(y)$ for any $x,y\in B\otimes_A A'$.
\end{definition}

\begin{notation}
Let $A$ be a ring and $M$ and $N$ be $A$-modules (resp. $A$-algebras).
We let $\Hompl{d}(M,N)$ (resp. $\Hommpl{d}(M,N)$) denote the polynomial laws
(resp. multiplicative polynomial laws) $M\to N$ which are homogeneous of
degree $d$.
\end{notation}

\begin{notation}
Let $A$ be a ring and $M$ an $A$-algebra. We denote the $d$\textsuperscript{th}
tensor product of $M$ over $A$ by $\T^d_A(M)$. We have an action of the
symmetric group $\SG{d}$ on $\T^d_A(M)$ permuting the factors. The invariant
ring of this action is the symmetric tensors and is denoted $\TS^d_A(M)$. By
$\T_A(M)$ and $\TS_A(M)$ we denote the graded $A$-modules $\bigoplus_{d\geq 0}
\T^d_A(M)$ and $\bigoplus_{d\geq 0}\TS^d_A(M)$ respectively.
\end{notation}

\begin{xpar}\label{X:TS_and_filt_dir_lims}
The covariant functor $\TS^d_A(\cdot)$ commutes with filtered direct limits.
In fact, denoting the group ring of $\SG{d}$ by $\Z[\SG{d}]$ we have that
$$\TS^d_A(\cdot)=\T^d_A(\cdot)^{\SG{d}}=
\Hom_{\Z[\SG{d}]}\bigl(\Z,\T^d_A(\cdot)\bigr)$$
where $\SG{d}$ acts trivially on $\Z$. As tensor products, being left
adjoints, commute with any (small) direct limit so does $\T^d$.
Reasoning as in~\cite[Prop.~0.6.3.2]{egaI_NE} it follows that
$\Hom_{\Z[\SG{d}]}(\Z,\cdot)$ commutes with filtered direct limits. In fact,
$\Z$ is a $\Z[\SG{d}]$-module of finite presentation and that $\Z[\SG{d}]$ is
non-commutative is not a problem here.
\end{xpar}

\begin{xpar}[Shuffle product]\label{X:shuffle_product}
When $B$ is an $A$-algebra, then $\TS^d_A(B)$ has a natural $A$-algebra
structure induced from the $A$-algebra structure of $\T^d_A(B)$. The
multiplication on $\TS^d_A(B)$ will be written as juxtaposition. For any
$A$-module $M$, we can equip $\T_A(M)$ and $\TS_A(M)$ with $A$-algebra
structures. The multiplication on $\T_A(M)$ is the ordinary tensor product
and the multiplication on $\TS_A(M)$ is called the \emph{shuffle product} and
is denoted by $\times$. If $x\in\TS^d_A(M)$ and $y\in\TS^e_A(M)$ then
$$x\times y=\sum_{\sigma\in\SG{d,e}} \sigma\left(x\otimes_A y\right)$$
where $\SG{d,e}$ is the subset of $\SG{d+e}$ such that
$\sigma(1)<\sigma(2)<\dots<\sigma(d)$ and
$\sigma(d+1)<\sigma(d+2)<\dots\sigma(d+e)$.
\end{xpar}

\end{subsection}

\begin{subsection}{Divided powers}\label{SS:divided_powers}
Most of the material in this section can be found in~\cite{roby_lois_pol}
and~\cite{ferrand_norme}.

\begin{xpar}
Let $A$ be a ring and $M$ an $A$-module. Then there exists a graded
$A$-algebra, the algebra of divided powers, denoted
$\Gamma_A(M)=\bigoplus_{d\geq 0}\Gamma^d_A(M)$ equipped with maps
$\map{\gamma^d}{M}{\Gamma^d_A(M)}$ such that, denoting the multiplication with
$\times$ as in~\cite{ferrand_norme}, we have that for every $x,y\in M$,
$a\in A$ and $d,e\in\N$
\begin{align}
\Gamma^0_A(M) & = A,\quad\text{and}\quad \gamma^0(x)=1 \label{gamma0}\\
\Gamma^1_A(M) & = M,\quad\text{and}\quad \gamma^1(x)=x \label{gamma1}\\
\gamma^d(ax) & = a^d\gamma^d(x)\label{E:Gamma-scalars} \\
\gamma^d(x+y) & = \textstyle\sum_{d_1+d_2=d} \gamma^{d_1}(x)\times \gamma^{d_2}(y)\label{E:Gamma-addition} \\
\gamma^d(x)\times\gamma^e(x) & = ((d,e))\gamma^{d+e}(x)
\end{align}
Using~\eqref{gamma0} and~\eqref{gamma1} we will identify $A$ with
$\Gamma^0_A(M)$ and $M$ with $\Gamma^1_A(M)$.
If $(x_\alpha)_{\alpha\in\I}$ is a family of elements of $M$ and
$\nu\in\N^{(\I)}$ then we let
$$\gamma^\nu (x) = \bigtimes_{\alpha\in \I} \gamma^{\nu_\alpha} (x_\alpha)$$
which is an element of $\Gamma^d_A(M)$ with
$d=|\nu|=\sum_{\alpha\in\I} \nu_\alpha$.
\end{xpar}

\begin{xpar}[Functoriality]
$\Gamma_A(\cdot)$ is a covariant functor from the category of $A$-modules to
the category of graded $A$-algebras~\cite[Ch.~III \S4, p.~251]{roby_lois_pol}.
\end{xpar}

\begin{xpar}[Base change]\label{X:Gamma_base_change}
If $A'$ is an $A$-algebra then there is a natural isomorphism
$\Gamma_A(M)\otimes_A A'\to\Gamma_{A'}(M\otimes_A A')$ mapping
$\gamma^d(x)\otimes_A 1$ to ${\gamma^d(x\otimes_A 1)}$
\cite[Thm.~III.3, p.~262]{roby_lois_pol}.
This shows that $\gamma^d$ is a homogeneous polynomial law of degree $d$.
\end{xpar}

\begin{xpar}[Universal property]\label{X:univ_prop_of_Gamma}
The map $\Hom_A\bigl(\Gamma^d_A(M),N\bigr)\to\Hompl{d}(M,N)$ given by
$F\to F\circ\gamma^d$ is an isomorphism~%
\cite[Thm.~IV.1, p.~266]{roby_lois_pol}.
\end{xpar}

\begin{xpar}[Basis and generators]\label{X:Gamma_basis}
If $(x_\alpha)_{\alpha\in\I}$ is a set of generators of $M$, then
$\bigl(\gamma^\nu(x)\bigr)_{\nu\in\N^{(\I)}}$ is a set of generators of
$\Gamma_A(M)$ as an $A$-module. If $(x_\alpha)_{\alpha\in\I}$ is a basis of $M$
then $\bigl(\gamma^\nu(x)\bigr)_{\nu\in\N^{(\I)}}$ is a basis of
$\Gamma_A(M)$~\cite[Thm.~IV.2, p.~272]{roby_lois_pol}.  Furthermore, if $A$ is
an algebra over an infinite field or $A$ is an algebra over
$\Lambda_d=\Z[T]/P_d(T)$ where $P_d$ is the unitary polynomial
$P_d(T)=\prod_{0\leq i<j\leq d} (T^i-T^j)-1$, then $\gamma^d(M)$ generates
$\Gamma^d_A(M)$ \cite[Lemme 2.3.1]{ferrand_norme}. In particular, there is
always a finite \emph{faithfully flat} base change $A\to A'$ such that
$\Gamma^d_{A'}(M')$ is generated by $\gamma^d(M')$. More generally
$\gamma^d(M)$ generates $\Gamma^d_A(M)$ if and only if every residue field of
$A$ has at least $d$ elements~\cite{rydh_gammasymchow_inprep}.
\end{xpar}

\begin{xpar}[Exactness]\label{X:Gamma_exactness}
The functor $\Gamma_A(\cdot)$ is a left
adjoint~\cite[Thm.~III.1, p.~257]{roby_lois_pol} and thus commutes with any
(small) direct limit. It is thus
\emph{right exact}~\cite[Def.~2.4.1]{sga4_grothendieck_prefaisceaux} but note
that $\Gamma_A(\cdot)$ is a functor from $A\xMod$ to $A\xAlg$ and that the
latter category is not abelian.
By~\cite[Rem.~2.4.2]{sga4_grothendieck_prefaisceaux} a functor is right
exact if and only if it takes the initial object onto the initial object and
commutes with finite coproducts and coequalizers. Thus
$\Gamma_A(0)=A$ and given an exact diagram of $A$-modules
$$\xymatrix{
{M'} \ar@<.5ex>[r]^f \ar@<-.5ex>[r]_g & {M} \ar[r]^h & {M''}
}$$
the diagram
$$\xymatrix{
{\Gamma_A(M')}\ar@<.5ex>[r]^{\Gamma f} \ar@<-.5ex>[r]_{\Gamma g} &
{\Gamma_A(M)}\ar[r]^{\Gamma h} & {\Gamma_A(M'')}
}$$
is exact in the category of $A$-algebras and
$$\Gamma_A(M\oplus M')=\Gamma_A(M)\otimes_A \Gamma_A(M').$$
The latter identification can be made
explicit~\cite[Thm.~III.4, p.~262]{roby_lois_pol} as
\begin{equation}\label{E:Gamma^d_of_dir_sum}
\begin{split}
\Gamma^d_A(M\oplus M') &= \bigoplus_{a+b=d}
     \left(\Gamma^a_A(M)\otimes_A \Gamma^b_A(M')\right) \\
\gamma^d(x+y) &= \sum_{a+b=d}\gamma^a(x)\otimes\gamma^b(y).
\end{split}
\end{equation}
This makes $\Gamma_A(M\oplus M')=
\bigoplus_{a,b\geq 0}
\Gamma^{a,b}(M\oplus M')$ into a bigraded algebra where
$\Gamma^{a,b}(M\oplus M')=\Gamma^a_A(M)\otimes_A\Gamma^b_A(M')$.
\end{xpar}

\begin{xpar}[Surjectivity]\label{X:surj_of_Gamma}
If $M\surj N$ is a surjection then it is easily seen
from the explicit generators of $\Gamma(N)$ in~\pref{X:Gamma_basis} that
$\Gamma_A(M)\surj\Gamma_A(N)$ is surjective. This also follows from the
right-exactness of $\Gamma_A(\cdot)$ as any right-exact functor from modules
to rings takes surjections onto surjections, cf.~\pref{X:Gamma_pres}
\end{xpar}

\begin{xpar}[Presentation]\label{X:Gamma_pres}
Let $M=G/R$ be a presentation of the $A$-module $M$. Then
$\Gamma_A(M)=\Gamma_A(G)/I$ where $I$ is the ideal of $\Gamma_A(G)$ generated
by the images in $\Gamma_A(G)$ of $\gamma^d(x)$ for every $x\in R$ and
$d\geq 1$~\cite[Prop.~IV.8, p.~284]{roby_lois_pol}.
In fact, denoting the inclusion of $R$ in $G$ by $i$, we can write $M$ as a
coequalizer of $A$-modules
$$\xymatrix{
{R} \ar@<.5ex>[r]^i \ar@<-.5ex>[r]_0 & {G} \ar[r]^h & {M}
}$$
which by~\pref{X:Gamma_exactness} gives the exact sequence
$$\xymatrix{
{\Gamma_A(R)} \ar@<.5ex>[r]^{\Gamma(i)} \ar@<-.5ex>[r]_{\Gamma(0)}
& {\Gamma_A(G)} \ar[r]^{\Gamma(h)} & {\Gamma_A(M)}
}$$
of $A$-algebras. Since $\Gamma^0_A(0)=\Gamma^0_A(i)=\id{A}$ and
$\Gamma^d_A(0)=0$ for $d>0$ it follows that $\Gamma_A(M)$ is the quotient
of $\Gamma_A(G)$ by the ideal generated by
$\Gamma(i)\bigl(\bigoplus_{d\geq 1}\Gamma^d(R)\bigr)$.
\end{xpar}

\begin{xpar}[Exactness of $\Gamma^d_A(\cdot)$]\label{X:surj_of_Gamma^d}
If $M\surj N$ is a surjection then $\Gamma^d_A(M)\surj\Gamma^d_A(N)$ is
surjective since $\Gamma_A(M)\surj\Gamma_A(N)$ is surjective.
This does, however, not imply that $\Gamma^d_A(\cdot)$ is
right exact. In fact, in general it is not since we have that
$\Gamma^d_A(M\oplus M')\neq \Gamma^d_A(M)\oplus\Gamma^d_A(M')$.
\end{xpar}

\begin{xpar}[Presentation of $\Gamma^d_A(\cdot)$]\label{X:Gamma^d_pres}
If $M=G/R$ is a quotient of $A$-modules then $\Gamma^d_A(M)=\Gamma^d_A(G)/I$
where $I$ is the $A$-submodule generated by the elements $\gamma^k(x)\times y$
for $1\leq k\leq d$, $x\in R$ and $y\in\Gamma^{d-k}_A(G)$. This follows
immediately from~\pref{X:Gamma_pres}.
\end{xpar}

\begin{xpar}[Filtered direct limits]\label{X:Gamma^d_mod_filt_dir_lims}
The functor $\Gamma^d_A(\cdot)$ commutes with \emph{filtered} direct limits.
In fact, if $(M_\alpha)$ is a directed filtered system of $A$-modules then
\begin{align*}
\bigoplus_{d\geq 0} \Gamma^d_A
  (\varinjlim_{\!\!\scriptscriptstyle A\xMod\!\!} M_\alpha)
&= \varinjlim_{\scriptscriptstyle A\xAlg}
  \bigoplus_{d\geq 0}\Gamma^d_A(M_\alpha) = \\
&= \varinjlim_{\scriptscriptstyle A\xMod}
  \bigoplus_{d\geq 0}\Gamma^d_A(M_\alpha) =
\bigoplus_{d\geq 0}\varinjlim_{\scriptscriptstyle A\xMod}\Gamma^d_A(M_\alpha).
\end{align*}
The first equality follows from~\pref{X:Gamma_exactness} and the second
from the fact that a filtered direct limit in the category of $A$-algebras
coincides with the corresponding filtered direct limit in the category of
$A$-modules~\cite[Cor.~2.9]{sga4_grothendieck_prefaisceaux}.
\end{xpar}

\begin{xpar}\label{X:Gamma^d_flat/free}
If $M$ is a free (resp. flat) $A$-module then $\Gamma^d_A(M)$ is a free
(resp. flat) $A$-module. This follows from \pref{X:Gamma_basis} and
\pref{X:Gamma^d_mod_filt_dir_lims} as any flat module is a filtered direct
limit of free modules \cite[Thm.~1.2]{lazard_autour}.
\end{xpar}

\begin{xpar}[$\Gamma$ and $\TS$]\label{X:Gamma-TS-as-mod}
The homogeneous polynomial law $M\to\TS^d_A(M)$ of degree $d$ given by
$x\mapsto x^{\otimes_A d} = x\otimes_A \dots\otimes_A x$ corresponds by the
universal property~\pref{X:univ_prop_of_Gamma} to an $A$-module homomorphism
$\map{\varphi}{\Gamma^d_A(M)}{\TS^d_A(M)}$. This extends to an $A$-algebra
homomorphism $\Gamma_A(M)\to \TS_A(M)$, where the multiplication in $\TS_A(M)$
is the shuffle product~\pref{X:shuffle_product}, cf.~\cite[Prop.~III.1,
p.~254]{roby_lois_pol}.

When $M$ is a free $A$-module the homomorphisms $\Gamma^d_A(M)\to \TS^d_A(M)$
and $\Gamma_A(M)\to \TS_A(M)$ are isomorphisms of $A$-modules
respectively $A$-algebras~\cite[Prop.~IV.5, p.~272]{roby_lois_pol}.
The functors $\TS^d_A$ and $\Gamma^d_A$ commute with filtered direct limits
by~\pref{X:TS_and_filt_dir_lims} and~\pref{X:Gamma^d_mod_filt_dir_lims}.
Since any flat $A$-module is the filtered direct limit of free $A$-modules
\cite[Thm.~1.2]{lazard_autour}, it thus follows that $\Gamma_A(M)\to \TS_A(M)$
is an isomorphism of graded $A$-algebras for any flat $A$-module $M$.

Moreover by~\cite[Prop.~III.3, p.~256]{roby_lois_pol}, there are natural
$A$-module homomorphisms $\TS^d_A(M)\inj
\T^d_A(M)\surj\SYM^d_A(M)\to\Gamma^d_A(M)\to\TS^d_A(M)$ such that going around
one turn in the diagram
$$\xymatrix@C=5mm{
& \SYM^d_A(M)\ar[dl] & \\
\Gamma^d_A(M)\ar[rr] && \TS^d_A(M)\ar[ul]}$$
is multiplication by $d!$. Here $\SYM^d_A(M)$ denotes the degree $d$ part of
the symmetric algebra. Thus if $d!$ is invertible then
$\Gamma^d_A(M)\to\TS^d_A(M)$ is an isomorphism. In particular, this is the case
when $A$ is purely of characteristic zero, i.e., contains the field of
rationals.
\end{xpar}

\begin{xpar}[Universal multiplication of laws]
\label{X:univ-mult-of-laws}
Let $d,e\in \N$. There is a canonical homomorphism
$$\map{\rho_{d,e}}{\Gamma^{d+e}_A(M)}
{\Gamma^{d}_A(M)\otimes_A\Gamma^{e}_A(M)}$$
given by the homogeneous polynomial law
${x\mapsto\gamma^{d}(x)\otimes\gamma^{e}(x)}$
of degree $d+e$ and the universal property~\pref{X:univ_prop_of_Gamma}.
In particular
\begin{equation}\label{E:rho-formula}
\rho_{d,e}\bigl(\gamma^{\nu}(x)\bigr)=
\sum_{\substack{\nu'+\nu''=\nu\\|\nu'|=d,\;|\nu''|=e}}
\gamma^{\nu'}(x)\otimes \gamma^{\nu''}(x).
\end{equation}
We can factor $\rho_{d,e}$ as
$\pi_{d,e}\circ\Gamma^{d+e}(p)$ where $\map{p}{M}{M\oplus M}$ is the
diagonal map $x\mapsto x\oplus x$ and $\pi_{d,e}$ is the projection on the
factor of bidegree $(d,e)$ of $\Gamma^{d+e}(M\oplus M)$,
cf.\ Equation~\eqref{E:Gamma^d_of_dir_sum}.

If $\map{F_1}{M}{N_1}$ and $\map{F_2}{M}{N_2}$ are polynomial laws homogeneous
of degrees $d$ and $e$ respectively we can form the polynomial law
$\map{F_1\otimes F_2}{M}{N_1\otimes_A N_2}$ given by $(F_1\otimes
F_2)(x)=F_1(x)\otimes F_2(x)$. The law $F_1\otimes F_2$ is homogeneous of
degree $d+e$. If $\map{f_1}{\Gamma^{d}(M)}{N_1}$,
$\map{f_2}{\Gamma^{e}(M)}{N_2}$ and
$\map{f_{1,2}}{\Gamma^{d+e}(M)}{N_1\otimes_A N_2}$ are the corresponding
homomorphisms then $f_{1,2}=(f_1\otimes f_2)\circ \rho_{d,e}$.
\end{xpar}

\end{subsection}

\begin{subsection}{Multiplicative structure}\label{SS:mul_div_powers}
Let $M,N$ be $A$-modules and $d$ a positive integer. There is a unique
homomorphism
$$\map{\mu}{\Gamma^d_A(M)\otimes_A \Gamma^d_A(N)}{\Gamma^d(M\otimes_A N)}$$
sending $\mu(\gamma^d(x)\otimes\gamma^d(y))$ to
$\gamma^d(x\otimes y)$~\cite{roby_lois_pol_mult}.
When $B$ is an $A$-algebra, the composition of $\mu$ and the multiplication
homomorphism $B\otimes_A B\to B$ induces a multiplication on $\Gamma^d_A(B)$
which we will denote by juxtaposition. The multiplication is such that
$\gamma^d(x)\gamma^d(y)=\gamma^d(xy)$ and this makes $\gamma^d$ into a
multiplicative polynomial law homogeneous of degree $d$.
The unit in $\Gamma^d_A(B)$ is $\gamma^d(1)$.

If $B$ is an $A$-algebra and $M$ is a $B$-module, then $\mu$ together with the
module structure $B\otimes_A M\to M$ induces a $\Gamma^d_A(B)$-module structure
on $\Gamma^d_A(M)$.

\begin{xpar}[Universal property]\label{X:univ_prop_of_Gamma:mult}
Let $B$ and $C$ be $A$-algebras. Then the map
$\Hom_{A\xAlg}\bigl(\Gamma^d_A(B),C\bigr)\to\Hommpl{d}(B,C)$ given by
$F\to F\circ\gamma^d$ is an isomorphism~\cite{roby_lois_pol_mult}. Also
see~\cite[Prop.~2.5.1]{ferrand_norme}.
\end{xpar}

\begin{xpar}[$\Gamma$ and $\TS$]\label{X:Gamma-TS}
The homogeneous polynomial law $M\to\TS^d_A(M)$ of degree $d$ given by
$x\mapsto x^{\otimes_A d} = x\otimes_A \dots\otimes_A x$ is multiplicative.
The homomorphism $\map{\varphi}{\Gamma^d_A(B)}{\TS^d_A(B)}$
in~\pref{X:Gamma-TS-as-mod} is thus an $A$-algebra homomorphism. It is an
isomorphism when $B$ is a flat over $A$ or when $A$
is of pure characteristic zero~\pref{X:Gamma-TS-as-mod}. The morphism
$\Spec\bigl(\TS^d_A(B)\bigr)\to\Spec\bigl(\Gamma^d_A(B)\bigr)$ is a universal
homeomorphism with trivial residue field extensions, see
Corollary~\pref{C:Sym-Gamma}. Further results about this morphism is found in
\cite{rydh_gammasymchow_inprep}.
\end{xpar}

\begin{xpar}[Filtered direct limits]\label{X:Gamma^d_alg_filt_dir_lims}
The functor $B\mapsto \Gamma^d_A(B)$ commutes with filtered direct limits. This
follows from~\pref{X:Gamma^d_mod_filt_dir_lims} and the fact that a filtered
direct limit in the category of $A$-algebras coincides with the corresponding
filtered direct limit in the category of
$A$-modules~\cite[Cor.~2.9]{sga4_grothendieck_prefaisceaux}.
\end{xpar}

\begin{xpar}\label{X:Gamma^d_of_prod_of_rings}
The isomorphism of $A$-modules given by equation~\eqref{E:Gamma^d_of_dir_sum}
gives an isomorphism of $A$-algebras
\begin{align*}
\Gamma^d_A(B\times C) &= \prod_{a+b=d}
     \left(\Gamma^a_A(B)\otimes_A \Gamma^b_A(C)\right) \\
\gamma^d\bigl((x,y)\bigr) &= 
     \left(\gamma^a(x)\otimes\gamma^b(y)\right)_{a+b=d}.
\end{align*}
\end{xpar}

\begin{xpar}[Universal multiplication of laws]
\label{X:univ-mult-and-comp:alg-case}
Replacing $M$ with an algebra $B$ in~\pref{X:univ-mult-of-laws}, the polynomial
law defining the homomorphism $\rho_{d,e}$ is multiplicative. The homomorphism
$\rho_{d,e}$ is thus an $A$-algebra homomorphism. For a geometrical
interpretation of $\rho_{d,e}$ as ``addition of cycles'' see
section~\S\ref{SS:addition-of-cycles}.
\end{xpar}

\begin{formula}[Multiplication formula {\cite[Form. 2.4.2]{ferrand_norme}}]
\label{F:Gamma_mult_form}
Let $(x_\alpha)_{\alpha\in\I}$ be a set of elements in $B$ and let
$\mu,\nu\in \N^{(\I)}$ with $d=|\mu|=|\nu|$. Then we have the following
identity in $\Gamma^d_A(B)$
$$\gamma^\mu(x)\gamma^\nu(x)=
\sum_{\xi\in N_{\mu,\nu}} \gamma^\xi(x_{(1)}x_{(2)})=
\sum_{\xi\in N_{\mu,\nu}} \bigtimes_{(\alpha,\beta)\in\I\times\I}
       \gamma^{\xi_{\alpha,\beta}}(x_\alpha x_\beta)
$$
where $N_{\mu,\nu}$ is the set of multi-indices $\xi\in \N^{(\I\times\I)}$
such that $\sum_{\beta\in\I} \xi_{\alpha,\beta}=\mu_\alpha$ for every
$\alpha\in \I$ and $\sum_{\alpha\in\I} \xi_{\alpha,\beta}=\nu_\beta$ for
every $\beta\in\I$.
\end{formula}

\begin{proposition}\label{P:Gamma^d_finiteness}
If $B$ is an $A$-algebra of finite type (resp. of finite presentation,
resp. finite over $A$, resp. integral over $A$) then
$\Gamma^d_A(B)$ is an $A$-algebra of finite type (resp. of finite
presentation, resp. finite, resp. integral).
\end{proposition}
\begin{proof}
If $B$ is an $A$-algebra of finite type then $B$ is a quotient of a polynomial
ring $A[x_1,x_2,\dots,x_n]$. The induced homomorphism
$\Gamma^d(A[x_1,x_2,\dots,x_n])\to\Gamma^d_A(B)$ is surjective, and thus it is
enough to show that $\Gamma^d(A[x_1,x_2,\dots,x_n])$ is an $A$-algebra of
finite type.  As $\Gamma^d$ commutes with base change it is further enough to
show that
$\Gamma^d_{\Z}(\Z[x_1,x_2,\dots,x_n])=\TS^d_{\Z}(\Z[x_1,x_2,\dots,x_n])$ is a
$\Z$-algebra of finite type. This is well-known, cf.~\cite[Ch.~V, \S1, No.~9,
Thm.~2]{bourbaki_alg_comm_5_6}.

If $B$ is an $A$-algebra of finite presentation then there is a noetherian
ring $A_0$ and a $A_0$-algebra of finite type $B_0$ such that
$B=B_0\otimes_{A_0} A$. The first part of the proposition shows that
$\Gamma^d_{A_0}(B_0)$ is an $A_0$-algebra of finite type and thus also
of finite presentation as $A_0$ is noetherian. As $\Gamma^d$ commutes with base
change this shows that $\Gamma^d_A(B)$ is an $A$-algebra of finite
presentation.

If $B$ is a finite $A$-algebra then $\Gamma^d_A(B)$ is a finite
$A$-algebra by~\pref{X:Gamma_basis}. If $B$ is an integral $A$-algebra then
$B$ is a filtered direct limit of finite $A$-algebras. As $\Gamma^d$ commutes
with filtered direct limits this shows that $\Gamma^d_A(B)$ is an integral
$A$-algebra.
\end{proof}

\end{subsection}


\begin{subsection}{The scheme $\Gamma^d(X/S)$ for $X/S$ affine}
\label{SS:Repr-affine}

Let $S$ be any scheme and $\sA$ a quasi-coherent sheaf of $\sO_S$-algebras. As
the construction of $\Gamma^d_A(B)$ commutes with localization with respect to
multiplicatively closed subsets of $A$ we may define a quasi-coherent sheaf of
$\sO_S$-algebras $\Gamma^d_{\sO_S}(\sA)$. 
This extends the definition of the covariant functor $\Gamma^d$ to the category
of quasi-coherent algebras on $S$.
If $\map{f}{X}{S}$ is an affine morphism we let $\Gamma^d(X/S)=\Spec\bigl(
\Gamma^d_{\sO_S}(f_*\sO_X)\bigr)$.
This defines a covariant functor
$$\map{\Gamma^d}{\AffSch_{/S}}{\AffSch_{/S}},\quad\quad X/S\mapsto \Gamma^d(X/S)$$
where $\AffSch_{/S}$ is the category of schemes affine over $S$. When it is not
likely to cause confusion, we will sometimes abbreviate $\Gamma^d(X/S)$ with
$\Gamma^d(X)$.

A polynomial law in this setting is a natural transformation of functors
from quasi-coherent $\sO_S$-algebras to sheaves of sets on $S$. We obtain
an isomorphism $\Hom_S\bigl(S',\Gamma^d(X/S)\bigr)\to
\Hommpl{d}_{,\sO_S}(\sO_X,\sO_{S'})$ for any affine $S$-scheme $S'$. Also
observe that
\begin{align*}
\Hom_S\bigl(S',\Gamma^d(X/S)\bigr) &\iso
\Hom_{S'}\bigl(S',\Gamma^d(X/S)\times_S S'\bigr) \\
&\iso \Hom_{S'}\bigl(S',\Gamma^d(X'/S')\bigr).
\end{align*}
More generally, if $S$ is an algebraic space and $X\to S$ is affine we define
$\Gamma^d(X/S)$ by \etale{} descent.

Defining $\Gamma^d(X/S)$ for any $S$-scheme $X$ is non-trivial. In the
following sections we will give a functorial description of $\Gamma^d(X/S)$ and
then show that this functor is represented by a scheme or algebraic space
$\Gamma^d(X/S)$.

A very useful fact that will repeatedly be used in the sequel is the following
rephrasing of paragraph~\pref{X:Gamma^d_of_prod_of_rings}:

\begin{proposition}\label{P:Gamma^d_of_disj_union}
Let $S$ be an algebraic space and let $X_1,X_2,\dots,X_n$ be
algebraic spaces affine over $S$. Then
$$\Gamma^d\left(\coprod_{i=1}^n X_i\right)=
\coprod_{\substack{d_i\in\N\\\sum_i d_i=d}}
\Gamma^{d_1}(X_1)\times_S \Gamma^{d_2}(X_2)\times_S\dots\times_S \Gamma^{d_n}(X_n).$$
\end{proposition}

Similarly, the following Proposition is a translation of
paragraph~\pref{X:surj_of_Gamma^d}:

\begin{proposition}\label{P:affine-closed-immersion}
If $Y$ is an algebraic space affine over $S$ and $X\inj Y$ a closed subspace,
then $\Gamma^d(X/S)$ is a closed subspace of $\Gamma^d(Y/S)$.
\end{proposition}

\end{subsection}

\end{section}


\begin{section}{Support and image of a family of zero cycles}\label{S:support}

Let $X/S$ be a scheme or an algebraic space, affine over $S$. In this section
we will show that a ``family of zero cycles'' $\alpha$ on $X$ parameterized by
$S$, that is, a morphism $\map{\alpha}{S}{\Gamma^d(X/S)}$, has a unique minimal
closed subspace $Z=\FCI(\alpha)\inj X$, the \emph{image of $\alpha$}, such that
$\alpha$ factors through the closed subspace $\Gamma^d(Z/S)\inj\Gamma^d(X/S)$.
The reduction $Z_\red$ will be denoted the
\emph{support of $\alpha$} and written as $\FCS(\alpha)$.

For general $X/S$ a family of zero cycles $\alpha$, parameterized by a
$S$-scheme $T$, should be thought of as one of the following
\begin{enumerate}
\item A morphism $T\to \Gamma^d(X/S)$.
\item An ``object'' living over $\FCI(\alpha)\inj X\times_S T$.
\item A ``multi-section'' $T\to X\times_S T$ with image $\FCI(\alpha)$.
\end{enumerate}
Note that in contrast to ordinary sections and families of closed subschemes, a
family of zero cycles is \emph{not} uniquely determined by its image. If
$\alpha$ is a family over a reduced scheme $T$, then
$\FCS(\alpha)=\FCI(\alpha)$ is reduced,
cf.~Proposition~\pref{P:supp-over-red}. In this case, the ``object'' in (ii)
can be interpreted as a cycle in the ordinary sense.

We will show the following results about the image and the support:
\begin{enumerate}
\item The image is integral over $S$. (\S\ref{SS:kernel-of-mult-law})
\item The image commutes with essentially smooth base change $S'\to S$ and
projective limits. In particular it commutes with \etale{} base change and
henselization. (\S\ref{SS:kernel-and-base-change})
\item The support commutes with any base change.
(\S\ref{SS:image-and-base-change})
\item The support has \emph{universally topologically finite fibers}, i.e., each
fiber over $S$ consists of a finite number of points and the separable
degrees of the corresponding field extensions are finite.
(\S\ref{SS:various-props-of-image-and-support})
\item The support is universally open over $S$.
(\S\ref{SS:top-props-of-support})
\end{enumerate}
Many of the results require rather technical but standard demonstrations. In
particular we will often need to reduce from the integral to the finite case by
the standard limit techniques of~\cite[\S8]{egaIV}. The fact that the
support is universally open over $S$ will not be needed in the following
sections but this result, as well as the fact that the support has universally
topologically finite fibers, shows that topologically the support behaves as if
it was of finite presentation over $S$.

\begin{subsection}{Kernel of a multiplicative law}\label{SS:kernel-of-mult-law}
We will first define the \emph{kernel} of a multiplicative polynomial law
$\map{F}{B}{C}$ of $A$-algebras. If $F$ is of degree $1$, i.e., a ring
homomorphism, then the kernel is the usual kernel. In general, the kernel of
$F$ is the largest ideal $I$ such that $F$ factors through $B\surj B/I$. We
will focus our attention on the case when $C=A$. Then $B/\ker(F)$ is integral
over $A$ as shown in Proposition~\pref{P:supp-integrality} and there is a
canonical filtration of $\ker(F)$ which degenerates in characteristic zero.

\begin{definition}
Let $B$ and $C$ be $A$-algebras. Given a multiplicative law
$\map{F}{B}{C}$ we define its \emph{kernel} $\ker(F)$ as the largest ideal
$I$ such that $F$ factors as $B\surj B/I\to C$. This is a well-defined ideal
since if $F$ factors through $B\surj B/I$ and $B\surj B/J$ then $F$ factors
through $B/(I+J)$.
\end{definition}

Note that $F$ factors through $B\surj B/I$ if and only if
$F_{A'}(b'+IB')=F_{A'}(b')$ for any $A$-algebra $A'$ and $b'\in B'=B\otimes_A
A'$. Also note that the kernel $\ker(F_{A'})$ contains $\ker(F)B'$ but this
inclusion is often strict.

\begin{notation}
We will in the following denote homogeneous laws by upper-case Latin letters
and the corresponding homomorphisms by lower-case letters. For example, if
$\map{F}{B}{C}$ is a homogeneous multiplicative polynomial law of degree $d$ we
let $\map{f}{\Gamma^d_A(B)}{C}$ be the corresponding homomorphism. If
$A'$ is an $A$-algebra we denote by $\map{F'}{B'}{C'}$ the multiplicative law
given by $F'_R=F_R$ for every $A'$-algebra $R$. The corresponding homomorphism
$\map{f'}{\Gamma^d_{A'}(B')}{C'}$ is then the base change of $f$ along
$A\to A'$.
\end{notation}

\begin{lemma}\label{L:elementwise_criteria_for_ker}
Let $A$ be a ring and let $B$ and $C$ be $A$-algebras. Given a multiplicative
law $\map{F}{B}{C}$ homogeneous of degree $d$, or equivalently given a morphism
$\map{f}{\Gamma^d_A(B)}{C}$, define the following subsets of $B$
\begin{align*}
L_1 &= \left\{b\in B\;:\; f\bigl(\gamma^k(b)\times y\bigr)=0,
\;\forall k,y\right\}\\
L_2 &= \left\{b\in B\;:\; f\bigl(\gamma^k(bx)\times y\bigr)=0,
\;\forall k,x,y\right\}\\
L_3 &= \left\{b\in B\;:\; f'\bigl(\gamma^k(bx')\times y'\bigr)=0,
\;\forall k,A',x',y'\right\}
\end{align*}
where $1\leq k\leq d$, $x\in B$, $y\in\Gamma^{d-k}_A(B)$,
$x'\in B'$, $y'\in\Gamma^{d-k}_{A'}(B')$ and $A\to A'$ is a ring
homomorphism. Then $\ker(F)=L_1=L_2=L_3$. In particular, these sets are
ideals.
\end{lemma}
\begin{proof}
Clearly $L_3\subseteq L_2 \subseteq L_1$. Let $b\in L_1$ and let $x\in B$. The
multiplication formula~\pref{F:Gamma_mult_form} shows that for any
$y\in\Gamma^{d-k}_A(B)$
$$\gamma^k(bx)\times y=
\bigl(\gamma^k(b)\times y\bigr)\bigl(\gamma^k(x)\times \gamma^{d-k}(1)\bigr)
+\sum_{i=1}^{k} \gamma^i(b)\times y_i$$
for some $y_i\in\Gamma^{d-i}_A(B)$. Thus $b\in L_2$ and hence $L_1=L_2$. From
Equations~\eqref{E:Gamma-scalars} and~\eqref{E:Gamma-addition} it follows that
$L_2=L_3$ and that this set is an ideal.

If $I$ is an ideal in $B$ then $\Gamma^d_A(B/I)=\Gamma^d_A(B)/J$ where $J$ is
the ideal generated by $\gamma^k(b)\times y$ where $b\in I$, $1\leq k\leq d$
and $y\in \Gamma^{d-k}_A(B)$, cf.~\pref{X:Gamma^d_pres}. Thus $\ker(F)$ is
contained in $L_2$. On the other hand, if $b$ is contained in $L_3$ then for
any $A$-algebra $A'$ and $b',x'\in B'=B\otimes_A A'$ we have that
$$F_{A'}(b'+bx')=\sum_{k=0}^d
f'\bigl(\gamma^k(bx')\times\gamma^{d-k}(b')\bigr)=
f'\bigl(\gamma^d(b')\bigr)=F_{A'}(b')$$
and thus $b\in\ker(F)$.
\end{proof}

\begin{proposition}[{\cite[Lem.~7.6]{ziplies_circ_comp_radicals}}]
\label{P:supp-over-red}
Let $A$ be a ring and $B,C$ be $A$-algebras together with a multiplicative
law $\map{F}{B}{C}$ homogeneous of degree $d$. If $C$ is reduced then
$B/\ker(F)$ is reduced.
\end{proposition}
\begin{proof}
Let $\map{f}{\Gamma^d_A(B)}{C}$ be the homomorphism corresponding to $F$.  Let
$b\in B$ such that $b^n\in\ker(F)$ for some $n\in\N$. Then by
Lemma~\pref{L:elementwise_criteria_for_ker} we have that $f\bigl(\gamma^k(b^n
x)\times y\bigr)=0$ for every $1\leq k\leq d$, $x\in B$ and
$y\in\Gamma^{d-k}_A(B)$. An easy calculation using the multiplication
formula~\pref{F:Gamma_mult_form} shows that the element
$\bigl(\gamma^k(b)\times y\bigr)^{\ceil{dn/k}}$ is in the kernel of $f$ for
every $1\leq k\leq d$ and $y\in\Gamma^{d-k}_A(B)$. As $C$ is reduced this
implies that $\gamma^k(b)\times y$ is in the kernel of $f$ and thus
$b\in\ker(F)$.
\end{proof}

\begin{definition}\label{D:kernel}
Let $\map{F}{B}{A}$ be a multiplicative law homogeneous of degree $d$. For any
$b\in B$ we define its characteristic polynomial as
$$\chi_{F,b}(t)=F_{A[t]}(b-t)=\sum_{k=0}^d
(-1)^k f\bigl(\gamma^{d-k}(b)\times\gamma^{k}(1)\bigr)t^k\in A[t].$$
We let
$$I_\CH(F)=\bigl(\chi_{F,b}(b)\bigr)_{b\in B}\subseteq B$$
be the \emph{Cayley-Hamilton ideal} of $F$. Here $\chi_{F,b}(b)$ is the
evaluation of $\chi_{F,b}(t)$ at $b\in B$, i.e., the image of $\chi_{F,b}(t)$
along $A[t]\to B[t]\to B[t]/(t-b)=B$.
\end{definition}

\begin{proposition}[{\cite[Satz 4]{bergmann_formen_auf_moduln}}]
\label{P:CH-supp}\label{P:supp-integrality}
Let $\map{F}{B}{A}$ be a multiplicative law. Then
$I_\CH(F)\subseteq \ker(F)\subseteq \sqrt{I_\CH(F)}$. In particular
it follows that $B/\ker(F)$ is integral over $A$.
\end{proposition}
\begin{proof}
Let $P\surj B$ be a surjection from a flat $A$-algebra $P$ and let
$\map{F'}{P}{A}$ be the multiplicative law given as the composition of $F$ with
$P\surj B$. As the images of $I_\CH(F')$ and $\ker(F')$ in $B$ are $I_\CH(F)$
and $\ker(F)$ respectively, we can, replacing $B$ with $P$ and $F$ with $F'$,
assume that $B$ is flat over $A$. Then $\Gamma^d_A(B)=\TS^d_A(B)$.

We will first show the inclusion $I_\CH(F)\subseteq \ker(F)$. By definition
this is equivalent with the following: For every base change $A\to A'$, every
$b\in B$ and every $b',x'\in B'=B\otimes_A A'$, the identity
$F_{A'}\bigl(\chi_{F,b}(b)x'+b'\bigr)=F_{A'}(b')$ holds.

\newcommand\Diag{\mathrm{Diag}}
\newcommand\diag{\mathrm{diag}}

For any ring $R$ we let $\Diag_d(R)=R^d$ denote the diagonal
$d\times d$-matrices with coefficients in $R$. Let
$\map{\Psi}{B}{\Diag_d\bigl(\T^d_A(B)\bigr)}$ be the ring homomorphism such
that
$\Psi(b)=\diag\left(b_1,b_2,\dots,b_d\right)$ where
$b_k=1^{\otimes k-1}\otimes b\otimes 1^{\otimes d-k}\in \T^d_A(B)$.
The determinant gives a multiplicative law
$$\map{\det}{\Diag_d\bigl(\T^d_A(B)\bigr)}{\T^d_A(B)}$$
which is homogeneous of degree $d$. Let $E=\TS^d_A(A[t])=A[e_1,e_2,\dots,e_d]$
be the polynomial ring over $A$ in $d$ variables. Here $e_k$ denotes the
elementary symmetric function $t^{\otimes k}\times 1^{\otimes d-k}$.  Let $b\in
B$ be any element. We have a homomorphism $\injmap{\rho_b}{E}{\TS^d_A(B)}$
induced by the morphism $A[t]\to B$ mapping $t$ on $b$. More explicitly
$\rho_b(e_k)=b^{\otimes k}\times 1^{\otimes d-k}$.

Let $A\to A'$ be any ring homomorphism and let $B'=B\otimes_A A'$,
$E'=E\otimes_A A'$. We have a commutative diagram
$$\xymatrix{
{B'}\ar[r]^-{\gamma^d} & {\TS^d_{A'}(B')}\ar[r]^{f'} & {A'} \\
{B'\otimes_{A'} E'}\ar[r]^-{\gamma^d}\ar@{->>}[u]_{(\id{}, f'\circ\rho'_b)}
    \ar[d]^{\Psi}
  & {\TS^d_{A'}(B')\otimes_{A'} E'}\ar@{->>}[u]_{(\id{}, f'\circ\rho'_b)}
    \ar@{->>}[r]^-{(\id{},\rho'_b)}\ar@{(->}[d]\ar@{}[ul]|\circ
  & {\TS^d_{A'}(B')}\ar[u]^{f'}\ar@{(->}[dd]\ar@{}[ul]|\circ \\
{\Diag_d\bigl(\T^d_{A'}(B')\otimes_{A'} E'\bigr)}\ar[r]^-{\det}
    \ar@{->>}[d]^{\Diag(\id{},\rho'_b)}
  & {\T^d_{A'}(B')\otimes_{A'} E'}\ar@{->>}[dr]^{(\id{},\rho'_b)}
    \ar@{}[ul]|\circ
  & {}\ar@{}[ul]|\circ \\
{\Diag_d\bigl(\T^d_{A'}(B')\bigr)}\ar[rr]^-{\det}
  & & {\T^d_{A'}(B').}\ar@{}[ull]|\circ \\
}$$
Let $\chi(t)=\sum_{k=0}^d(-1)^k e_{d-k} t^k\in E[t]$ where we let
$e_0=1$. Let
$$\chi_b(t)=\rho_b\circ \chi(t)=
\sum_{k=0}^d (-1)^k \gamma^{d-k}(b)\times\gamma^{k}(1) t^k\in \TS^d_A(B)[t].$$
Then $f(\chi_b(t))=\chi_{F,b}(t)\in B[t]$. Let $b',x'\in B'$ be any
elements. We begin with the elements $\chi_{F,b}(b)x'+b'$ and $b'$ in the
upper-left corner $B'$ of the diagram and want to show that their images by
$F_{A'}=f'\circ\gamma^d$ in the upper-right corner $A'$ coincide. As
$\chi_{F,b}(b)x'+b'$ lifts to $\chi(b)x'+b'\in B'\otimes_{A'} E'$ it is enough
to show that images of $b',\chi(b)x'+b'\in B'\otimes_{A'} E'$ in the lower-left
corner $\Diag_d\bigl(\T^d_{A'}(B')\bigr)$ are equal.
%
%

For any ring $R$ and diagonal matrix $D\in\Diag_d(R)$ let $P_D(t)\in R[t]$ be
the characteristic polynomial of $D$. Then by Cayley-Hamilton's theorem
$P_D(D)=0$ in $\Diag_d(R)$. Note that the determinant and the characteristic
polynomial commute with arbitrary base change $R\to R'$. Now, the image of
$\chi(b)$ by $\Diag(\id{},\rho_b)\circ\Psi$ is easily seen to be
$\chi_b(\Psi(b))=P_{\Psi(b)}\bigl(\Psi(b)\bigr)=0$. Thus the images of
$\chi(b)x'+b'$ and $b'$ in the lower-left corner are equal. This concludes the
proof of the inclusion $I_\CH(F)\subseteq\ker(F)$.

If $b\in\ker(F)$ then by Lemma~\pref{L:elementwise_criteria_for_ker}
$f\bigl(\gamma^k(b)\times\gamma^{d-k}(1)\bigr)=0$ for every $k=1,2,\dots,d$.
Thus $\chi_{F,b}(t)=t^d$ and hence $b^d\in I_\CH(F)$ which shows the second
inclusion. Finally $B/I_\CH(F)$ is clearly integral over $A$ and thus also
$B/\ker(F)$.
\end{proof}

\begin{remark}
Ziplies defines the \emph{radical} of a not necessarily homogeneous
polynomial law in \cite[Def.~6.7]{ziplies_circ_comp_radicals}. When the
polynomial law is homogeneous the radical coincides with the kernel as
defined in~\pref{D:kernel}. Ziplies further proves in
\cite[Lem.~7.4]{ziplies_circ_comp_radicals} that if $I_\CH(F)$ is zero in
$B$ then $\ker(F)$ is contained in the Jacobson radical of
$B$. Proposition~\pref{P:CH-supp} shows more generally that under this
assumption $\ker(F)$ is contained in the nilradical of $B$. Note that both
inclusions $I_\CH(F)\subseteq\ker(F)\subseteq\sqrt{I_\CH(F)}$ can be
strict\footnote{There is a misprint in
\cite[Lem.~7.4]{ziplies_circ_comp_radicals}. ``equals'' should be replaced
with ``is contained in''. Also $A$ should be a $B$-algebra as well as an
$R$-algebra in his notation.}.

In \cite[3.4]{ziplies_gamma_hat} Ziplies also shows that $I_\CH(F)$ is
contained in the ideal
\begin{align*}
I_F^{(1)} &=
  \bigl\{b\in B\;:\; f\bigl(bx\times\gamma^{d-1}(1)\bigr)=0,
                \;\forall x\in B\bigr\}\\
 &= \bigl\{b\in B\;:\; f(b\times y)=0,\;\forall y\in\Gamma^{d-1}_A(B)\bigr\}.
\end{align*}
As this ideal by Lemma~\pref{L:elementwise_criteria_for_ker} clearly contains
$\ker(F)$, the first inclusion of Proposition~\pref{P:CH-supp} is a
generalization of this result.
\end{remark}

\end{subsection}


\begin{subsection}{Kernel and base change}\label{SS:kernel-and-base-change}

\begin{definition}
Let $A$ be a ring and let $B$ and $C$ be $A$-algebras. Given a multiplicative
law $\map{F}{B}{C}$ homogeneous of degree $d$, or equivalently given a morphism
$\map{f}{\Gamma^d_A(B)}{C}$, we let
$$I_F^{(k)}=\left\{b\in B\;:\; f\bigl(\gamma^i(b)\times y\bigr)=0,
\;\forall 1\leq i\leq k,\; y\in\Gamma^{d-i}_A(B)\right\}.$$
for $k=0,1,2,\dots,d$.
\end{definition}

\begin{proposition}
Let $B$ and $C$ be $A$-algebras and let $\map{F}{B}{C}$ be a multiplicative law
homogeneous of degree $d$. Then the sets $I_F^{(k)}$ are ideals of $B$ and we
have a filtration
$$B=I_F^{(0)}\supseteq I_F^{(1)}\supseteq\dots\supseteq I_F^{(d)}=\ker(F).$$
If $A'$ is an $A$-algebra and $B'=B\otimes_A A'$ then
$I_{F_{A'}}^{(k)}\supseteq I_F^{(k)}B'$. In particular $\ker(F_{A'})\supseteq
\ker(F)B'$.
\end{proposition}
\begin{proof}
That $I_F^{(k)}$ are ideals follows exactly as in the proof of
Lemma~\pref{L:elementwise_criteria_for_ker}. That $I_F^{(d)}=\ker(F)$ is
Lemma~\pref{L:elementwise_criteria_for_ker} and the other assertions are
trivial.
\end{proof}

The main application for the filtration
$I_F^{(0)}\supseteq I_F^{(1)}\supseteq\dots\supseteq I_F^{(d)}$ is that
the elements in $I_F^{(k-1)}$ behave ``quasi-linear'' modulo $I_F^{(k)}$
with respect to $\gamma^k$ in a certain sense. This will be utilized in
Lemma~\pref{L:image-and-pol-extensions}.

\begin{lemma}\label{L:binom}
Let $n\in \N$ and $p$ be a prime. Then $p\divides\binom{n}{k}$ for every
$1\leq k\leq n-1$ if and only if $n=p^s$.
\end{lemma}
\begin{proof}
Assume that $p\divides\binom{n}{k}$ for $1\leq k\leq n-1$. It easily follows
that $a^n=a$ in $\F_p$ for every $a\in\F_p$. Thus $x^p-x$ divides $x^n-x$ in
$\F_p[x]$ which shows that $p\divides n$. We obtain that $a^{n/p}=a$ for
every $a\in\F_p$ and by induction on $s$ that $n=p^s$. The converse is easy.
\end{proof}

\begin{proposition}\label{P:supp-filtration-stability}
Let $A$ be either a $\Z_{(p)}$-algebra with $p$ a prime or a $\Q$-algebra
in which case we let $p=1$. Then $I_F^{(k)}=I_F^{(k-1)}$ if $k\geq 1$ and
$k\neq p^s$. In particular, if $A$ is a $\Q$-algebra then $\ker(F)=I_F^{(1)}$.
\end{proposition}
\begin{proof}
Let $A'=A[t]$ and $b'_1,b'_2\in I_{F'}^{(k-1)}$. Then for any
$y'\in\Gamma^{d-k}_{A'}(B')$
$$f'\bigl(\gamma^k(b'_1+b'_2)\times y'\bigr)=
f'\bigl(\gamma^k(b'_1)\times y'\bigr)+f'\bigl(\gamma^k(b'_2)\times y'\bigr).$$
In particular for any $b\in I_{F}^{(k-1)}$ and $y\in\Gamma^{d-k}_A(B)$
$$(1+t)^k f'\bigl(\gamma^k(b)\times y\bigr)=
f'\bigl(\gamma^k((1+t)b)\times y\bigr)=
(1+t^k) f\bigl(\gamma^k(b)\times y\bigr)$$
which shows that $\binom{k}{i}$ annihilates $f\bigl(\gamma^k(b)\times y\bigr)$
for any $1\leq i\leq k-1$. By Lemma~\pref{L:binom}, it follows that if $k\neq
p^s$ then $f\bigl(\gamma^k(b)\times y\bigr)=0$ and thus $b\in I_F^{(k)}$.
\end{proof}

\begin{lemma}\label{L:image-and-inv-lim-of-carrier-scheme}
Let $A$ be a ring and $B=\varinjlim B_\lambda$ be a filtered direct limit of
$A$-algebras with induced homomorphisms $\map{\varphi_\lambda}{B_\lambda}{B}$.
Let $\map{f}{\Gamma^d_A(B)}{C}$ and denote by $f_\lambda$ the composition of
$\map{\Gamma^d_A(\varphi_\lambda)}{\Gamma^d_A(B_\lambda)}{\Gamma^d_A(B)}$ and
$f$. Then $I_F^{(k)}=\varinjlim I_{F_\lambda}^{(k)}$ for every
$k=0,1,\dots,d$. In particular $\ker(F)=\varinjlim \ker(F_\lambda)$.
\end{lemma}
\begin{proof}
As $f_\lambda$ factors as $\Gamma^d_A(B_\lambda)\to\Gamma^d_A(B)\to C$ it
follows that $\varphi_\lambda^{-1}\bigl(I_F^{(k)}\bigr)\subseteq
I_{F_\lambda}^{(k)}$. Thus $I_F^{(k)} \subseteq \varinjlim
I_{F_\lambda}^{(k)}$. Conversely, for any $b\in B\mysetminus I_F^{(k)}$ there
is an $i\leq k$ and $y\in \Gamma^{d-i}_A(B)$ such that
$f\bigl(\gamma^i(b)\times y\bigr)\neq 0$. If we let $\alpha$ be such that
${\varphi_\alpha}^{-1}(b)\neq \emptyset$ and
${\Gamma^{d-i}(\varphi_\alpha)}^{-1}(y)\neq \emptyset$ then for any
$\lambda\geq\alpha$ and $b_\lambda\in B_\lambda$ such that
$\varphi_\lambda(b_\lambda)=b$ we have that $b_\lambda\notin
I_{F_\lambda}^{(k)}$. Thus $\varinjlim I_{F_\lambda}^{(k)}\subseteq I_F^{(k)}$.
\end{proof}

\begin{proposition}\label{P:image-and-localization}
Let $A$ be a ring and $S$ a multiplicative closed subset. Let $\map{F}{B}{A}$
be a multiplicative homogeneous law of degree $d$ and denote by
$\map{S^{-1}F}{S^{-1}B}{S^{-1}A}$ the map corresponding to the $A$-algebra
$S^{-1}A$. Then $S^{-1}I_F^{(k)}=I_{S^{-1}F}^{(k)}$. In particular
$S^{-1}\ker(F)=\ker(S^{-1}F)$, i.e., the kernel commutes with localization.
\end{proposition}
\begin{proof}
By Proposition~\pref{P:supp-integrality} the quotient $B/\ker(F)$ is integral
over $A$. Replacing $B$ by $B/\ker(F)$ we can thus assume that $B$ is integral
over $A$. As $B$ is the filtered direct limit of its finite sub-$A$-algebras
and both the kernel of a multiplicative law,
Lemma~\pref{L:image-and-inv-lim-of-carrier-scheme}, and tensor products commute
with filtered direct limits we can assume that $B$ is a finite $A$-algebra.
Then $\Gamma^i_A(B)$ is a finite $A$-algebra for all $i=0,1,\dots,d$
by Proposition~\pref{P:Gamma^d_finiteness}.

Let $x/s\in I_{S^{-1}F}^{(k)}$, i.e., by definition $x/s\in S^{-1}B$ such that
$S^{-1}f\bigl(\gamma^i(x/s)\times y\bigr)=0$ for all $1\leq i\leq k$ and
$y\in\Gamma^{d-i}_A(B)$. For any $y\in\Gamma^{d-i}_A(B)$ there is then a
$t\in S$ such that $t f\bigl(\gamma^i(x)\times y\bigr)=0$ in $A$.
As $\Gamma^{d-i}_A(B)$ is a finite $A$-algebra we can find a common $t$ that
works for all $i\leq k$ and $y$. Then $f\bigl(\gamma^i(tx)\times y\bigr)=
t^i f\bigl(\gamma^i(x)\times y\bigr)=0$ for all $i\leq k$ and $y$. As
$x/s=tx/st$, this shows that $I_{S^{-1}F}^{(k)}=S^{-1}I_F^{(k)}$.
\end{proof}

\begin{proposition}\label{P:image-and-limits}
Let $A$ be a ring and $B$ an $A$-algebra. Let $A'=\varinjlim A'_\lambda$ be a
filtered direct limit of $A$-algebras with induced homomorphisms
$\map{\varphi_\lambda}{A'_\lambda}{A'}$. Let $\map{F}{B}{A}$ be a
multiplicative polynomial law of degree $d$. Then $I_{F_{A'}}^{(k)}=
\varinjlim I_{F_{A'_\lambda}}^{(k)}$ for every
$k=0,1,\dots,d$. In particular $\ker(F_{A'})=\varinjlim \ker(F_{A'_\lambda})$.
\end{proposition}
\begin{proof}
As in the proof of Proposition~\pref{P:image-and-localization} we can assume
that $B$ is finite over $A$ and hence that $\Gamma^i_A(B)$ is a finite
$A$-module. Choose generators $y_{i1},y_{i2},\dots,y_{in_i}$ of
$\Gamma^{d-i}_A(B)$ as an $A$-module for $i=1,2,\dots,d$. Let $B'=B\otimes_A
A'$ and $B'_\lambda=B\otimes_A A'_\lambda$.
Let $b'\in I_{F_{A'}}^{(k)}$. Then there exists an $\alpha$ and $b'_\alpha\in
B'_\alpha$ such that $b'$ is the image of $b'_\alpha$ by $B'_\alpha\to B$. As
the image of $f_{A'_\alpha}(\gamma^i(b'_\alpha)\times y_{ij})$ in $A'$ is
$f_{A'}(\gamma^i(b')\times y_{ij})$ and hence zero for $i=1,2,\dots,k$, there
is a $\beta\geq\alpha$ such that $b'_\alpha \in I_{F_{A'_\lambda}}^{(k)}$ for
all $\lambda\geq\beta$. Thus $b'\in \varinjlim_\lambda
I_{F_{A'_\lambda}}^{(k)}$ and $I_{F_{A'}}^{(k)}\subseteq \varinjlim_{\lambda}
I_{F_{A'_\lambda}}^{(k)}$. The reverse inclusion is obvious.
\end{proof}

We will now show that the kernel, always commutes with smooth base change and
that it commutes with flat base change in characteristic zero.

\begin{proposition} 
Let $A$ be a ring and let $\map{F}{B}{A}$ a multiplicative homogeneous law of
degree $d$. Let $A'$ be a flat $A$-algebra and denote by $F'$ the
multiplicative law corresponding to $A'$.
Then $I_F^{(1)}B'=I_{F'}^{(1)}$. In particular, if $A$ is a
$\Q$-algebra then the kernel commutes with flat base change.
\end{proposition}
\begin{proof}
We reduce to $B$ a finite $A$-algebra as in the proof of
Proposition~\pref{P:image-and-localization}. For any $y\in\Gamma^{d-1}_A(B)$
let $\varphi_y$ be the $A$-module homomorphism $B\to\Gamma^d_A(B)$ given by
$b\mapsto b\times y$. Then $I_f^{(1)}=\bigcap_{y\in\Gamma^{d-1}_A(B)}
\ker(f\circ \varphi_y)$. As $\Gamma^{d-1}_A(B)$ is a finitely generated
$A$-module and $\varphi_y$ is linear in $y$, this intersection coincides with
an intersection over a finite number of $y$'s. As both finite intersections and
kernels commute with flat base change the first statement of the proposition
follows. The last statement follows from
Proposition~\pref{P:supp-filtration-stability}.
\end{proof}

Recall that a monic polynomial $g\in A[t]$ is \emph{separable} if
$(g,g')=A[t]$, where $g'$ is the formal derivative of $g$. Further recall that
$A\inj A[t]/g$ is \emph{\etale} if and only if $g$ is separable.
We will need the following basic lemma to which we, for a lack of suitable
reference, include a proof.

\begin{lemma}\label{L:etale-independence}
Let $A\inj A'=A[t]/g$ be an \etale{} homomorphism, i.e., such that $g$ is a
separable polynomial. If $A$ is a local ring of residue characteristic $p>0$
then for any prime power $q=p^s$, $s\in\N$, the elements
$1,t^q,t^{2q},\dots,t^{(n-1)q}$ form an $A$-module basis of $A'$ where
$n=\deg(g)$.
\end{lemma}
\begin{proof}
\newcommand\tbar{\overline{t}}
Let $k=A/\im_A$. By Nakayama's lemma it is enough to show that a basis of
$A'/\im_A A'=k[t]/\overline{g}$ over $k$ is given by
$1,t^q,t^{2q},\dots,t^{(n-1)q}$. Replacing $A$, $A'$ and $g$ with
$k$, $A'/\im_A A'$ and $\overline{g}$ respectively, we can thus assume that
$A=k$ is a field of characteristic $p$.

Let $g=g_1 g_2\dots g_m$ be a factorization of $g$ into irreducible
polynomials. We have that $A'=k[t]/g=k'_1\times k'_2\times\dots\times k'_m$
where $k\inj k'_i=k[t]/g_i$ are separable field extensions. The subring
generated by $t^q$ is the image of $k[t^q]/g^q=\prod k[t^q]/g_i^q$ in $\prod_i
k'_i$. To show that $t^q$ generates $k[t]/g$ it is thus enough to show that its
image in $k'_i$ generates $k'_i$ for every $i$. Thus, we can assume that $g$ is
irreducible such that $A'=k[t]/g=k'$ is a field.

The field extension $k\inj k(t^q)\inj k(t)=k'$ is separable which shows that
so is $k(t^q)\inj k(t)$. Thus $k(t^q)=k(t)$ and $t^q$ generates $k'$.
\end{proof}

\begin{lemma}\label{L:image-and-pol-extensions}
Let $\map{F}{B}{A}$ be a multiplicative polynomial law of degree $d$.
Let $A'=A[t]/g$ where either $g=0$ or $g$ is separable. Then 
$I_F^{(k)}$ and $\ker(F)$ commute with the base change $A\inj A'$.
\end{lemma}
\begin{proof}
If $g=0$ we let $n=\infty$ and otherwise we let $n=\deg(g)$. A basis of $A'$
as an $A$-module is then given by $1,t,t^2,\dots,t^{n-1}$. By
Proposition~\pref{P:image-and-localization} we can assume that $A$ is a local
ring. Let $p$ be the exponential characteristic of the residue field $A/\im_A$,
i.e., $p$ equals the characteristic if it is positive and $1$ if the
characteristic is zero.

We will proceed by induction on $k$ to show that $I_F^{(k)}B'=I_{F'}^{(k)}$.
As $I_F^{(0)}=B$ and $I_{F'}^{(0)}=B'$ the case $k=0$ is obvious.
Proposition~\pref{P:supp-filtration-stability} shows that
$I_F^{(k)}=I_F^{(k-1)}$ if $k\neq p^s$ and we can thus assume that $k=p^s$.

Let $x'\in I_{F'}^{(p^s)}\subseteq I_{F'}^{(p^s-1)}$. By induction $x'\in
I_F^{(p^s-1)}B'$ and we can thus write uniquely $x'=\sum_{i=0}^{n-1} x_i t^i$
where $x_i\in I_F^{(p^s-1)}$ are almost all zero. Let
$y\in\Gamma^{d-p^s}_A(B)$. Then
$$f'\bigl(\gamma^{p^s}(x')\times y\bigr)=
\sum_{i=0}^{n-1} t^{p^s i} f\bigl(\gamma^{p^s}(x_i)\times y\bigr).$$
If $g=0$ then $1,t^{p^s},t^{2p^s},\dots$ are linearly independent in
${A'=A[t]}$. If $g$ is separable then $1,t^{p^s},t^{2p^s},t^{(n-1)p^s}$ are
linearly independent by Lemma~\pref{L:etale-independence}. This shows that
$f\bigl(\gamma^{p^s}(x_i)\times y\bigr)=0$ for every $y$ and thus $x_i\in
I_F^{(p^s)}$ as $x_i\in I_F^{(p^s-1)}$. Hence $x'\in I_F^{(p^s)}B'$ which shows
that $I_F^{(p^s)}B'=I_{F'}^{(p^s)}$.
\end{proof}

\end{subsection}


\begin{subsection}{Image and base change}\label{SS:image-and-base-change}

As the kernel of a multiplicative law commutes with localization by
Proposition~\pref{P:image-and-localization} it is possible to define the kernel
for a multiplicative law for schemes:

\begin{definition}\label{D:FCI}
Let $S$ be a scheme, $\sA$ a quasi-coherent sheaf of $\sO_S$-algebras and
$\map{F}{\sA}{\sO_S}$ a multiplicative polynomial law,
cf.~\S\ref{SS:Repr-affine}. We let $\ker(F)\subseteq \sA$ be the
quasi-coherent ideal sheaf given by $\ker(F)|_U=\ker\bigl(F|_{U}\bigr)$ for any
affine open subset $U\subseteq S$.
If $\map{f}{X}{S}$ is an affine morphism of schemes and
$\map{\alpha}{S}{\Gamma^d(X/S)}$ is a morphism then we let the \emph{image} of
$\alpha$, denoted $\FCI(\alpha)$, be the closed subscheme of $X$ corresponding
to the ideal sheaf $\ker(F_\alpha)$ where $\map{F_\alpha}{f_*\sO_X}{\sO_S}$ is
the polynomial law corresponding to $\alpha$.
\end{definition}

We say that a morphism $S'\to S$ is \emph{essentially smooth} if every local
ring of $S'$ is a local ring of a scheme which is smooth over $S$. The results
of the previous section are summarized in the following proposition.

\begin{theorem}\label{T:image-and-ess-smooth}
Let $\map{f}{X}{S}$ be an affine morphism of schemes and let
$\map{\alpha}{S}{\Gamma^d(X/S)}$ be a morphism. If $S'\to S$ is an
\emph{essentially smooth} morphism then $\FCI(\alpha)\times_S
S'=\FCI(\alpha\times_S S')$, i.e., the image commutes with essentially smooth
base change.
\end{theorem}
\begin{proof}
As $\FCI(\alpha)$ commutes with localization we can assume that $S=\Spec(A)$
is local and that $S'\to S$ is smooth. Further it is enough that for any
$x\in S'$ there is an affine neighborhood $S''\subseteq S'$ such that the
image commutes with the base change $S''\to S$. By \cite[Cor.~17.11.4]{egaIV}
we can choose $S''$ such that $S''\to S$ is the composition of an \etale{}
morphism followed by a morphism
$\A{n}_S=\Spec(A[t_1,t_2,\dots,t_n])\to S=\Spec(A)$. We can thus assume that
either $S'\to S$ is \etale{} or $S'=\A{1}_S$. 

If $S'\to S$ is \etale{} and $S=\Spec(A)$ is local, then for any $s'\in S'$ we
have that $\sO_{S',s'}=A[t]/g$ where $g\in A[t]$ is a separable
polynomial~\cite[Thm.~18.4.6 (ii)]{egaIV} and it is thus enough to consider
base changes $S'\to S$ of the form $A\to A[t]/g$. The result now follows from
Lemma~\pref{L:image-and-pol-extensions}.
\end{proof}

\begin{corollary}\label{C:image-and-hensel}
Let $S=\Spec(A)$ and $S'=\Spec(A')$ such that $A'$ is a direct limit of
essentially smooth $A$-algebras. Let $\map{f}{X}{S}$ be an affine morphism and
let $\map{\alpha}{S}{\Gamma^d(X/S)}$ be a morphism.  Then
$\FCI(\alpha')=\FCI(\alpha)\times_S S'$. In particular this holds if $S'$ is
the henselization or the strict henselization of a local ring of $S$.
\end{corollary}
\begin{proof}
Follows from Proposition~\pref{P:image-and-limits}
and Theorem~\pref{T:image-and-ess-smooth}.
\end{proof}

\begin{remark}
If $S$ and $S'$ are locally noetherian and $S'\to S$ is a flat morphism with
geometrically regular fibers, then $S'$ is a filtered direct limit of smooth
morphisms by Popescu's
theorem~\cite{swan_popescus_theorem,spivakovsky_popescus_theorem}. Thus the
image of a family $\map{\alpha}{S}{\Gamma^d(X/S)}$ commutes with the base
change $S'\to S$ under this hypothesis. In particular we can apply this with
$S'=\Spec(\widehat{\sO_{S,s}})$ for $s\in S$ if $S$ is an excellent
scheme~\cite[Def.~7.8.2]{egaIV}.
\end{remark}

\begin{definition}\label{D:FCS}
Let $\map{f}{X}{S}$ be an affine morphism of \emph{algebraic spaces} and let
$\map{\alpha}{S}{\Gamma^d(X/S)}$ be a morphism. We let $\FCI(\alpha)$ be the
closed subspace of $X$ such that for any \emph{scheme} $S'$ and \etale{}
morphism $S'\to S$ we have that $\FCI(\alpha)\times_S S'=\FCI(\alpha\times_S
S')$. As \etale{} morphisms descend closed subspaces and the image commutes
with \etale{} base change, this is a unique and well-defined closed
subspace. When $S$ is a scheme, this definition of $\FCI(\alpha)$ and the one
in Definition~\pref{D:FCI} agree. We let $\FCS(\alpha)=\FCI(\alpha)_\red$ and
call this subscheme the \emph{support} of $\alpha$.
\end{definition}

\begin{theorem}\label{T:supp-and-arbitrary-base-change}
Let $S$ and $X$ be algebraic spaces such that $X$ is affine over $S$. Let
$\map{\alpha}{S}{\Gamma^d(X/S)}$ be a morphism and let $S'\to S$ be any
morphism. Then $\bigl(\FCS(\alpha)\times_S S'\bigr)_\red=\FCS(\alpha\times_S
S')$, i.e., the support commutes with arbitrary base change.
\end{theorem}
\begin{proof}
We can assume that $S=\Spec(A)$ and $S'=\Spec(A')$ are affine. Let $P$ be a,
possibly infinite-dimensional, polynomial algebra over $A$ such that there is
a surjection $P\to A'$. Then as $\Spec(P)$ is a limit of smooth $S$-schemes we
can by Theorem~\pref{T:image-and-ess-smooth} replace $A$ with $P$ and assume
that $A\to A'$ is surjective.

Let $X=\Spec(B)$, let $\map{f}{\Gamma^d_A(B)}{A}$ correspond to $\alpha$
and let $\map{F}{B}{A}$ be the corresponding multiplicative law. Pick an
element $b'\in \ker(F_{A'})\subseteq B\otimes_A A'$ and choose a lifting
$b\in B$ of $b'$. Then by Lemma~\pref{L:elementwise_criteria_for_ker},
the elements $f(\gamma^{d-k}(b)\times \gamma^{k}(1))$, $k=0,1,2,\dots,d-1$
lie in the kernel of $A\to A'$. In particular, the image of $\chi_{F,b}(b)$
in $B$ is $b'^d$. Thus $\ker(F_{A'})\subseteq \sqrt{I_{\CH}(F)}(B\otimes_A
A')$. As $\sqrt{I_{\CH}(F)}=\sqrt{\ker{F}}$ by
Proposition~\pref{P:CH-supp} the theorem follows.
\end{proof}

\begin{examples}\label{Es:kernel-does-not-commute}
We give two examples. The first shows that $\ker(F)$ does not commute with
arbitrary base change even in characteristic zero. The second shows that
$\ker(F)$ does not commute with flat base change in positive characteristic.

\begin{enumerate}
\item Let $A=k[x]$ and $B=k[x,y]/(x^2-y^2)$. Then $B$ is a free $A$-module of
rank $2$. The norm $\map{N}{B}{A}$ is a multiplicative law of degree $2$. It
can further be seen that $\ker(N)=0$. Let $A'=k[x]/x$. Then $B'=B\otimes_A A'
=k[y]/y^2$ is not reduced and by Proposition~\pref{P:supp-over-red} the kernel
of $N'$ cannot be trivial. In fact, we have that $\ker(N')=(y)$.

\item Let $k$ be a field of characteristic $p$ and $A=B=k$. We let
$\map{F}{B}{A}$ be the polynomial law given by $x\mapsto x^p$, i.e., the
Frobenius. Clearly $\ker(F)=0$. Let $A'=A[t]/t^p$ which is a flat
$A$-algebra. Then $\ker(F')=(t)$ as $(b''+tx'')^p={b''}^p+t^p{x''}^p={b''}^p$
for any $A'\to A''$ and $b'',x''\in B''=A''$.

It is further easily seen that $\ker(F)$ does not commute with any base change
such that $A'$ is not reduced. In fact, if $t\in A'$ is such that $t^p=0$ then
$t\in\ker(F')$.
\end{enumerate}
\end{examples}

\end{subsection}


\begin{subsection}{Various properties of the image and support}
\label{SS:various-props-of-image-and-support}
A morphism $\map{\alpha}{S}{\Gamma^d(X/S)}$ is, as we will see later on, a
``family of zero cycles of degree $d$ on $X$ parameterized by $S$''.
The subscheme $\FCS(\alpha)\inj X$ is the support of this family
of cycles. In particular it should, topologically at least, have finite fibers
over $S$.

\begin{proposition}\label{P:supp-at-most-d-components}
Let $S$ be a \emph{connected} algebraic space and $X$ a space affine over
$S$. Let $\map{\alpha}{S}{\Gamma^d(X/S)}$ be a morphism. If $X=\coprod_{i=1}^n
X_i$, then there are uniquely defined integers $d_1,d_2,\dots,d_n\in\N$ such
that $d=d_1+d_2+\dots+d_n$ and such that $\alpha$ factors through the closed
subspace $\Gamma^{d_1}(X_1)\times_S
\Gamma^{d_2}(X_2)\times_S\dots\times_S \Gamma^{d_n}(X_n)\inj \Gamma^d(X/S)$.
The support $\FCS(\alpha)$ is contained in the union of the $X_i$'s with
$d_i>0$. In particular $\FCS(\alpha)$ has at most $d$ connected components.
\end{proposition}
\begin{proof}
By Proposition~\pref{P:Gamma^d_of_disj_union} there is a decomposition
$$\Gamma^d(X/S)=\coprod_{\substack{d_i\in\N\\\sum_i d_i=d}}
\Gamma^{d_1}(X_1)\times_S \Gamma^{d_2}(X_2)\times_S\dots\times_S \Gamma^{d_n}(X_n).$$
As $S$ is connected $\alpha$ factors uniquely through one of the spaces in this
decomposition. It is further clear that $X_i\cap \FCS(\alpha)\neq\emptyset$ if
and only if $d_i>0$. The last observation follows after replacing $X$ with
$\FCI(\alpha)$ as then $n$ is at most $d$ in any decomposition.
\end{proof}

\begin{definition}
Let $S$ and $X$ be as in Proposition~\pref{P:supp-at-most-d-components}. The
\emph{multiplicity} of $\alpha$ on $X_i$ is the integer $d_i$.
\end{definition}

\begin{proposition}\label{P:supp-over-field}
Let $S=\Spec(k)$ where $k$ is a field and let $X/S$ be an affine scheme. Let
$\map{\alpha}{S}{\Gamma^d(X/S)}$ be a morphism. Then
$\FCI(\alpha)=\FCS(\alpha)=\coprod_{i=1}^n\Spec(k_i)$ is a disjoint union of a
at most $d$ points such that the separable degree of each $k_i/k$ is
finite.
\end{proposition}
\begin{proof}
Propositions~\pref{P:supp-over-red} and~\pref{P:supp-integrality} shows that
$\FCI(\alpha)$ is reduced and affine of dimension zero, hence totally
disconnected. By Proposition~\pref{P:supp-at-most-d-components} it is thus a
disjoint union of at most $d$ reduced points. As the support commutes with
arbitrary base change by
Theorem~\pref{T:supp-and-arbitrary-base-change}, it follows after
considering the base change $k\inj \overline{k}$ that the separable degree of
$k_i/k$ is finite.
\end{proof}



\begin{corollary}\label{C:push-fwd-and-supp}
Let $X$, $Y$ and $S$ be algebraic spaces with affine morphisms $\map{f}{X}{Y}$
and $\map{g}{Y}{S}$. Let $\map{\alpha}{S}{\Gamma^d(X/S)}$ be a morphism and
denote by $f_*\alpha$ the composition of $\alpha$ and the morphism
$\map{\Gamma^d(f)}{\Gamma^d(X/S)}{\Gamma^d(Y/S)}$. Then
$\FCS(f_*\alpha)=f\bigl(\FCS(\alpha)\bigr)$.
\end{corollary}
\begin{proof}
As the support and the set-theoretic image commute with any base change, we can
assume that $S=\Spec(k)$ where $k$ is a field. Then
$$\FCI(\alpha)=\coprod_{i=1}^n \Spec(k_i)=\{x_1,x_2,\dots,x_n\}$$
by Proposition~\pref{P:supp-over-field}. Further, by
Proposition~\pref{P:Gamma^d_of_disj_union} there are positive integers
$d_1,d_2,\dots,d_n$ such that $\alpha$ factors through $\prod_{i=1}^n
\Gamma^{d_i}\bigl(\Spec(k_i)\bigr)\inj \Gamma^d(X/S)$. Let
$$f\bigl(\FCI(\alpha)\bigr)=\coprod_{j=1}^m
\Spec(k'_j)=\{y_1,y_2,\dots,y_m\}$$
where $m\leq n$. It is then immediately seen that $f_*\alpha$ factors through
the closed subspace
$\prod_{j=1}^m\Gamma^{e_j}\bigl(\Spec(k'_j)\bigr)\inj \Gamma^d(Y/S)$ where
$e_j=\sum_{f(x_i)=y_j} d_i$. As $d_i$ is positive so is $e_j$ and thus $y_j\in
\FCS(f_*\alpha)$. This shows that $\FCS(f_*\alpha)=f\bigl(\FCS(\alpha)\bigr)$.
\end{proof}

\begin{proposition}\label{P:supp-comp-dom}
Let $X$ be an algebraic space affine over $S$ and let
$\map{\alpha}{S}{\Gamma^d(X/S)}$ be a morphism. Then every irreducible
component of $\FCS(\alpha)$ maps onto an irreducible component of $S$.
\end{proposition}
\begin{proof}
As the support commutes with any base change it is enough to
consider the case where $S=\Spec(A)$ is irreducible, reduced and affine. Let
$\FCI(\alpha)=\Spec(B)$ and $\map{F}{B}{A}$ be the multiplicative polynomial
law corresponding to $\alpha$. We have a commutative diagram
$$\xymatrix{
{\Gamma^d_A(B)}\ar[d]_f\ar@{->>}[r] & {\Gamma^d_A(B/I)}\ar[r]
  & {\Gamma^d_{K(A)}\bigl(B\otimes_A K(A)\bigr)}\ar[d]_-{f} \\
{A}\ar@{(->}[rr] && {K(A)}\ar@{}[ull]|\circ}
$$
where $I=\ker\bigl(B\to B\otimes_A K(A)\bigr)$. This shows that
$I\subseteq\ker(F)=0$. As $V(I)$ is the union of the irreducible
components of $\FCS(\alpha)$ which dominate $S$ this shows that every
component surjects onto $S$.
\end{proof}

In the following theorem we restate the main properties of the image and
support of a family of cycles:

\begin{theorem}\label{T:supp-properties}
Let $X$ be an algebraic space affine over $S$ and
let $\map{\alpha}{S}{\Gamma^d(X/S)}$ be a morphism. Then
\begin{enumerate}
\item If $S$ is reduced then $\FCI(\alpha)$ is reduced.\label{IS:reduced}
\item $\FCI(\alpha)\to S$ is integral.\label{IS:integral}
\item If $S$ is connected then $\FCS(\alpha)$ has at most $d$ connected
components.\label{IS:<=d-comps}
\item If $S=\Spec(k)$ where $k$ is a field then
$\FCI(\alpha)=\coprod_{i=1}^n\Spec(k_i)$ is a disjoint union of a finite
number of points, at most $d$, such that the separable degree of each
$k_i/k$ is finite.
\label{IS:field}
\item $\FCS(\alpha)\to S$ has universally topologically finite fibers,
cf.\ Definition~\pref{D:tff}. Moreover, each fiber has at most $d$ points.
\label{IS:utff}
\item If $S$ is a semi-local scheme, i.e., the spectrum of a semi-local ring,
then $\FCS(\alpha)$ is semi-local.
\label{IS:semi-local}
\item Every irreducible component of $\FCS(\alpha)$ maps onto an irreducible
component of $S$.
\label{IS:comp-dom}
\end{enumerate}
\end{theorem}
\begin{proof}
Properties \ref{IS:reduced} and \ref{IS:integral} follows from
Propositions~\pref{P:supp-over-red} and \pref{P:supp-integrality}
respectively.  Properties \ref{IS:<=d-comps} and \ref{IS:field} are
Propositions~\pref{P:supp-at-most-d-components} and~\pref{P:supp-over-field}
respectively.
Property \ref{IS:utff} follows from \ref{IS:field} as the support commutes with
any base change and property \ref{IS:semi-local} follows immediately from
\ref{IS:integral} and \ref{IS:utff}. Property \ref{IS:comp-dom} is
Proposition~\pref{P:supp-comp-dom}.
\end{proof}

The following examples show that the support is not always finite.


\begin{example}
Let $k=\F_p(t_1,t_2,\dots)$ and $K=\F_p(t_1^{1/p},t_2^{1/p},\dots)$. We have a
polynomial law $\map{F}{K}{k}$ given by $a\mapsto a^p$. The support of the
corresponding family $\map{\alpha}{\Spec(k)}{\Gamma^d(\Spec(K))}$ is $\Spec(K)$
and $k\inj K$ is not finite.
\end{example}


The following example shows that even if $X\to S$ is of finite presentation
then the image of a family $\map{\alpha}{S}{\Gamma^d(X/S)}$ need not
be of finite presentation.

\begin{example}
Let $X=S=\Spec(A)$ where $A=k[t_1,t_2,\dots]/(t_1^p,t_2^p,\dots)$ and
$k$ is a field of characteristic $p$. Let $\alpha$ correspond to the
multiplicative polynomial law $\map{F}{A}{A}$, $x\mapsto x^p$. Then, as in
Examples~\pref{Es:kernel-does-not-commute} the kernel of $F$ is
$(t_1,t_2,\dots)$ which is not finitely generated. Hence
$\FCI(\alpha)=\Spec(k)\inj X$ is not finitely presented over $S$.
\end{example}

\end{subsection}

\begin{subsection}{Topological properties of the support}
\label{SS:top-props-of-support}

\begin{definition}[{\cite[Def.~3.9.2]{egaI_NE}}]
We say that a morphism of algebraic spaces $\map{f}{X}{Y}$ is \emph{generizing}
if for any $x\in X$ and generization $y'\in Y$ of $y=f(x)$ there exists a
generization $x'$ of $x$ such that $f(x')=y'$. Equivalently, if $X$ and $Y$
are schemes, the image of $\Spec(\sO_{X,x})$ by $f$ is $\Spec(\sO_{Y,y})$.
We say that $f$ is \emph{component-wise dominating} if every irreducible
component of $X$ dominates an irreducible component of $Y$.
We say that $f$ is \emph{universally generizing} (resp.
\emph{universally component-wise dominating}) if $\map{f'}{X'}{Y'}$ is
generizing (resp. dominating) for any morphism $\map{g}{Y'}{Y}$ where
$X'=X\times_Y Y'$. 
\end{definition}

\begin{remark}
A morphism $\map{f}{X}{Y}$ is generizing (resp. universally generizing) if and
only if $f_\red$ is generizing (resp. universally generizing). If
$\map{g}{Y'}{Y}$ is a generizing surjective morphism, we have that $f$ is
generizing if $f'$ is generizing. If $\map{g}{Y'}{Y}$ is a universally
generizing surjective morphism, then $f$ is generizing (resp. universally
generizing) if and only if $f'$ is generizing (resp. universally generizing).
Any flat morphism $Y'\to Y$ of algebraic spaces is universally generizing.
\end{remark}

\begin{lemma}\label{L:univ-gen_is_univ-dom}
Let $\map{f}{X}{Y}$ be a morphism of algebraic spaces. Then $f$ is universally
generizing if and only if it is universally component-wise dominating.
\end{lemma}
\begin{proof}
A generizing morphism is component-wise dominating so the condition is
necessary. For sufficiency, assume that $f$ is universally component-wise
dominating. Let $x\in X$, $y=f(x)$ and choose a generization $y'\in Y$. Let
$Y'=\overline{\{y'\}}$ with the reduced structure and consider the base
change $Y'\inj Y$. As $f'$ is component-wise dominating, there is a
generization $x'$ of $x$ above $y'$.
\end{proof}

\begin{proposition}\label{P:supp-univ-gen}
Let $\map{f}{X}{S}$ be an affine morphism of algebraic spaces. Let
$\map{\alpha}{S}{\Gamma^d(X/S)}$ be a family with support
$Z=\FCS(\alpha)\inj X$. Then $f|_Z$ is universally generizing.
\end{proposition}
\begin{proof}
Follows immediately from Lemma~\pref{L:univ-gen_is_univ-dom} as the support of
a family of cycles is universally component-wise dominating by
Theorems~\pref[IS:comp-dom]{T:supp-properties}
and~\pref{T:supp-and-arbitrary-base-change}.
\end{proof}

\begin{remark}
If $Z\to S$ is of finite presentation, e.g., if $S$ is locally noetherian
and $X\to S$ is locally of finite type, then it immediately follows that
$f|_{Z}$ is \emph{universally open} from \cite[Prop.~7.3.10]{egaI_NE}.
We will show that $f|_{Z}$ is universally open without any hypothesis on $f$.
The following lemma settles the case when $X\to S$ is locally of finite type.
\end{remark}

\begin{lemma}\label{L:supp-almost-fp}
Let $S$ and $X$ be \emph{affine} schemes and $\map{f}{X}{S}$ a morphism
of \emph{finite type}. Let $\map{\alpha}{S}{\Gamma^d(X/S)}$ be a family of
cycles and $Z=\FCS(\alpha)$ its support. There is then a bijective
closed immersion $Z\inj Z'$ such that $Z'$ is of finite presentation over $S$.
\end{lemma}
\begin{proof}
Let $S=\Spec(A)$, $Z=\Spec(B)$ and let $\map{F}{B}{A}$ be the multiplicative
law corresponding to $\alpha$ restricted to its image. Let
$C=A[t_1,t_2,\dots,t_n]\to B$ be a surjection. The multiplicative law $F$
induces a multiplicative law $\map{G}{C\surj B}{A}$. Note that
$B=C/\ker(G)$. Corresponding to $G$ is a homomorphism
$\map{g}{\Gamma^d_A(C)\surj\Gamma^d_A(B)}{A}$. As $\Gamma^d_A(C)$ is a finitely
presented $A$-algebra, cf.\ Proposition~\pref{P:Gamma^d_finiteness}, this
homomorphism descends to a homomorphism $\map{g_0}{\Gamma^d_{A_0}(C_0)}{A_0}$
with $A_0$ noetherian such that $C=C_0\otimes_{A_0} A$ and
$g=g_0\otimes_{A_0} \id{A}$. As $A_0$ is noetherian $C_0/\ker(G_0)$ is a finite
$A_0$-algebra of finite presentation.

Let $Z_0=\Spec\bigl(C_0/\ker(G_0)\bigr)$ and $Z'=Z_0\times_{\Spec(A_0)}
\Spec(A)$. As the support commutes with base change by
Theorem~\pref{T:supp-and-arbitrary-base-change} we have that $Z\inj Z'$
is a bijective closed immersion.
\end{proof}

\begin{proposition}\label{P:supp-univ-open}
Let $S$ and $X$ be algebraic spaces and $\map{f}{X}{S}$ an affine morphism. Let
$Z$ be the support of a family $\map{\alpha}{S}{\Gamma^d(X/S)}$. Then the
restriction of $f$ to $Z$ is \emph{universally open}.
\end{proposition}
\begin{proof}
The statement is \etale{}-local so we can assume that $S=\Spec(A)$ and
$Z=\Spec(B)$. Further as the support commutes with any base change,
cf.\ Theorem~\pref{T:supp-and-arbitrary-base-change}, it is enough to
show that $\map{f|_Z}{Z}{S}$ is open.

We can write $B$ as a filtered direct limit of finite
$A$-subalgebras $B_\lambda\inj B$. Let $Z_\lambda=\Spec(B_\lambda)$. As
$B_\lambda\inj B$ is integral and injective it follows that
$Z\to Z_\lambda$ is closed and dominating and thus surjective. Let
$\map{\alpha}{S}{\Gamma^d(Z/S)}$ be a family with support $Z$ and let
$\map{\alpha_\lambda}{S}{\Gamma^d(Z_\lambda/S)}$ be the family given by
push-forward along $\map{\varphi_\lambda}{Z}{Z_\lambda}$.

By Corollary~\pref{C:push-fwd-and-supp} we have that
$\FCS(\alpha_\lambda)=\varphi_\lambda(Z_\red)=(Z_\lambda)_\red$.
Further by Lemma~\pref{L:supp-almost-fp} there is a scheme $Z'_\lambda$ of
finite presentation over $S$ such that $\FCS(\alpha_\lambda)$ and $Z_\lambda$
are homeomorphic to $Z'_\lambda$. As $\FCS(\alpha_\lambda)\to S$ is
generizing by Proposition~\pref{P:supp-univ-gen} so is
$Z'_\lambda\to S$. As $Z'_\lambda\to S$ is also of finite presentation it is
\emph{open} by \cite[Prop.~7.3.10]{egaI_NE} and hence so is
$Z_\lambda\to S$.

To show that $\map{f|_Z}{Z}{S}$ is open it is enough to show that the image
of any quasi-compact open subset of $Z$ is open. Let $U\subseteq Z$ be a
quasi-compact open subset. Then according to~\cite[Cor.~8.2.11]{egaIV} there is
a $\lambda$ and $U_\lambda\subseteq Z_\lambda$ such that
$U={\varphi_\lambda}^{-1}(U_\lambda)$. As $\varphi_\lambda$ is surjective and
$Z_\lambda\to S$ is open this shows that $f|_Z (U)$ is open.
\end{proof}

\end{subsection}

\end{section}


\begin{section}{Definition and representability of $\FGamma^d_{X/S}$}
\label{S:representability}

We will define a functor $\FGamma^d_{X/S}$ and show that when $X/S$ is affine
it is represented by $\Gamma^d(X/S)$. It is then easy to prove that
$\FGamma^d_{X/S}$ is represented by a scheme for any \AF{}-scheme $X/S$. To
prove representability in general, i.e., when $X/S$ is any separated algebraic
space, is more difficult. For \emph{any} morphism $\map{f}{X}{Y}$ there is a
natural transformation $\map{f_*}{\FGamma^d_{X/S}}{\FGamma^d_{Y/S}}$ which is
``push-forward of cycles''. If $f$ is \etale{}, then $f_*$ is \etale{} over a
certain open subset of $\Gamma^d(X/S)$. We will use this result to show
representability of $\FGamma^d_{X/S}$ giving an explicit \etale{} covering.


\begin{subsection}{The functor $\FGamma^d_{X/S}$}\label{SS:FGamma}

Recall that a morphism of algebraic spaces $\map{f}{X}{S}$ is said to be
\emph{integral} if it is affine and the corresponding homomorphism $\sO_S\to
f_*\sO_X$ is integral. Equivalently, for any affine scheme $T=\Spec(A)$ and
morphism $T\to S$ the space $X\times_S T=\Spec(B)$ is affine and $A\to B$
is integral. Further recall, Proposition~\pref{P:affine-closed-immersion}, that
if $X/S$ is affine and $Z$ is a closed subspace of $X$, then
$\Gamma^d(Z/S)$ is a closed subspace of $\Gamma^d(X/S)$.

\begin{definition}
Let $S$ be an algebraic space and $X/S$ an algebraic space \emph{separated}
over $S$. A family of zero cycles of degree $d$ consists of a closed subscheme
$Z\inj X$ such that $Z\inj X\to S$ is \emph{integral} together with a morphism
$\map{\alpha}{S}{\Gamma^d(Z/S)}$. Two families $(Z_1,\alpha_1)$ and
$(Z_2,\alpha_2)$ are equivalent if there is a closed subscheme $Z$ of both
$Z_1$ and $Z_2$ and a morphism $\map{\alpha}{S}{\Gamma^d(Z/S)}$ such that
$\alpha_i$ is the composition of $\alpha$ and the morphism $\Gamma^d(Z/S)\inj
\Gamma^d(Z_i/S)$ for $i=1,2$.

If $\map{g}{S'}{S}$ is a morphism of spaces and $(Z,\alpha)$ a family of
cycles on $X/S$, we let $g^{*}(Z,\alpha)=\bigl(g^{*}(Z),g^{*}\alpha\bigr)$ be
the pull-back along $g$.
The image and support of a family of cycles $(Z,\alpha)$ is the image and
support of $\alpha$, cf.\ Definitions~cf.~\pref{D:FCI} and~\pref{D:FCS}.
\end{definition}

\begin{remark}\label{R:minimal-repr}
It is clear that the pull-backs of equivalent families are equivalent and that
the image and support of equivalent families coincide. If $(Z,\alpha)$ is a
family then the family $(\FCI(\alpha),\alpha')$ is a minimal representative
in the same equivalence class. Here $\alpha'$ is the restriction of $\alpha$
to its image, i.e., the morphism $S\to \Gamma^d(\FCI(\alpha)/S)$ which composed
with $\Gamma^d(\FCI(\alpha)/S)\inj\Gamma^d(Z/S)$ is $\alpha$.

The pull-back $g^{*}\alpha$ of a minimal representative $\alpha$ will not in
general be a minimal representative. However note that by
Theorem~\pref{T:supp-and-arbitrary-base-change} we have a canonical
bijective closed immersion $\FCI(g^{*}\alpha)\inj g^{*}\FCI(\alpha)$.
\end{remark}

\begin{definition}
We let $\FGamma^d_{X/S}$ be the contravariant functor from $S$-schemes to sets
defined as follows. For any $S$-scheme $T$ we let $\FGamma^d_{X/S}(T)$ be the
set of equivalence classes of families of zero cycles $(Z,\alpha)$ of degree
$d$ of $X\times_S T/T$. For any morphism $\map{g}{T'}{T}$ of $S$-schemes, the
map $\FGamma^d_{X/S}(g)$ is the pull-back of families of cycles as defined
above.
\end{definition}

In the sequel we will suppress the space of definition $Z$ and write
$\alpha\in\FGamma^d_{X/S}(T)$. We will not make explicit use of $Z$. Instead,
we will use the subspace $\FCI(\alpha)\inj X\times_S T$ which is independent on
the choice of $Z$ by Remark~\pref{R:minimal-repr}.

\begin{proposition}\label{P:Gamma-representable:affine}
If $X$ is \emph{affine} over $S$ then the functor $\FGamma^d_{X/S}$ is
represented by the algebraic space $\Gamma^d(X/S)$, defined
in~\S\upref{SS:Repr-affine}, which is affine over~$S$.
\end{proposition}
\begin{proof}
There is a natural transformation from $\FGamma^d_{X/S}$ to
$\Hom_S(-,\Gamma^d(X/S))$ given by composing a family
$\map{\alpha}{T}{\Gamma^d(Z/T)}$ with
$\Gamma^d(Z/T)\inj \Gamma^d(X\times_S T/T)=\Gamma^d(X/S)\times_S T\to
\Gamma^d(X/S)$. If $\map{\alpha}{T}{\Gamma^d(X/S)}$ is any morphism then
$\alpha\times_S \id{T}$ factors through $\Gamma^d(Z/T)\inj\Gamma^d(X\times_S
T/T)$ where $Z\inj X\times_S T$ is the image of $\alpha\times_S \id{T}$.
As $Z$ is integral over $S$ by
Theorem~\pref[IS:integral]{T:supp-properties}, we have that
the morphism $\alpha$ corresponds to a unique equivalence class of families.
It is thus clear that $\Gamma^d(X/S)$ represents $\FGamma^d_{X/S}$.
\end{proof}

\begin{remark}\label{R:Gamma^1}
For an affine morphism of algebraic spaces $X\to S$, we have that
$\Gamma^1(X/S)=X$ and that the $T$-points of $\Gamma^1(X/S)$
parameterizes sections of $X\times_S T\to T$. Thus, for any separated algebraic
space $X/S$ it follows that $\FGamma^1_{X/S}$ parameterizes sections of
$X\to S$ and that $\FGamma^1_{X/S}$ is represented by $X$.
\end{remark}

\begin{proposition}\label{P:Gamma-is-etale-sheaf}
The functor $\FGamma^d_{X/S}$ is a sheaf in the \etale{} topology.
\end{proposition}
\begin{proof}
Let $T$ be an $S$-scheme and $\map{f}{T'}{T}$ an \etale{} surjective
morphism. Let $T''=T'\times_T T'$ with projections $\pi_1$ and $\pi_2$. Given
an element $\alpha'\in\FGamma^d_{X/S}(T')$ such that
$\pi_1^{*}\alpha'=\pi_2^{*}\alpha'$ we have to show that there is a unique
$\alpha\in\FGamma^d_{X/S}(T)$ such that $f^{*}\alpha=\alpha'$. Let $Z'\inj
X\times_S T'$ be the image of $\alpha'$. As the image commutes with
\etale{} base change, cf.\ Theorem~\pref{T:image-and-ess-smooth}, the
image of $\alpha''$ is $Z''=\pi_1^{-1}(Z')=\pi_2^{-1}(Z')$. As closed
immersions satisfy effective descent with respect to \etale{}
morphisms~\cite[Exp.~VIII, Cor.~1.9]{sga1}, there is a closed subspace $Z\inj
X\times_S T$ such that $Z'=Z\times_T T'$. Moreover $Z$ is affine over $T$. Any
$\alpha\in\FGamma^d_{X/S}(T)$ such that $f^{*}\alpha=\alpha'$ is then in the
subset $\FGamma^d_{Z/T}(T)\subseteq\FGamma^d_{X/S}(T)$.  It is thus enough to
show that $\FGamma^d_{Z/S}$ is a sheaf in the \etale{} topology. But
$\FGamma^d_{Z/S}$ is represented by the space $\Gamma^d(Z/S)$ which is affine
over $S$. As the \etale{} topology is sub-canonical, it follows that
$\FGamma^d_{Z/S}$ is a sheaf.
\end{proof}

\begin{proposition}\label{P:Gamma-immersion}
Let $X/S$ and $Y/S$ be separated algebraic spaces. If $\map{f}{X}{Y}$ is an
immersion (resp. a closed immersion, resp. an open immersion) then
$\FGamma^d_{X/S}$ is a locally closed subfunctor (resp. a closed subfunctor,
resp. an open subfunctor) of $\FGamma^d_{Y/S}$.
\end{proposition}
\begin{proof}
Let $T$ be an $S$-scheme and let $\alpha\in \FGamma^d_{X/S}(T)$ be a family
with $Z=\FCI(\alpha)\inj X_T$. Then $Z\inj X_T$ is a closed subscheme such that
$Z\to T$ is integral and hence universally closed. As $Y\to S$ is separated it
thus follows that $Z\inj X_T\inj Y_T$ is a closed subscheme. It follows that
$\FGamma^d_{X/S}$ is a subfunctor of $\FGamma^d_{Y/S}$.

Let $\map{\alpha}{T}{\FGamma^d_{Y/S}}$ be a family of cycles. We have to show
that if $f$ is a closed (resp. open) immersion then there is a closed (resp.
open) subscheme $U\inj T$ such that if $\map{g}{T'}{T}$ and
$\alpha'=g^*\alpha\in\FGamma^d_{X/S}(T')$ then $g$ factors through $U$. 
Let $X_T=X\times_S T$, $Y_T=Y\times_S T$, $Z=\FCI(\alpha)\subset Y_T$ and
$W=Z\cap X_T=Z\times_{Y_T} X_T\inj X_T$.

If $f$ is an open immersion we let $V$ be the closed subset $Y_T\mysetminus
X_T$ and $U$ be the complement of the image of $V\cap Z=Z\mysetminus W$ by
$Z\to T$. Thus $U$ is the open subset of $T$ such that $t\in U$ if and only if
the fiber $Z_t$ does not meet $V$ or equivalently is contained in $W$. As the
support commutes with arbitrary base change, see
Theorem~\pref{T:supp-and-arbitrary-base-change}, it is easily seen that
$Z\times_T T'$ factors through $X_{T'}$ if and only if $T'\to T$ factors
through $U$. Hence $T \times_{\FGamma^d_{Y/S}} \FGamma^d_{X/S} = T|_U$ which
shows that $\FGamma^d_{X/S}$ is an open subfunctor.

If $f$ is a closed immersion we consider the cartesian diagram
$$\xymatrix{
{T \times_{\FGamma^d_{Z/T}} \FGamma^d_{W/T}}\ar@{(->}[d]\ar[r]
 & {\FGamma^d_{W/T}}\ar@{(->}[d]\ar@{(->}[r]
 & {\FGamma^d_{X_T/T}}\ar@{(->}[d]\ar[r]
 & {\FGamma^d_{X/S}}\ar@{(->}[d]      \\
{T}\ar[r]
 & {\FGamma^d_{Z/T}}\ar@{(->}[r]\ar@{}[ul]|\square
 & {\FGamma^d_{Y_T/T}}\ar[r]\ar@{}[ul]|\square
 & {\FGamma^d_{Y/S}.}\ar@{}[ul]|\square
}$$
As $W$ and $Z$ are affine over $S$, the functors $\FGamma^d_{W/T}$ and
$\FGamma^d_{Z/T}$ are represented by $\Gamma^d(W/T)$ and $\Gamma^d(Z/T)$
respectively. As $\Gamma^d(W/T)\inj \Gamma^d(Z/T)$ is a closed immersion by
Proposition~\pref{P:affine-closed-immersion} it follows that $\FGamma^d_{X/S}$
is a closed subfunctor of $\FGamma^d_{Y/S}$.
\end{proof}

\begin{proposition}\label{P:Gamma^d_of_disj_union:general}
Let $S$ be an algebraic space and let $X_1,X_2,\dots,X_n$ be
algebraic spaces separated over $S$. Then
$$\FGamma^d_{\coprod_{i=1}^n X_i}=
\coprod_{\substack{d_i\in\N\\\sum_i d_i=d}}
\FGamma^{d_1}_{X_1}\times_S \FGamma^{d_2}_{X_2}\times_S\dots
\times_S \FGamma^{d_n}_{X_n}.$$
\end{proposition}
\begin{proof}
Follows from Proposition~\pref{P:Gamma^d_of_disj_union}.
\end{proof}

\begin{corollary}\label{C:geom-points-of-Gamma}
Let $X/S$ be a separated algebraic space. Let $k$ be an algebraically closed
field and $\map{s}{\Spec(k)}{S}$ a geometric point of $S$. There is a
one-to-one correspondence between $k$-points of $\FGamma^d_{X/S}$ and effective
zero cycles of degree $d$ on $X_s$. In this correspondence, a zero cycle
$\sum_{i=1}^n d_i[x_i]$ on $X_s$ corresponds to the family $(Z,\alpha)$ where
$Z=\{x_1,x_2,\dots,x_n\}\subseteq X_s$ and $\alpha$ is the morphism
$$\alpha\;:\; \Spec(k) \iso \Gamma^{d_1}(x_1/k)\times_k
\Gamma^{d_2}(x_2/k)\times_k\dots\times_k \Gamma^{d_n}(x_n/k)
\inj \Gamma^d(Z/k)$$
\end{corollary}
\begin{proof}
Let $\alpha\in\FGamma^d_{X/S}(k)$ be a $k$-point. By
Theorem~\pref[IS:field]{T:supp-properties} we have that $Z=\FCI(\alpha)\inj
X_s$ is a finite disjoint union of points $x_1,x_2,\dots,x_n$, all with residue
field $k$ as $k$ is algebraically closed. According to
Proposition~\pref{P:Gamma^d_of_disj_union:general}, there are positive integers
$d_1,d_2,\dots,d_n$ such that $d=d_1+d_2+\dots+d_n$ and such that
$\map{\alpha}{k}{\Gamma^d(Z/k)}$ factors through the open and closed subscheme
$\Gamma^{d_1}(x_1/k)\times_k\Gamma^{d_2}(x_2/k)\times_k\dots\times_k
\Gamma^{d_n}(x_n/k)$. As $k(x_i)=k$, we have that $\Gamma^{d_i}(x_i/k)\iso
k$. The point $\alpha$ corresponds to $\sum_{i=1}^n d_i[x_i]$.
\end{proof}

\begin{proposition}\label{P:cover-of-Gamma:AF}
Let $X/S$ be a separated algebraic space. Let $\{U_\beta\}$ be an open covering
of $X$ such that any set of $d$ points in $X$ above the same point in $S$ lies
in one of the $U_\beta$'s. Then
$\coprod_\beta \FGamma^d_{U_\beta/S}\to \FGamma^d_{X/S}$ is an open covering.
If $X/S$ is an \AF{}-scheme then such a covering with the $U_\beta$'s
\emph{affine} exists.
\end{proposition}
\begin{proof}
Let $k$ be a field and $\alpha\in\FGamma^d_{X/S}(k)$. Then by
Theorem~\pref[IS:field]{T:supp-properties} there is a $\beta$ such that
$\alpha\in\FGamma^d_{U_\beta/S}(k)\subseteq \FGamma^d_{X/S}(k)$. Thus
$\coprod_\beta \FGamma^d_{U_\beta/S}\to \FGamma^d_{X/S}$ is an open covering by
Proposition~\pref{P:Gamma-immersion}.
\end{proof}

\begin{theorem}\label{T:Gamma-representable:AF}
Let $S$ be a scheme and $X/S$ an \AF{}-scheme. The functor $\FGamma^d_{X/S}$ is
then represented by an \AF{}-scheme $\Gamma^d(X/S)$.
\end{theorem}
\begin{proof}
As $\FGamma^d_{X/S}$ is a sheaf in the Zariski topology, we can assume that $S$
is affine. Let $\{U_\beta\}$ be an open covering of $X$ by affines such that
any set of $d$ points in $X$ above the same point in $S$ lies in one of the
$U_\beta$'s. As $\FGamma^d_{U_\beta/S}$ is represented by an affine scheme,
Proposition~\pref{P:cover-of-Gamma:AF} shows that $\FGamma^d_{X/S}$ is
represented by a scheme $\Gamma^d(X/S)$.

If $\alpha_1,\alpha_2,\dots,\alpha_m$ are points of $\Gamma^d(X/S)$ above
the same point of $S$, then the union of their supports consists of at most
$dm$ points and there is thus an affine subset $U\subseteq X$ such that
$\alpha_1,\alpha_2,\dots,\alpha_m\in\Gamma^d(U/S)$. This shows that
$\Gamma^d(X/S)/S$ is an \AF{}-scheme.
\end{proof}

\end{subsection}


\begin{subsection}{Effective pro-representability of $\FGamma^d_{X/S}$}
\label{SS:eff-pro-rep}

Let $A$ be a henselian local ring and $T=\Spec(A)$ together with a morphism
$T\to S$. The image of a family of cycles $\alpha\in\FGamma^d_{X/S}(T)$ over
$T$ is then a semi-local scheme $Z$, integral over $T$ by
\ref{IS:integral} and \ref{IS:semi-local} of
Theorem~\pref{T:supp-properties}. Furthermore,
Proposition~\pref{P:int_hensel} implies that $Z$ is a finite disjoint union of
local henselian schemes.

Let $z_1,z_2,\dots,z_n$ be the closed points of $Z\inj X_T$ and
$\{x_1,x_2,\dots,x_m\}$ their images in $X$ where the $x_i$'s are chosen to be
distinct. As $z_i$ lies over the closed point of $T$, all $x_i$ lies over a
common point $s\in S$. Let $\hensel{X}_{x_i}=\Spec(\hensel{\sO_{X,x_i}})$,
$\hensel{X}_{x_1,x_2,\dots,x_m}=\coprod_{i=1}^m \hensel{X}_{x_i}$ and
$\hensel{S}_{s}=\Spec(\hensel{\sO_{S,s}})$ be the henselizations of $X$ and
$S$ at the $x_i$'s and $s$. As $\sO_{Z,z_i}$ is henselian it
follows that $Z\inj X_T\to X$ factors uniquely through
$\hensel{X}_{x_1,x_2,\dots,x_m} \to X$. Thus $Z\inj X_T$ factors uniquely
through $\hensel{X}_{x_1,x_2,\dots,x_m}\times_{\hensel{S}_s} T\to X_T$ and
$\alpha$ corresponds to a unique element of
$\FGamma^d_{\hensel{X}_{x_1,x_2,\dots,x_m}/\hensel{S}_s}(T)$. As
$\hensel{X}_{x_1,x_2,\dots,x_m}$ is affine, we have a unique morphism $T\to
\Gamma^d(\hensel{X}_{x_1,x_2,\dots,x_m}/\hensel{S}_s)$.

Further, by Proposition~\pref{P:Gamma^d_of_disj_union}
$$\Gamma^d\left(\hensel{X}_{x_1,x_2,\dots,x_m}/\hensel{S}_s\right)=
\coprod_{\substack{d_i\in\N\\\sum_i d_i=d}}
\prod_{i=1}^m \Gamma^{d_i}\left(\hensel{X}_{x_i}/\hensel{S}_s\right).$$
and as $T$ is connected
$T\to \Gamma^d(\hensel{X}_{x_1,x_2,\dots,x_m}/\hensel{S}_s)$ factors through
one of these components.

To conclude, there are uniquely determined points $x_1,x_2,\dots,x_m\in X$,
unique positive integers $d_i$ and a unique morphism
$$\map{\varphi}{T}{\prod_{i=1}^m
\Gamma^{d_i}\left(\hensel{X}_{x_i}/\hensel{S}_s\right)
\inj \Gamma^d\left(\hensel{X}_{x_1,x_2,\dots,x_m}/\hensel{S}_s\right)}$$
such that $\alpha$ is equivalent to $\varphi\times_{\hensel{S}_s} \id{T}$.
This implies the following:

\begin{proposition}\label{P:strictly-local-rings-of-Gamma}
Let $X/S$ be a separated algebraic space and assume that $\FGamma^d_{X/S}$ is
represented by an algebraic space $\Gamma^d(X/S)$. Let $\beta\in\Gamma^d(X/S)$
be a point with residue field $k$ and $s$ its image in $S$. The point $\beta$
corresponds uniquely to points $x_1,x_2,\dots,x_m\in X$, positive integers
$d_1,d_2,\dots,d_m$ with sum $d$ and morphisms
$\map{\varphi_i}{k}{\Gamma^d\bigl(k(x_i)/k(s)\bigr)}$. The local henselian ring
(resp. strictly local ring) at $\beta$ is the local henselian ring
(resp. strictly local ring) of $\prod_{i=1}^m
\Gamma^{d_i}(\hensel{X}_{x_i}/\hensel{S}_s)$ at the point corresponding to the
morphisms~$\varphi_i$.
\end{proposition}

\begin{xpar}
If $X/S$ is \emph{locally of finite type} and $A$ is a complete local
\emph{noetherian} ring, then the support of any family of cycles $\alpha$ on
$X$ parameterized by $T=\Spec(A)$ is finite over $T$. Thus $\FCI(\alpha)$ is a
disjoint union of a finite number of complete local rings. Let $s\in S$ and
$x_i\in X$ be defined as above and let
$\widehat{X}_{x_i}=\Spec(\widehat{\sO}_{X,x_i})$,
$\widehat{S}_{s}=\Spec(\widehat{\sO}_{S,s})$ and
$\widehat{X}_{x_1,x_2,\dots,x_m}=\coprod_{i=1}^m \widehat{X}_{x_i}$ be the
completions of $X$ and $S$ at the corresponding points. Repeating the
reasoning above we conclude that there is a unique morphism
$$\map{\varphi}{T}
{\prod_{i=1}^m \Gamma^{d_i}(\widehat{X}_{x_i}/\widehat{S}_s)
\inj \Gamma^d(\widehat{X}_{x_1,x_2,\dots,x_m}/\widehat{S}_s)}$$
such that $\alpha$ is equivalent to $\varphi\times_{\widehat{S}_s} \id{T}$.
Thus we obtain:
\end{xpar}

\begin{proposition}\label{P:formal-local-rings-of-Gamma}
Let $S$ be locally noetherian and $X$ an algebraic space separated and
locally of finite type over $S$ and assume that $\FGamma^d_{X/S}$ is
represented by an algebraic space $\Gamma^d(X/S)$.
Let $\beta\in\Gamma^d(X/S)$ be a point with residue field $k$
and $s$ its image in $S$. The point $\beta$ corresponds uniquely to points
$x_1,x_2,\dots,x_m\in X$, positive integers $d_1,d_2,\dots,d_m$ with sum $d$
and morphisms $\map{\varphi_i}{k}{\Gamma^d\bigl(k(x_i)/k(s)\bigr)}$. The
formal local ring at $\beta$ is the formal local ring of
$\prod_{i=1}^m \Gamma^{d_i}(\widehat{X}_{x_i}/\widehat{S}_s)$ at
the point corresponding to the morphisms $\varphi_i$.
\end{proposition}

\begin{corollary} 
Let $S$ be locally noetherian and $X$ an algebraic space separated and locally
of finite type over $S$. The functor $\FGamma^d_{X/S}$ is \emph{effectively
pro-representable} by which we mean the following: Let $k$ be any field and
$\beta_0\in\FGamma^d_{X/S}(k)$. There is then a complete local noetherian
ring $\widehat{A}$ and an object $\widehat{\beta}\in\FGamma^d_{X/S}(\Spec(A))$
such that for any local artinian scheme $T$ and family
$\alpha\in\FGamma^d_{X/S}(T)$, coinciding with $\beta_0$ at the closed point
of $T$, there is a unique morphism $\map{f}{T}{\Spec(\widehat{A})}$ such that
$\alpha=f^{*}\widehat{\beta}$.
\end{corollary}

\begin{remark}
Assume that $\FGamma^d_{X/S}$ is represented by an algebraic space
$\Gamma^d(X/S)$. Questions about properties of $\Gamma^d(X/S)$ which only
depend on the strictly local rings, such as being flat or reduced, can be
reduced to the case where $X$ is affine using
Proposition~\pref{P:strictly-local-rings-of-Gamma}. As some properties cannot
be read from the strictly local rings we will need the stronger result of
Proposition~\pref{P:Gamma-neighborhoods} which shows that any point in
$\Gamma^d(X/S)$ has an \etale{} neighborhood which is an open subset of
$\Gamma^d(U/S)$ for some affine scheme $U$.
\end{remark}

\end{subsection}


\begin{subsection}{Push-forward of families of cycles}\label{SS:push-fwd}

\begin{definition}\label{D:push-forward}
Let $\map{f}{X}{Y}$ be a morphism of algebraic spaces separated over $S$.
If $(Z,\alpha)\in\FGamma^d_{X/S}(T)$ is a family of cycles over $T$ we let
$f_*(Z,\alpha)=(f_T(Z),f_*\alpha)$ where $f_T(Z)$ is the schematic image
of $Z$ along $X\times_S T\to Y\times_S T$ and
$f_*\alpha$ is the composition of $\map{\alpha}{T}{\Gamma^d(Z/T)}$ and
$\Gamma^d(Z/T)\to\Gamma^d(f_T(Z)/T)$. This induces a natural transformation
of functors $\map{f_*}{\FGamma^d_{X/S}}{\FGamma^d_{Y/S}}$ denoted the
push-forward.
\end{definition}

\begin{remark}
If $\map{g}{Y}{Z}$ is another morphism of $S$-spaces then clearly $g_* \circ
f_*=(g\circ f)_*$. If $X$ and $Y$ are affine over $S$, the push-forward
$\map{f_*}{\FGamma^d_{X/S}}{\FGamma^d_{Y/S}}$ coincides with the morphism
$\Gamma^d(X/S)\to \Gamma^d(Y/S)$ given by the covariance of the functor
$\Gamma^d$.

Definition~\pref{D:push-forward} only makes sense after we have checked that
$f_T(Z)$ is integral over $T$. If $Y/S$ is locally of finite type then $f_T(Z)$
is quasi-finite and proper and hence finite,
cf.\ Proposition~\pref{P:image_of_finite}. More generally, as $Z\to T$ is
integral with topological finite fibers by
Theorem~\pref[IS:utff]{T:supp-properties}, it follows from
Theorem~\pref{T:image_of_integral_wtff} that $f_T(Z)$ is integral without any
hypothesis on $Y/S$ except the separatedness.
\end{remark}

\begin{definition}\label{D:regular}
Let $X/S$ and $Y/S$ be separated algebraic spaces and let $\map{f}{X}{Y}$ be
any morphism of $S$-spaces. We say that $\alpha\in\FGamma^d_{X/S}(T)$ is
\emph{regular} (resp. \emph{quasi-regular}) with respect to $f$ if
${f_T}|_{\FCI(\alpha)}$ is a closed immersion (resp. universally injective) or
equivalently if
$\map{{f_T}|_{\FCI(\alpha)}}{\FCI(\alpha)}{f_T\bigl(\FCI(\alpha)\bigr)}$ is
an isomorphism (resp. a universal bijection). We let
$\FGamma^d_{X/S,\reg/f}(T)$ (resp. $\FGamma^d_{X/S,\qreg/f}(T)$) be the
elements which are regular (resp. quasi-regular) with respect to $f$.
\end{definition}

\begin{definition}
Let $\fF$ and $\fG$ be contravariant functors from $S$-schemes to sets.
We say that a morphism of functors $\map{f}{\fF}{\fG}$ is
\emph{topologically surjective} if for any field $k$ and element
$y\in\fG\bigl(\Spec(k)\bigr)$ there is a field extension
$\map{g}{\Spec(k')}{\Spec(k)}$ and an element $x\in\fF\bigl(\Spec(k')\bigr)$
such that $f(x)=g^{*}y$ in $\fG\bigl(\Spec(k')\bigr)$. If $\fF$ and $\fG$ are
represented by algebraic spaces, we have that $f$ is topologically surjective
if and only if the corresponding morphism of spaces is surjective.
\end{definition}

\begin{definition}
A morphism $\map{f}{X}{Y}$ is \emph{unramified} if it is formally unramified
and locally of finite type.
\end{definition}

In~\cite{egaIV} unramified morphisms are locally of finite presentation
but the above definition is more useful and also commonly used.

\begin{proposition}\label{P:unramified:quasi-regular=regular}
Let $X/S$ and $Y/S$ be separated algebraic spaces and let $\map{f}{X}{Y}$ be a
morphism of $S$-spaces. Let $\alpha\in\FGamma^d_{X/S}(T)$. If $f$ is unramified
then $\alpha$ is quasi-regular if and only if $\alpha$ is regular.
\end{proposition}
\begin{proof}
If $\alpha$ is quasi-regular and $f$ unramified then $\FCI(\alpha)\inj
X\times_S T\to Y\times_S T$ is unramified and universally
injective. By~\cite[Prop.~17.2.6]{egaIV} this implies that
$\map{{f_T}|_{\FCI(\alpha)}}{\FCI(\alpha)}{Y\times_S T}$ is a
monomorphism. As $\FCI(\alpha)\to T$ is universally closed and $Y_T\to T$ is
separated it follows that ${f_T}|_{\FCI(\alpha)}$ is a proper monomorphism
and hence a closed immersion~\cite[Cor.~18.12.6]{egaIV}.
\end{proof}

\begin{proposition}\label{P:push-fwd-and-supp}
Let $X/S$ and $Y/S$ be separated algebraic spaces and let $\map{f}{X}{Y}$ be
a morphism of $S$-spaces. Let $T$ be an $S$-scheme and
$\map{f_T}{X\times_S T}{Y\times_S T}$ the base change of $f$ along $T\to S$.
Let $\alpha\in\FGamma^d_{X/S}(T)$. Then
\begin{enumerate}
\item $\FCI(f_*\alpha)\inj f_T\bigl(\FCI(\alpha))$.
  \label{P:push-fwd-and-supp:inj}
\item $\FCS(f_*\alpha)=f_T\bigl(\FCS(\alpha))$.
  \label{P:push-fwd-and-supp:red}
\item $\FCS(\alpha) \to f_T\bigl(\FCS(\alpha))=\FCS(f_*\alpha)$ is a bijection
if $\alpha$ is quasi-regular with respect to $f_*$.
  \label{P:push-fwd-and-supp:q-reg}
\item $\FCI(\alpha)\iso f_T\bigl(\FCI(\alpha))=\FCI(f_*\alpha)$ if $\alpha$
is regular with respect to $f_*$.
  \label{P:push-fwd-and-supp:regular}
\end{enumerate}
\end{proposition}
\begin{proof}
\ref{P:push-fwd-and-supp:inj} follows immediately by the definition of $f_*$
and~\ref{P:push-fwd-and-supp:red} follows from
Corollary~\pref{C:push-fwd-and-supp}.
\ref{P:push-fwd-and-supp:q-reg} follows from the definition of a
quasi-regular family, as $f_T\bigl(\FCS(\alpha))=\FCS(f_*\alpha)$ by
Corollary~\pref{C:push-fwd-and-supp}. \ref{P:push-fwd-and-supp:regular}
follows by the definition of regular as
$\FCI(\alpha)\iso f_T\bigl(\FCI(\alpha)\bigr)$ easily implies
that $f_T\bigl(\FCI(\alpha)\bigr)=\FCI(f_*\alpha)$.
\end{proof}

\begin{examples}\label{Ex:push-fwd-examples}
We give two examples on bad behavior of the image with respect to
push-forward. In the first example $f$ is \etale{}, $\alpha$ not
(quasi\nobreakdash-)regular and $\FCI(f_*\alpha)\inj f_T\bigl(\FCI(\alpha))$
is not an isomorphism. In the second example $f$ is universally injective and
$\alpha$ quasi-regular but not regular.
\begin{enumerate}
\item
Let $S=\Spec(A)$, $Y=\Spec(B)$ and $X=Y\amalg Y=\Spec(B\times B)$ where
$A=k[\epsilon]/\epsilon^2$ and
$B=k[\epsilon,\delta]/(\epsilon^2,\delta^2,\epsilon\delta)$. We let
$\map{f}{X}{Y}$ be the \etale{} map given by the identity on the two
components. Finally we let $\alpha\in\FGamma^2_{X/S}(S)$ be the family of
cycles corresponding to the multiplicative polynomial law
$\mapname{F} B\times B\to B/(\delta-\epsilon)\times B/(\delta+\epsilon)\iso
A\times A \to A\otimes_A A\iso A$ which is homogeneous of degree $2$. The
support of $\alpha$ corresponds to $\ker(F)=
\bigl( (\delta-\epsilon), (\delta+\epsilon)\bigr)\subset B\times B$.
It is easily seen that $f\bigl(\FCI(\alpha)\bigr)=V(0)$. On the other
hand an easy calculation shows that $\FCI(f_*\alpha)=V(\delta)$.

\item
Let $k$ be a field of characteristic different from $2$.  Let $S=\Spec(A)$,
$Y=\Spec(B)$ and $X=\Spec(C)$ where $A=k[\epsilon]/\epsilon^2$,
$B=k[\epsilon,\delta]/(\epsilon,\delta)^2$ and
$C=k[\epsilon,\delta,\tau]/\bigl(\epsilon^2,\epsilon\delta,\epsilon\tau,
\delta^2,\tau^2,\delta\tau-\epsilon\bigr)$. Let $\map{f}{X}{Y}$ be the natural
morphism. An easy calculation shows that $\Gamma^2_A(C)$ is generated by
$\gamma^2(\delta)$, $\gamma^2(\tau)$, $\delta\times 1$, $\tau\times 1$ and
$\delta\times\tau$. After finding explicit relations for these generators in
$\Gamma^2_A(C)$, it can also be shown that
$\gamma^2(\delta),\gamma^2(\tau),\delta\times 1,\tau\times 1\mapsto 0$ and
$\delta\times\tau\mapsto -2\epsilon$ defines a family
$\map{\alpha}{S}{\Gamma^2(X/S)}$. It is easy to check that $\FCI(\alpha)=X$,
$f(\FCI(\alpha))=Y$ but $\FCI(f_*\alpha)=V(\delta)$.
\end{enumerate}
\end{examples}

\begin{proposition}\label{P:reg-functors}
Let $\map{f}{X}{Y}$ be a morphism between algebraic spaces separated over $S$.
Then:
\begin{enumerate}
\item $\FGamma^d_{X/S,\reg/f}$ and $\FGamma^d_{X/S,\qreg/f}$ are subfunctors of
$\FGamma^d_{X/S}$.
\item If $\map{f}{X}{Y}$ is unramified then
$\FGamma^d_{X/S,\reg/f}=\FGamma^d_{X/S,\qreg/f}$ is an open subfunctor of
$\FGamma^d_{X/S}$.\label{PRF:reg-subset-open-if-unramified}
\item If $f$ is an immersion then
$\FGamma^d_{X/S,\reg/f}=\FGamma^d_{X/S,\qreg/f}=\FGamma^d_{X/S}$.
\item If $f$ is surjective then $\FGamma^d_{X/S,\reg/f}\to\FGamma^d_{Y/S}$ is
topologically surjective.\label{PRF:reg-subset-top-surj}
\end{enumerate}
\end{proposition}
\begin{proof}
(i) As the support commutes with arbitrary base change it follows that the
requirement for $\alpha\in\FGamma^d_{X/S}(T)$ to be quasi-regular is stable
under arbitrary base change. Thus the pull-back
$\FGamma^d_{X/S}(T)\to\FGamma^d_{X/S}(T')$ induced by $T'\to T$ restricts to
$\FGamma^d_{X/S,\qreg/f}$. If $\alpha\in\FGamma^d_{X/S}(T)$ is regular then by
definition $\FCI(\alpha)\iso f_T\bigl(\FCI(\alpha)\bigr)=\FCI(f_*\alpha)$. If
$\map{g}{T'}{T}$ is any morphism then clearly $\FCI(g^*\alpha)\iso\FCI(g^*f_*
\alpha)=\FCI(f_*g^*\alpha)$ and thus $g^*\alpha\in\FGamma^d_{X/S,\reg/f}(T')$.

(ii) Proposition~\pref{P:unramified:quasi-regular=regular} shows that
$\FGamma^d_{X/S,\qreg/f}=\FGamma^d_{X/S,\reg/f}$. To show that
$\FGamma^d_{X/S,\reg/f}\subseteq \FGamma^d_{X/S}$ is open we let
$\map{\alpha}{T}{\FGamma^d_{X/S}}$ be a morphism. This factors through
$T\to\Gamma^d(Z/T)$ where $Z=\FCI(\alpha)\inj X_T$ and $X_T=X\times_S T$. As
$f$ is unramified $\injmap{(f_T)|_Z}{Z}{X_T\to Y_T}$ is unramified. In
particular $\map{(f_T)|_Z}{Z}{f_T(Z)}$ is finite and unramified. By Nakayama's
lemma, the rank of the fibers of a finite morphism is upper semicontinuous.
Thus, the subset $W$ of $f_T(Z)$ over which the geometric fibers of $(f_T)|_Z$
contain more than one point is closed. Let $U=T\mysetminus g_T(W)$, where
$\map{g}{Y}{S}$ is the structure morphism. Then
$\FGamma^d_{X/S,\qreg/f}\times_{\FGamma^d_{X/S}} T=U$ which shows that
$\FGamma^d_{X/S,\qreg/f}\subseteq \FGamma^d_{X/S}$ is an open subfunctor.

(iii) Obvious from the definitions.

(iv) Let $\beta\in\FGamma^d_{Y/S}(k)$ where $k=\overline{k}$ is an
algebraically closed field. Then by
Theorem~\pref[IS:field]{T:supp-properties} the image $W:=\FCI(\beta)\inj
Y_k$ is a finite disjoint union of reduced points, each with residue field
$k$. As $f$ is surjective we can then find a field extension $k\inj k'$ and a
closed subspace $Z\inj X_{k'}$ such that $f_{k'}(Z)=W_{k'}$ and
$\map{{f_{k'}}|_Z}{Z}{W_{k'}}$ is an isomorphism. This gives an element
$\alpha\in\FGamma^d_{X/S}(k')$ such that $f_*\alpha=\beta$.
\end{proof}

\begin{proposition}\label{P:pushfwd-diagram}
Let
$$\xymatrix{
X'\ar[r]^{g'}\ar[d]^{f'} & X\ar[d]^{f} \\
Y'\ar[r]^{g} & Y\ar@{}[ul]|\square
}$$
be a cartesian square of algebraic spaces separated over $S$. Let
\begin{align*}
\FGamma^d_{X'/S,\reg/g}
&= \FGamma^d_{X'/S}\times_{\FGamma^d_{Y'/S}}\FGamma^d_{Y'/S,\reg/g}\\
&= \left\{ \alpha\in\FGamma^d_{X'/S}\;:\; f'_*\alpha\text{ is regular with
respect to $g$} \right\}.
\end{align*}
Then
\begin{enumerate}
\item \label{PD:reg-inclusion}
If $g$ is unramified or $f$ is an immersion then
$$\FGamma^d_{X'/S,\reg/g} \subseteq \FGamma^d_{X'/S,\reg/g'}.$$
\item \label{PD:cartesian-if-etale-or-immersion}
If $g$ is \etale{} or $f$ is an immersion then we have a cartesian diagram
$$\xymatrix{
\FGamma^d_{X'/S,\reg/g}\ar[d]^{f'_*}\ar[r]^{g'_*}
  & \FGamma^d_{X/S}\ar[d]^{f_*} \\
\FGamma^d_{Y'/S,\reg/g}\ar[r]^{g_*}
  & \FGamma^d_{Y/S}.\ar@{}[ul]|\square
}$$
\item \label{PD:univhomeo-if-unramified}
For arbitrary $g$ the results of \ref{PD:reg-inclusion} and
\ref{PD:cartesian-if-etale-or-immersion} are true over reduced
$S$-schemes, i.e., for any reduced $S$-scheme $T$ we have that
$$\FGamma^d_{X'/S,\reg/g}(T) \subseteq \FGamma^d_{X'/S,\reg/g'}(T)$$
and the diagram in \ref{PD:cartesian-if-etale-or-immersion} is cartesian in the
subcategory of functors from reduced $S$-schemes.
\end{enumerate}
\end{proposition}
\begin{proof}
\ref{PD:reg-inclusion} Let $\alpha'\in \FGamma^d_{X'/S}(T)$. If $f$ is an
immersion then $\FCI(f'_*\alpha')=\FCI(\alpha')$ and
$\FCI(f_*g'_*\alpha')=\FCI(g'_*\alpha')$. It is thus obvious that $\alpha'$ is
regular if and only if $f'_*\alpha'$ is regular, i.e.,
$\FGamma^d_{X'/S,\reg/g}=\FGamma^d_{X'/S,\reg/g'}$.

Assume instead that $f$ is arbitrary but $g$ is unramified. Let
$Z'=\FCI(\alpha')$ and $W'=f'_T(Z')$.
If $\alpha'\in\FGamma^d_{X'/S,\reg/g}(T)$, i.e., if $f'_*\alpha'$ is regular
with respect to $g$,
we have that $\FCI(f'_*\alpha')\inj W'\inj Y'_T\to Y_T$ is a closed
immersion. But $\FCI(f'_*\alpha')\inj W'$ is universally bijective and thus
$W'\to Y_T$ is universally injective and
unramified. By~\cite[Prop.~17.2.6]{egaIV} this implies that $W'\to Y_T$ is a
monomorphism and hence a closed immersion.
Thus $Z'\inj W'\times_{Y'_T} X'_T=W'\times_{Y_T} X_T\inj X_T$ is a closed
immersion which shows that $\alpha'$ is regular with respect to $g'$.

\ref{PD:cartesian-if-etale-or-immersion}
The commutativity of the diagrams is obvious. This gives us a canonical
morphism
$$\map{\Lambda}{\FGamma^d_{X'/S,\reg/g}}
{\FGamma^d_{X/S}\times_{\FGamma^d_{Y/S}}\FGamma^d_{Y'/S,\reg/g}}.$$
We construct an inverse $\Lambda^{-1}$ of this morphism as follows: Let $T$ be
an $S$-scheme, $\alpha\in\FGamma^d_{X/S}(T)$ and $\beta'\in
\FGamma^d_{Y'/S,\reg/g}(T)$ such that $\beta=g_*\beta'=f_*\alpha
\in\FGamma^d_{Y/S}(T)$. As $\beta'$ is regular with respect to $g$ we have that
$\FCI(\beta')\inj Y'_T$ is isomorphic to $\FCI(\beta)\inj Y_T$. Let
$Z=\FCI(\alpha)\inj X_T$. If $f$ is an immersion then $\alpha$ is regular with
respect to $f$ and $Z \inj X_T$ is isomorphic to $\FCI(\beta)$ and we let
$Z'=\FCI(\beta')\times_{\FCI(\beta)} \FCI(\alpha)\iso Z$.

For arbitrary $f$ but \etale{} $g$, let $W=f_T(Z)$. Then
$\FCI(\beta)\inj W$ is a bijective closed immersion. By the regularity of
$\beta'$, we have that $\FCI(\beta')$ is a
section of $g^{-1}_T\bigl(\FCI(\beta)\bigr)\to \FCI(\beta)$. As $g$ is
unramified it thus follows that $\FCI(\beta')$ is open and closed in
$g^{-1}_T\bigl(\FCI(\beta)\bigr)\inj g^{-1}_TW$. Let $W'$ be the corresponding
open and closed subscheme of $g^{-1}_T W$. As $g$ is
\etale{} $W'\iso W$ and we let $Z'=W'\times_W Z$.

In both cases we have obtained a canonical closed subscheme $Z'\inj X'_T$
such that $Z'\iso Z$. This gives a unique lifting of the family
$\alpha\in\FGamma^d_{Z}(T)$ to a family
$\alpha'\in\FGamma^d_{Z'}(T)\subseteq \FGamma^d_{X'/S}(T)$. By the
construction of $Z'$ and the regularity of $\beta'$, it is clear that
$f'_*\alpha'=\beta'$. We let $\Lambda^{-1}(T)(\alpha,\beta')=\alpha'$ and it is
obvious that $\Lambda$ is a morphism since the construction is functorial. By
construction $\Lambda\circ\Lambda^{-1}$ is the identity and as
$\FGamma^d_{X'/S,\reg/g} \subseteq \FGamma^d_{X'/S,\reg/g'}$ it follows
that $\Lambda^{-1}\circ\Lambda$ is the identity as well.

\ref{PD:univhomeo-if-unramified}
Over reduced schemes all the involved images are reduced by
Theorem~\pref[IS:reduced]{T:supp-properties} and the support of the
push-forward coincides with the image. The arguments of \ref{PD:reg-inclusion}
and \ref{PD:cartesian-if-etale-or-immersion} then simplify and go through
without any hypotheses on $f$ and $g$.
\end{proof}

\begin{corollary}\label{C:pushfwd-diagram}
Let $\map{f}{X}{Y}$ and $\map{g}{Y'}{Y}$ be morphism of algebraic spaces,
separated over $S$. Assume that for every involved space $Z$, the functor
$\FGamma^d_{Z/S}$ is represented by a space which we denote by $\Gamma^d(Z/S)$.
\begin{enumerate}
\item If $g$ is unramified, then $\FGamma^d_{Y'/S,\reg/g}$ is represented by an
open subspace $U=\reg(g)$ of $\Gamma^d(Y'/S)$.
\item If $g$ is \etale{}, then we have a cartesian diagram
$$\xymatrix{
{\Gamma^d(X'/S)|_{{f'_{*}}^{-1}(U)}}\ar[d]^{f'_*}\ar[r]^-{g'_*}
  & {\Gamma^d(X/S)}\ar[d]^{f_*} \\
{\Gamma^d(Y'/S)|_U}\ar[r]^{g_*}
  & {\Gamma^d(Y/S).}\ar@{}[ul]|\square
}$$
\item If $g$ is unramified, the canonical morphism
$$\map{\Lambda}{\Gamma^d(X'/S)|_{{f'_{*}}^{-1}(U)}}
{\Gamma^d(Y'/S)|_U\times_{\Gamma^d(Y/S)}\Gamma^d(X/S)}$$
is a universal homeomorphism such that $\Lambda_\red$ is an isomorphism.
\end{enumerate}
\end{corollary}
\begin{proof}
Follows immediately from Propositions~\pref{P:reg-functors}
and~\pref{P:pushfwd-diagram}.
\end{proof}

\begin{corollary}\label{C:pushfwd-diagram-2}
Let $\map{f_i}{X_i}{Y}$, $i=1,2$ be morphism of algebraic spaces, separated
over $S$. Let $\map{\pi_i}{X_1\times_Y X_2}{X_i}$ be the projections. Assume
that for every involved space $Z$, the functor $\FGamma^d_{Z/S}$ is represented
by a space which we denote by $\Gamma^d(Z/S)$. Assume that $f_1$ and $f_2$ are
both \etale{} and let $U_i=\reg(f_i)$ and $U_{12}=\reg(f_1\circ
\pi_1)=\reg(f_2\circ \pi_2)$. Then
\begin{enumerate}
\item $U_{12}=((\pi_1)_*)^{-1}(U_1)\cap ((\pi_2)_*)^{-1}(U_2)$.
\item The diagram
$$\xymatrix{
{\Gamma^d(X_1\times_Y X_2/S)|_{U_{12}}}
    \ar[d]^{(\pi_1)_*}\ar[r]^-{(\pi_2)_*}
  & {\Gamma^d(X_2/S)|_{U_2}}\ar[d]^{(f_2)_*} \\
{\Gamma^d(X_1/S)|_{U_1}}\ar[r]^{(f_1)_*}
  & {\Gamma^d(Y/S)}\ar@{}[ul]|\square
}$$
is cartesian.
\end{enumerate}
\end{corollary}
\begin{proof}
It follows from \ref{PD:reg-inclusion} of Proposition~\pref{P:pushfwd-diagram}
that
$$((\pi_1)_*)^{-1}(U_1)\cap ((\pi_2)_*)^{-1}(U_2)\subseteq U_{12}$$
and the reverse inclusion is obvious. That the diagram is cartesian now follows
from Corollary~\pref{C:pushfwd-diagram}.
\end{proof}

\begin{remark}
The diagrams in Proposition~\pref{P:pushfwd-diagram} and
Corollary~\pref{C:pushfwd-diagram} are not always cartesian if $g$ is
unramified but not \etale{}. In fact, by Examples~\pref{Ex:push-fwd-examples}
there is a morphism $\map{f}{X}{Y}$ and a family $\alpha\in \FGamma^d_{X/S}(S)$
such that $\FCI(\alpha)=X$, $f(\FCI(\alpha))=Y$ and such that
$\FCI(f_*\alpha)\inj Y$ is not an isomorphism. If we let $Y'=\FCI(f_*\alpha)$
and $\beta'=f_*\alpha\in \FGamma^d_{Y'/S}(S)$, then we cannot lift
$(\alpha,\beta')$ to a family $\alpha'\in \FGamma^d_{X'/S}(S)$.
On the other hand, it is easily seen that Corollary~\pref{C:pushfwd-diagram-2}
remains valid if we replace \etale{} with unramified.
\end{remark}

\begin{remark}
Let $X$, $Y$, $U$, $f$ and $g$ as in Corollary~\pref{C:pushfwd-diagram} and let
$U'$ be the open subscheme of $\Gamma^d(X'/S)$ which represents
$\FGamma^d_{Y'/S,\reg/g'}$. Then ${{f'_{*}}^{-1}(U)\subseteq U'}$ by
Proposition~\pref[PD:reg-inclusion]{P:pushfwd-diagram}, i.e., the points of
$\Gamma^d(X'/S)|_{{f'_*}^{-1}(U)}$ are regular with respect to $g'$. On the
other hand, a point which is regular with respect to $g'$ need not be regular
with respect to $g$, i.e., the inclusion ${f'_{*}}^{-1}(U)\subseteq U'$ is
strict in general.
\end{remark}

\begin{proposition}\label{P:pushfwd-over-etale-and-reg-is-etale}
If $\map{f}{X/S}{Y/S}$ is an \etale{} (resp. \etale{} and surjective)
morphism of algebraic spaces separated over $S$, then
the push-forward
$\map{f_*}{\FGamma^d_{X/S,\reg/f}}{\FGamma^d_{Y/S}}$ is
representable and \etale{} (resp. \etale{} and surjective).
\end{proposition}
\begin{proof}
If $f$ is surjective then
$\map{f_*}{\FGamma^d_{X/S,\reg/f}}{\FGamma^d_{Y/S}}$ is topologically
surjective by Proposition~\pref[PRF:reg-subset-top-surj]{P:reg-functors}.

\textbf{I)} \emph{Reduction to $X\to S$ quasi-compact.}
Let $\{U_\beta\}$ be an open cover of $X$ such that $U_\beta$ is
quasi-compact and any set of $d$ points in $X$ over the same point in $S$
lies in some $U_\beta$. Then $\{\FGamma^d_{U_\beta,\reg/f|_{U_\beta}}\to
\FGamma^d_{X,\reg/f}\}$ is an open cover by
Proposition~\pref{P:cover-of-Gamma:AF}. Replacing $X$ with $U_\beta$ we can
thus assume that $X$ is quasi-compact.

\textbf{II)} \emph{Reduction to $X,Y$ and $S$ affine and $Y$ integral over
$S$.}
Let $T$ be an affine scheme and $T\to \FGamma^d_{Y/S}$ a morphism. Then it
factors as $T\to \Gamma^d(W/T)$ where $W\inj Y_T=Y\times_S T$ is a closed
subspace such that $W\to T$ is integral.
Let $Z=f_T^{-1}(W)$. Note that $f$ is separated and quasi-compact as $X\to
S$ is separated and quasi-compact. Hence $f$ is quasi-affine as well as
$Z\to W\to T$ which is the composition of two quasi-affine morphisms.  Thus
$\FGamma_{Z/T}$ and $\FGamma_{W/T}$ are both representable by
Theorem~\pref{T:Gamma-representable:AF}.
As $W\inj Y_T$ is a closed immersion it follows from
Proposition~\pref[PD:cartesian-if-etale-or-immersion]{P:pushfwd-diagram}
that we have a cartesian diagram
$$\xymatrix{
{\Gamma^d(Z/T)|_{\reg\left(f_T|_Z\right)}}\ar[d]^-{(f_T|_{Z})_*}\ar@{(->}[r]
& {\FGamma^d_{X_T/T,\reg/f_T}}\ar[d]^{(f_T)_*}\ar[r] &
  {\FGamma^d_{X/S,\reg/f}}\ar[d]^{f_*} \\
{\Gamma^d(W/T)}\ar@{(->}[r]
& {\FGamma^d_{Y_T/T}}\ar[r]\ar@{}[ul]|\square
& {\FGamma^d_{Y/S}.}\ar@{}[ul]|\square
}$$
This shows that $f_*$ is representable. To show that
$\map{f_*}{\FGamma^d_{X/S,\reg/f}}{\FGamma^d_{Y/S}}$ is \etale{} it is thus
enough to show that $\Gamma^d(Z/T)\to\Gamma^d(W/T)$ is \etale{} over the
open subset $\reg\left(f_T|_{Z}\right)$. Further, as $\Gamma^d(Z/T)$ is
covered by open affine subsets of the form $\Gamma^d(U/T)$ where
$U\subseteq Z$ is an affine open subset by
Proposition~\pref{P:cover-of-Gamma:AF}, we can assume that $Z/T$ is
affine. Replacing $X$, $Y$ and $S$ with $Z$, $W$ and $T$ we can then assume
that $X$ and $S$ are affine and $Y$ is integral over $S$.

\textbf{III)} \emph{Reduction to $X$ and $Y$ quasi-finite and finitely
presented over $S$.} Let $S=\Spec(A)$, $Y=\Spec(B)$ and $X=\Spec(C)$. We
can write $B$ as a filtered direct limit of finite and finitely presented
$A$-algebras $B_\lambda$. As $B\to C$ is of finite presentation, we can
find an $\mu$ and a $B_\mu$-algebra $C_\mu$ such that
$C=C_\mu\otimes_{B_\mu} B$. Let $C_\lambda=C_\mu\otimes_{B_\mu} B_\lambda$,
$X_\lambda=\Spec(C_\lambda)$ and $Y_\lambda=\Spec(B_\lambda)$ for every
$\lambda\geq\mu$.
As $\Gamma^d$ commutes with filtered direct limits, cf.
paragraph~\pref{X:Gamma^d_alg_filt_dir_lims}, we have that
$\Gamma^d_A(B)=\varinjlim_\lambda \Gamma^d_A(B_\lambda)$ and
$\Gamma^d_A(C)=\varinjlim_\lambda \Gamma^d_A(C_\lambda)$.

Let $U=\reg(f)\subseteq\Gamma^d(X/S)$ and let $u\in U$ be a point with
residue field $k$ and let $\alpha\in\FGamma_{X/S}(k)$ be the corresponding
family of cycles with image $Z\inj X_k$. Let $\beta=f_*\alpha$ and
$W=\FCI(\beta)$. As $\alpha$ is regular $Z\to W$ is an isomorphism. Now as
$W$ consists of a finite number of points each with a residue field of
finite separable degree over $k$, it is easily seen that there is a
$\lambda\geq\mu$ such that $(Y\times_S k)|_W
\to Y_\lambda\times_S k$ is universally injective. Thus the push-forward of
$\alpha$ along $\map{\psi_\lambda}{X}{X_\lambda}$ is quasi-regular with
respect to $f_\lambda$ and thus regular as $f_\lambda$ is \etale{}.
Corollary~\pref{C:pushfwd-diagram} gives the cartesian diagram
$$\xymatrix{
\Gamma^d(X/S)|_{\psi_\lambda^{-1}(V)}\ar[r]^-{f_*}\ar[d]^{{\psi_\lambda}_*}
 & \Gamma^d(Y/S)\ar[d] \\
\Gamma^d(X_\lambda)|_{V}\ar[r]^{(f_\lambda)_*}
  & \Gamma^d(Y_\lambda/S)\ar@{}[ul]|\square
}$$
where $V=\reg(f_\lambda)$ and $u\in\psi_\lambda^{-1}(V)$ as
$(\psi_\lambda)_*\alpha$ is regular.

Replacing $X$ and $Y$ with $X_\lambda$ and $Y_\lambda$ we can thus assume
that $X$ and $Y$ are of finite presentation over $S$. Further as $f$ is
quasi-finite and of finite presentation and $Y\to S$ is finite and of
finite presentation it follows that $X\to S$ is quasi-finite and of finite
presentation.  Proposition~\pref{P:Gamma^d_finiteness} then shows that
$\Gamma^d(X/S)$ and $\Gamma^d(Y/S)$ are of finite presentation over
$S$. Thus $\map{f_*}{\Gamma^d(X/S)}{\Gamma^d(Y/S)}$ is also of finite
presentation.

\textbf{IV)} \emph{Reduction to $S$ strictly local.}
Let $\alpha\in\Gamma^d(X/S)$ and let $\beta=f_*(\alpha)$ and $s\in S$ be
its images. Let $S'\to S$ be a flat morphism such that $s$ is in its image.
Then, as $f_*$ is of finite presentation, $f_*$ is \etale{} at a point
$\alpha\in\Gamma^d(X/S)$ if the morphism
$\Gamma^d(X'/S')\to\Gamma^d(Y'/S')$ is
\etale{} at a point $\alpha'\in\Gamma^d(X'/S')$ above
$\alpha$~\cite[Prop.~17.7.1]{egaIV}. We take $S'$ as the strict
henselization of $\sO_{S,s}$. As $\FGamma^d_{X/S,\reg/f}$ is an open
subfunctor of $\FGamma^d_{X/S}$ we have that $\reg(f)\times_S S'=\reg(f')$.
We can thus replace $X$, $Y$ and $S$ with $X'$, $Y'$ and $S'$ and assume
that $S$ is strictly local.

\textbf{V)} \emph{Conclusion}
We have now reduced the proposition to the following situation: $S$ is
strictly local, $X\to S$ is quasi-finite and finitely presented and $Y\to
S$ is finite and finitely presented. The support of
$\alpha\in\Gamma^d(X/S)$ consists of a finite number of points
$x_1,x_2,\dots,x_m\in X$ lying above the closed point $s\in S$. As $X\to S$
is quasi-finite and $S$ is henselian it follows that
$X=\bigl(\coprod_{i=1}^m X_i\bigr) \amalg X'$ where $X_i$ are strictly
local schemes, finite over $S$, such that $x_i\in X_i$. Then
$\alpha\in\Gamma^d\bigl(\coprod_{i=1}^m X_i\bigr)\inj\Gamma^d(X/S)$ and we
can thus assume that $X=\coprod_{i=1}^m X_i$ is finite over $S$.

As $S$ is strictly local and $Y\to S$ is finite it follows that
$Y=\coprod_{j=1}^n Y_j$ is a finite disjoint union of strictly local
schemes.  For every $i=1,2,\dots,m$ there is a $j(i)$ such that $f(X_i)\inj
Y_{j(i)}$ and $\map{f|_{X_i}}{X_i}{Y_{j(i)}}$ is an isomorphism as $f$ is
\etale{}.  We have further by Proposition~\pref{P:Gamma^d_of_disj_union}
that
$$\Gamma^d(X/S) = \coprod_{\sum_i d_i=d}\prod_{i=1}^m \Gamma^{d_i}(X_i),
\quad \Gamma^d(Y/S) =
\coprod_{\sum_j e_j=d}\prod_{j=1}^n \Gamma^{e_j}(Y_j).$$
It is obvious that the regular subset $U\subseteq \Gamma^d(X/S)$ is given
by the connected components with $d_1,d_2,\dots,d_m$ such that for every
$j=1,2,\dots,n$ there is at most one $i$ with $d_i>0$ such that
$j(i)=j$. As
$$\prod_{i=1}^m \Gamma^{d_i}(X_i)\to
\prod_{i=1}^m \Gamma^{d_i}\left(Y_{j(i)}\right)$$
is an isomorphism this completes the demonstration.
\end{proof}

\begin{corollary}\label{C:exact-sequence}
Let $X/S$ be a separated algebraic space and
$\{\map{f_\alpha}{U_\alpha}{X}\}_\alpha$ an \etale{} separated cover. Assume
that for every involved space $Z$, the functor $\FGamma^d_{Z/S}$ is represented
by a space which we denote by $\Gamma^d(Z/S)$. Then
\begin{equation}\label{E:etale-equiv-rel}
\let\objectstyle=\displaystyle
\xymatrix{%
\coprod_{\alpha,\beta}\Gamma^d(U_\alpha\times_{X} U_\beta/S)|_\reg
   \ar@<.5ex>[r]\ar@<-.5ex>[r]
& \coprod_{\alpha}\Gamma^d(U_\alpha/S)|_\reg\ar[r]
& \Gamma^d(X/S)
}
\end{equation}
is an \etale{} equivalence relation. Here $\reg$ denotes the regular locus
with respect to the push-forward to $X$.
\end{corollary}
\begin{proof}
This follows from Corollary~\pref{C:pushfwd-diagram-2} and
Proposition~\pref{P:pushfwd-over-etale-and-reg-is-etale}.
\end{proof}

\end{subsection}


\begin{subsection}{Representability of $\FGamma^d_{X/S}$ by an algebraic space}
\label{SS:Rep_Gamma}

In this subsection, it will be shown that for \emph{any} algebraic
space $X$ separated over $S$, the functor $\FGamma^d_{X/S}$ is represented by
an algebraic space, separated over $S$.

\begin{theorem}\label{T:Gamma-representable:alg-space}
Let $S$ be an algebraic space and $X/S$ a separated algebraic space. Then the
functor $\FGamma^d_{X/S}$ is represented by a separated algebraic space
$\Gamma^d(X/S)$.
\end{theorem}
\begin{proof}
Let $\map{f}{X'}{X}$ be an \etale{} cover such that $X'$ is a disjoint
union of affine schemes. Then $X'$ is an \AF{}-scheme and
$\FGamma^d_{X'/S}$ is represented by the scheme $\Gamma^d(X'/S)$, cf.\
Theorem~\pref{T:Gamma-representable:AF}. By
Propositions~\pref{P:Gamma-is-etale-sheaf}
and~\pref{P:pushfwd-over-etale-and-reg-is-etale}, the functor
$\FGamma^d_{X/S}$ is a sheaf in the \etale{} topology and the push-forward
$\map{f_*}{\Gamma^d(X'/S)|_{\reg(f)}}{\FGamma^d_{X/S}}$ is an \etale{}
presentation.

To show that $\FGamma^d_{X/S}$ is a separated algebraic space, it is thus
sufficient to show that the diagonal is represented by closed immersions.
Let $T$ be an $S$-scheme and $\alpha,\beta\in\FGamma^d_{X/S}(T)$. Let
$Z_\alpha,Z_\beta\inj X\times_S T$ be the images of $\alpha$ and $\beta$.
Let $Z_0=Z_\alpha\cap Z_\beta = Z_\alpha\times_{X_T} Z_\beta$. We then let
$T_0=\alpha^{-1}(\Gamma^d(Z_0/S))\cap\beta^{-1}(\Gamma^d(Z_0/S))$ where we have
considered $\alpha$ and $\beta$ as morphisms $T\to\Gamma^d(Z_\alpha/T)$ and
$T\to\Gamma^d(Z_\beta/T)$ respectively. Then $T_0\inj T$ is a closed subscheme
and
\begin{align*}
(\alpha,\beta)^*\Delta_{\FGamma^d_{X/S}/S}
 &= \FGamma^d_{X/S} \times_{\FGamma^d_{X/S}\times_S\FGamma^d_{X/S}} T \\
 &= \FGamma^d(Z_0/T) \times_{\FGamma^d(Z_0/T)\times_S\FGamma^d(Z_0/T)} T_0 \\
 &= (\alpha|_{T_0},\beta|_{T_0})^*\Delta_{\Gamma^d(Z_0/T)/T}
\end{align*}
which is a closed subscheme of $T_0$ as $\Gamma^d(Z_0/T)\to T$ is affine.
\end{proof}


\begin{proposition}\label{P:Gamma-neighborhoods}
Let $X/S$ be a separated algebraic space. Let $s\in S$ and let
$\alpha\in\Gamma^d(X/S)$ be a point over $s\in S$. There is then a finite
number of points $x_1,x_2,\dots,x_n\in X$ with $n\leq d$ such that the
following condition holds:
\begin{enumerate}\renewcommand{\labelenumi}{{\upshape{(*)}}}
\item Choose an \etale{} neighborhood $S'\to S$ of $s$ and
\etale{} neighborhoods $\{U_i\to X\}$ of $\{x_i\}$ such that the $U_i$'s are
algebraic $S'$-spaces. There is then an open subset $V$ of
$\Gamma^d\bigl(\coprod_{i=1}^n U_i/S'\bigr)$ such that $V\to \Gamma^d(X/S)$ is
an \etale{} neighborhood of $\alpha$.
\end{enumerate}
Furthermore, if we choose the $U_i$'s such that there is a point above $x_i$
with trivial residue field extension, then there is a point in $V$ above
$\alpha$ with trivial residue field extension.

In particular, $\Gamma^d(X/S)$ has an \etale{} covering of the form $\coprod_i
\Gamma^d(X_i/S_i)|_{V_i}$ where $S_i$ and $X_i$ are affine and $S_i\to S$ and
$X_i\to X$ \etale{}.
\end{proposition}
\begin{proof}
The point $\alpha$ corresponds to a family $\Spec(k(\alpha))\to\Gamma^d(X/S)$
where $k(\alpha)$ is the residue field. Let $Z\inj X\times_S \Spec(k(\alpha))$
be the image of this family. Then $Z$ is reduced and consists of a finite
number of points $z_1,z_2,\dots,z_m$ such that $m\leq d$. Let
$W=\{x_1,x_2,\dots,x_n\}$ be the projection of $Z$ on $X$. Then $\alpha$ lies
in the closed subset $\Gamma^d(W/S)\inj\Gamma^d(X/S)$.

If $\map{f}{U}{X}$ is an \etale{} neighborhood of $W$ then it is obvious that
there is a lifting of $\alpha$ to $V=\Gamma^d(U/S)|_{\reg(f)}$. Furthermore,
if $f$ has trivial residue field extensions over $W$, then we can choose a
lifting with the residue field $k(\alpha)$. That $V\to \Gamma^d(X/S)$ is
\etale{} is Proposition~\pref{P:pushfwd-over-etale-and-reg-is-etale}.
\end{proof}

\end{subsection}

\end{section}


\begin{section}{Further properties of $\Gamma^d(X/S)$}\label{S:further-props}


\begin{subsection}{Addition of cycles and non-degenerate families}
\label{SS:addition-of-cycles}

In paragraphs\ \pref{X:univ-mult-of-laws}
and~\pref{X:univ-mult-and-comp:alg-case} we defined the universal
multiplication of laws
$\map{\rho_{d,e}} {\Gamma^{d+e}_A(B)}{\Gamma^d_A(B)\otimes_A\Gamma^e_A(B)}$. We
will give a corresponding morphism
$\Gamma^d(X/S)\times_S\Gamma^e(X/S)\to\Gamma^{d+e}(X/S)$ for arbitrary $X/S$.

\begin{definition-proposition}\label{D:addition-of-cycles}
Let $X/S$ be a separated algebraic space and let $d,e$ be positive integers.
Then there exists a morphism
$$\map{+}{\Gamma^d(X/S)\times_S\Gamma^e(X/S)}{\Gamma^{d+e}(X/S)}$$
which on points is addition of cycles. When $X/S$ is affine, this morphism
corresponds to the homomorphism $\rho_{d,e}$. The operation $+$ makes the space
$\Gamma(X/S)=\coprod_{d\geq 0}\Gamma^d(X/S)$ into a graded commutative
monoid.
\end{definition-proposition}
\begin{proof}
The morphism $+$ is the composition of the open and closed immersion
$\Gamma^d(X/S)\times\Gamma^e(X/S)\inj \Gamma^{d+e}(X\amalg X/S)$ of
Proposition~\pref{P:Gamma^d_of_disj_union:general} and the push-forward along
$X\amalg X\to X$. It is clear that this is an associative and commutative
operation as push-forward is functorial.
When $X/S$ is affine, it is clear from~\pref{X:univ-mult-of-laws} that the
addition of cycles corresponds to the homomorphism~$\rho_{d,e}$.
\end{proof}

\begin{proposition}
Let $X/S$ be a separated algebraic space and $T$ an $S$-scheme. Let
$\alpha\in\FGamma^d_{X/S}(T)$ and $\beta\in\FGamma^e_{X/S}(T)$.
\begin{enumerate}
\item If $T$ is connected and $\FCI(\alpha)=\coprod_{i=1}^n Z_i$ then there are
integers $d_i\geq 1$ and families of cycles
$\alpha_i\in\FGamma^{d_i}_{Z_i/S}(T)$ such that $d=d_1+d_2+\dots+d_n$ and
$\alpha=\alpha_1+\alpha_2+\dots+\alpha_n$.
\item $\FCS(\alpha+\beta)=\FCS(\alpha)\cup\FCS(\beta)$.
\item Let $\map{f}{X}{Y}$ be a morphism of separated algebraic spaces. Then
$f_*(\alpha+\beta)=f_*\alpha+f_*\beta$.
\end{enumerate}
\end{proposition}
\begin{proof}
(i) is obvious from Proposition~\pref{P:Gamma^d_of_disj_union:general}. (ii)
follows from Proposition~\pref[P:push-fwd-and-supp:red]{P:push-fwd-and-supp}.
(iii) follows easily from the definitions and the functoriality of the
push-forward.
\end{proof}

\begin{proposition}\label{P:addition-etale-if-disjoint}
The morphism $\Gamma^d(X/S)\times_S\Gamma^e(X/S)\to\Gamma^{d+e}(X/S)$ is
\etale{} over the open subset $U\subseteq \Gamma^d(X/S)\times_S\Gamma^e(X/S)$
where $(\alpha,\beta)\in U$ if $\FCS(\alpha)$ and $\FCS(\beta)$ are
disjoint.
\end{proposition}
\begin{proof}
The morphism $X\amalg X\to X$ is \etale{}. By
Propositions~\pref{P:Gamma^d_of_disj_union:general}
and\ \pref{P:pushfwd-over-etale-and-reg-is-etale} we have that
$\Gamma^d(X/S)\times_S\Gamma^e(X/S)\to\Gamma^{d+e}(X/S)$ is \etale{} at
$(\alpha,\beta)$ if $\alpha\amalg\beta$ is regular with respect to $X\amalg
X\to X$. This is fulfilled if and only if $\FCS(\alpha)$ and $\FCS(\beta)$ are
disjoint.
\end{proof}

\begin{notation}
We let $(X/S)^d$ denote the fiber product $X\times_S X\times_S \dots \times_S
X$ of $d$ copies of $X$ over $S$.
\end{notation}

\begin{proposition}\label{P:Gamma-is-top-quot}
Let $X/S$ be a separated algebraic space. The symmetric group on $d$
letters $\SG{d}$ acts on $(X/S)^d$ by permutation of factors. We equip
$\Gamma^d(X/S)$ with the trivial $\SG{d}$-action. Then:
\begin{enumerate}
\item There is a canonical $\SG{d}$-equivariant morphism
$\map{\Psi_X}{(X/S)^d}{\Gamma^d(X/S)}$.
\label{PSI:equivariant}
\item $\Psi_X$ is integral and universally open. Its fibers are the
orbits of $(X/S)^d$ and this also holds after base change.
\label{PSI:utq+int+uo}
\item $\Psi_X$ is \etale{} outside the diagonals of $(X/S)^d$.
\label{PSI:etale}
\item If $\map{f}{X}{Y}$ is a morphism of separated algebraic spaces we have
a commutative diagram
$$\xymatrix{
(X/S)^d\ar[d]^{\Psi_X}\ar[r]^{f^d} & (Y/S)^d\ar[d]^{\Psi_Y} \\
\Gamma^d(X/S)\ar[r]^{f_*} & \Gamma^d(Y/S).\ar@{}[ul]|\circ
}$$
If $f$ is unramified (resp. \etale{}) and $U=\reg(f)$ then the canonical
morphism
$$\map{\Lambda}{(X/S)^d|_{\Psi_X^{-1}(U)}}
{\Gamma^d(X/S)|_U \times_{\Gamma^d(Y/S)} (Y/S)^d}$$
is a universal homeomorphism (resp. an isomorphism).
\label{PSI:cartesian-if-etale}
\end{enumerate}
\end{proposition}
\begin{proof}
\ref{PSI:equivariant} As
$\Hom_S\bigl(T,(X/S)^d\bigr)=\Hom_S(T,X)^d=\FGamma^1_{X/S}(T)^d$ by
Remark\ \pref{R:Gamma^1} we obtain by addition of cycles the morphism
$\map{\Psi_X}{(X/S)^d}{\Gamma^d(X/S)}$ and this is clearly an
$\SG{d}$-equivariant morphism as addition of cycles is commutative.

\ref{PSI:etale} Follows immediately from
Proposition~\pref{P:addition-etale-if-disjoint}.

\ref{PSI:cartesian-if-etale} Follows from the definition of $\Psi$ and
Corollary~\pref{C:pushfwd-diagram} since
$$\xymatrix{
(X/S)^d\ar[r]^{f^d}\ar[d] & (Y/S)^d\ar[d] \\
\Gamma^d(\coprod_{i=1}^d X)\ar[r]
& \Gamma^d(\coprod_{i=1}^d Y)\ar@{}[ul]|\square
}$$
is cartesian.

\ref{PSI:utq+int+uo} We first show that the fibers of $\Psi$ are the
$\SG{d}$-orbits and that this holds after any base change. Let
$\map{f}{\Spec(k)}{\Gamma^d(X/S)}$ be a morphism. Then $f$ factors through
$\Gamma^d(Z/k)\to\Gamma^d(X/S)$ where $Z\inj X\times_S \Spec(k)$ is a closed
subspace integral over $k$.

As $\Gamma^d$ commutes with base change, we can replace $S$ with
$\Spec(k)$. Furthermore, using the unramified part of
\ref{PSI:cartesian-if-etale}, we can replace $X$ with $Z$. We can thus assume
that $S=\Spec(k)$ and that $X=Z=\Spec(B)$. Then
$(X/k)^d=\Spec\bigl(\T^d_k(B)\bigr)$ and
$\Gamma^d(X/k)=\Spec\bigl(\TS^d_k(B)\bigr)=\Sym^d(X/k)$. As the fibers of
$(X/k)^d\to \Sym^d(X/k)$ are the $\SG{d}$-orbits it follows that the same holds
for $\Psi$.

If $U\inj (X/S)^d$ is an open (resp. closed subset) then
$\Psi^{-1}\bigl(\Psi(U)\bigr)=\bigcup_{\sigma\in \SG{d}} \sigma U$. As this
also holds after any base change $T\to \Gamma^d(X/S)$ it follows that $\Psi$ is
universally closed and universally open.

We will now show that $\Psi_X$ is affine. As $\Psi_X$ is universally closed it
then follows that $\Psi_X$ is integral by~\cite[Prop.~18.12.8]{egaIV}. As
affineness is local in the \etale{} topology we can assume that $S$ is
affine. Let $\map{f}{X'}{X}$ be an \etale{} covering such that $X'$ is a
disjoint union of affine schemes and in particular an \AF{}-scheme. By
Proposition~\pref{P:pushfwd-over-etale-and-reg-is-etale} the push-forward
morphism $\map{f_*}{\Gamma^d(X'/S)|_{\reg(f)}}{\Gamma^d(X/S)}$ is an \etale
cover. Using~\ref{PSI:cartesian-if-etale} and replacing $X$ with $X'$ we can
thus assume that $X$ is \AF{}. Proposition~\pref{P:cover-of-Gamma:AF} then
shows that $\Gamma^d(X/S)$ is covered by open subsets $\Gamma^d(U/S)$ where $U$
is affine. Finally $\Psi_U$ is affine as $(U/S)^d$ is affine.
\end{proof}

\begin{definition}\label{D:non-degenerate}
Let $X/S$ be a separated algebraic space, $T$ an $S$-space and
$\alpha\in\FGamma^d_{X/S}(T)$ a family of cycles. Let $t\in T$ be a point and
let $\overline{k}$ be an algebraic closure of its residue field $k$. We say
that $\alpha$ is \emph{non-degenerated} in a point $t\in T$ if the support of
the cycle $\alpha_t\times_k \overline{k}$ consists of $d$ distinct points. Here
$\alpha_t\times_k \overline{k}$ denotes the family given by the composition of
$\Spec(\overline{k})\to \Spec(k)\to T$ and $\alpha$.
The \emph{non-degeneracy locus} is the set of points $t\in T$ such that
$\alpha$ is non-degenerate in $t$.
\end{definition}

\begin{definition}
We let $\Gamma^d(X/S)_{\nondeg}\subseteq \Gamma^d(X/S)$ denote the subset of
non-degenerate families.
\end{definition}

\begin{proposition}\label{P:etale-over-non-deg}
The subset $\Gamma^d(X/S)_{\nondeg}\subseteq \Gamma^d(X/S)$ is open. The
morphism $\map{\Psi_X}{(X/S)^d}{\Gamma^d(X/S)}$ is \etale{} of rank $d!$ over
$\Gamma^d(X/S)_{\nondeg}$ and the addition morphism
$\map{+}{\Gamma^d(X/S)\times_S \Gamma^e(X/S)}{\Gamma^{d+e}(X/S)}$ is \etale{}
of rank $((d,e))$ over $\Gamma^{d+e}(X/S)_{\nondeg}$.
\end{proposition}
\begin{proof}
Let $U$ be the complement of the diagonals of $(X/S)^d$, which is an open
subset. Then $\Gamma^d(X/S)_{\nondeg}=\Psi_X(U)$ which is an open subset as
$\Psi_X$ is open. The last two statements follow from
Proposition~\pref{P:addition-etale-if-disjoint}.
\end{proof}

\end{subsection}


\begin{subsection}{The $\Sym^d\to\Gamma^d$ morphism}
\newcommand\SymGamma{\mathrm{SG}}

\begin{definition}[{\cite{kollar_quotients,rydh_finite_quotients}}]
If $G$ is a group and $\map{f}{X}{Y}$ a $G$-equivariant morphism, then we say
that $f$ is fixed-point reflecting, or fpr, at $x\in X$ if the stabilizer of
$x$ coincides with the stabilizer of $f(x)$. The subset of $X$ where $G$ is
fixed-point reflecting is $G$-stable and denoted $\fpr(f)$.
\end{definition}

\begin{remark}
Let $X/S$ be a separated algebraic space. There is then a uniform geometric and
categorical quotient $\Sym^d(X/S):=(X/S)^d/\SG{d}$,
cf.~\cite{rydh_finite_quotients}. Furthermore we have that
$\map{q}{(X/S)^d}{\Sym^d(X/S)}$ is integral, universally open and a topological
quotient, i.e., it satisfies (ii) of Proposition~\pref{P:Gamma-is-top-quot}.
Moreover (iii) and the \etale{} part of (iv) also holds for $q$ instead of
$\Psi$ if we replace $\reg(f)$ with $\fpr(f)$,
cf.~\cite{rydh_finite_quotients}.

As $\map{\Psi_X}{(X/S)^d}{\Gamma^d(X/S)}$ is
$\SG{d}$-equivariant and $\Sym^d(X/S)$ is a \emph{categorical} quotient, we
obtain a canonical morphism
$\map{\SymGamma_X}{\Sym^d(X/S)}{\Gamma^d(X/S)}$ such that
$\Psi_X=\SymGamma_X\circ q$.
\end{remark}

\begin{lemma}\label{L:qreg=fpr}
Let $\map{f}{X}{Y}$ be a morphism of algebraic spaces and
let $\alpha\in\Gamma^d(X/S)$ be a point. Then $\alpha$ is quasi-regular with
respect to $f$ if and only if $f^d$ is fixed-point reflecting at
$\Psi_X^{-1}(\alpha)$ with respect to the action of $\SG{d}$.
\end{lemma}
\begin{proof}
\newcommand{\kab}{\overline{k(\alpha)}}
Let $k$ be the algebraic closure of the residue field $k(\alpha)$.
The supports of $\alpha$ and $f_*\alpha$ are finite disjoint unions of points.
Thus $\map{\alpha}{\Spec(k)}{\Gamma^d(X/S)}$ and
$\map{f_*\alpha}{\Spec(k)}{\Gamma^d(Y/S)}$ factors as
$$\Spec(k)\to\prod_{i=1}^n \Gamma^{d_i}(x_i/k)\to\Gamma^d(X/S)$$
and
$$\Spec(k)\to\prod_{j=1}^m \Gamma^{e_i}(y_j/k)\to\Gamma^d(Y/S)$$
where $x_i$ and $y_j$ are points of $X\times_S \Spec(k)$ and $Y\times_S
\Spec(k)$ respectively and $k(x_i)=k(y_j)=k$. Every point of $(X/S)^d$ (resp.
$(Y/S)^d$) above $\alpha$ (resp. $f_*\alpha$) is thus such that, after a
permutation, the first $d_1$ (resp. $e_1$) projections agree, the next $d_2$
(resp. $e_2$) projections agree, etc, but no other two projections are equal.
Thus the stabilizers of the points of $\Psi_X^{-1}(\alpha)$ (resp.
$\Psi_Y^{-1}(f_*\alpha)$) are $\SG{d_1}\times\SG{d_2}\times\dots\times\SG{d_n}$
(resp.  $\SG{e_1}\times\SG{e_2}\times\dots\times\SG{e_m}$). Equality holds if
and only if $f$ is quasi-regular.
%
\end{proof}

\begin{proposition}\label{P:SG-cartesian-if-etale}
Let $\map{f}{X}{Y}$ be an \etale{} morphism of algebraic spaces. Then
$\Psi_X^{-1}\bigl(\reg(f)\bigr)=\fpr(f^d)$, and we have
a cartesian diagram
$$\xymatrix{
{(X/S)^d|_{\fpr(f^d)}}\ar[r]^-q\ar[d]^{f^d}
& {\Sym^d(X/S)|_{\fpr(f^d)}}
\ar[r]^-{\SymGamma_X}\ar[d]^{f^d/\SG{d}}
& {\Gamma^d(X/S)|_{\reg(f)}}\ar[d]^{f_*} \\
{(Y/S)^d}\ar[r]^q
& {\Sym^d(Y/S)}\ar[r]^-{\SymGamma_{Y}}\ar@{}[ul]|\square
& {\Gamma^d(Y/S)}\ar@{}[ul]|\square
}$$
In particular $f^d/\SG{d}$ is \etale{} over the open subset
$q\bigl(\fpr(f^d)\bigr)=\SymGamma_X^{-1}\bigl(\reg(f)\bigr)$.
\end{proposition}
\begin{proof}
As $f$ is unramified $\reg(f)=\qreg(f)$ by
Proposition~\pref{P:unramified:quasi-regular=regular}, the first statement
follows from Lemma~\pref{L:qreg=fpr}. The outer square is cartesian by
Proposition~\pref[PSI:cartesian-if-etale]{P:Gamma-is-top-quot} and as $q$ is a
uniform quotient the formation of the quotient commutes with \etale{} base
change which shows that the right square is cartesian. It follows that the left
square is cartesian too.
\end{proof}

\begin{corollary}\label{C:Sym-Gamma}
Let $X/S$ be a separated algebraic space. The canonical
morphism $\map{\SymGamma_X}{\Sym^d(X/S)}{\Gamma^d(X/S)}$ is a universal
homeomorphism with trivial residue field extensions. If $S$ has pure
characteristic zero or $X/S$ is flat, then $\SymGamma_X$ is an isomorphism.
\end{corollary}
\begin{proof}
\newcommand\abar{\overline{a}}
\newcommand\bbar{\overline{b}}
Using Proposition~\pref{P:SG-cartesian-if-etale} and the covering in
Proposition~\pref{P:Gamma-neighborhoods} we can assume that $X=\Spec(B)$ and
$S=\Spec(A)$ are affine. Then $(X/S)^d=\Spec\bigl(\T^d_A(B)\bigr)$,
$\Gamma^d(X/S)=\Spec\bigl(\Gamma^d_A(B)\bigr)$ and
$\Sym^d(X/S)=\Spec\bigl(\TS^d_A(B)\bigr)$ are all affine.
As $\Gamma^d_A(B)\to\TS^d_A(B)\inj \T^d_A(B)$ is integral by
Proposition~\pref[PSI:utq+int+uo]{P:Gamma-is-top-quot}, we have that
$\map{\SymGamma_X}{\Spec\bigl(\TS^d_A(B)\bigr)}
{\Spec\bigl(\Gamma^d_A(B)\bigr)}$ is integral.

The geometric fibers of both $\Psi_X$ and $\map{q}{(X/S)^d}{\Sym^d(X/S)}$
are the geometric orbits of $(X/S)^d$. Thus $\SymGamma_X$ is universally
bijective and hence a universal homeomorphism.
That $\SymGamma_X$ is an isomorphism when $S$ is purely of characteristic zero
or $X/S$ is flat follows from paragraph~\pref{X:Gamma-TS} as
$X$ and $S$ are affine.

Let $a\in\Sym^d(X/S)$ be any point, $b=\SymGamma_X(a)\in\Gamma^d(X/S)$ and $s$
its image in $S$. We have a commutative diagram
$$\xymatrix@C+1cm{
{\Sym^d\bigl(X_s/k(s)\bigr)}\ar[r]^{\SymGamma_{X_s}}_{\cong}\ar[d]
  & {\Gamma^d\bigl(X_s/k(s)\bigr)}\ar[d]^{\cong}\\
{\Sym^d(X/S)\times_S k(s)}\ar[r]^{\SymGamma_X\times_S \id{k(s)}}
  & {\Gamma^d(X/S)\times_S k(s)}\ar@{}[ul]|\circ
}$$
which gives a commutative diagram of residue fields
$$\xymatrix{
{k(\abar)}\ar@{}[dr]|\circ & {k(\bbar)}\ar@{(x->}[l]_{\cong}\\
{k(a)}\ar@{(x->}[u] & {k(b)}\ar@{(x->}[l]\ar@{(x->}[u]_{\cong}.
}$$
and thus $k(a)=k(b)$.
\end{proof}

\begin{proposition}\label{P:Sym-Gamma:nondeg}
Let $X/S$ be a separated algebraic space. The canonical
morphism $\map{\SymGamma_X}{\Sym^d(X/S)}{\Gamma^d(X/S)}$ is an isomorphism
over $\Gamma^d(X/S)_{\nondeg}$.
\end{proposition}
\begin{proof}
Let $U$ be the complement of the diagonals in $(X/S)^d$. Then
$\Psi_X(U)=\Gamma^d(X/S)_{\nondeg}$ and $\SG{d}$ acts freely on $U$. By
Proposition~\pref{P:etale-over-non-deg} the morphism $\Psi_X$ is \etale{} of
rank $d!$ over $\Gamma^d(X/S)_{\nondeg}$. It is further well-known that
$\map{q}{(X/S)^d}{\Sym^d(X/S)}$ is \etale{} of rank $d!$ over $q(U)$. In fact,
$\Sym^d(X/S)|_{q(U)}$ is the quotient sheaf in the \etale{} topology of the
\etale{} equivalence relation $\groupoidIL{\SG{d}\times U}{U}{}{}$.
\end{proof}

\end{subsection}


\begin{subsection}{Properties of $\Gamma^d(X/S)$ and the push-forward}

\begin{proposition}\label{P:Gamma^d-properties}
Let $S$ be an algebraic space and $X$ an algebraic space separated over $S$.
Consider for a morphism of algebraic spaces the property of being
\begin{enumerate}
\item quasi-compact
\item finite type
\item finite presentation
\item locally of finite type
\item locally of finite presentation
\item flat
\end{enumerate}
If $X\to S$ has one of these properties then so does $\Gamma^d(X/S)\to S$.
\end{proposition}
\begin{proof}
If $X\to S$ is quasi-compact then $(X/S)^d\to S$ is quasi-compact. As there is
a surjective morphism $\map{\Psi_X}{(X/S)^d}{\Gamma^d(X/S)}$ it follows that
$\Gamma^d(X/S)$ is quasi-compact over $S$. This shows (i).  (ii) and (iii)
follows from (i), (iv) and (v) as $\Gamma^d(X/S)$ is separated. It is thus
enough to show (iv), (v) and (vi).

As the question is local over $S$ we can assume that $S$ is affine. 
By Proposition~\pref{P:Gamma-neighborhoods} any point of $\Gamma^d(X/S)$
has an \etale{} neighborhood $V$ such that $V$ is an open subset of
$\Gamma^d(U/S)$ where $U$ is an affine scheme and $U\to X$ \etale{}. If
$V\to S$ is locally of finite type (resp. locally of finite presentation,
resp. flat) for any such neighborhood $V$ then it follows by
\cite[Lem.~17.7.5]{egaIV} and \cite[Cor.~2.2.11 (iv)]{egaIV} that
$\Gamma^d(X/S)$ is locally of finite type (resp. locally of finite
presentation, resp. flat) over $S$. Replacing $X$ with $U$ we can thus assume
that $X$ is affine. The proposition now follows from
Proposition~\pref{P:Gamma^d_finiteness} and
paragraph~\pref{X:Gamma^d_flat/free}.
\end{proof}

\begin{corollary}\label{C:Gamma^d_for_flat+reduced}
Let $S$ and $X$ be algebraic spaces. If $\map{f}{X}{S}$ is flat with geometric
reduced fibers then $\Gamma^d(X/S)\to S$ is flat with geometric reduced fibers.
In particular, if in addition $S$ is reduced then $\Gamma^d(X/S)$ is reduced.
\end{corollary}
\begin{proof}
Proposition~\pref{P:Gamma^d-properties} shows that $\Gamma^d(X/S)\to S$ is
flat. It is thus enough to show that $\Gamma^d(X_k/k)$ reduced for any
algebraic closed field $k$ and morphism $\Spec(k)\to S$. As $X_k$ is reduced by
hypothesis and hence also $(X_k/k)^d$ it follows that $\Sym^d(X_k/k)$ is
reduced and $\Gamma^d(X_k/k)=\Sym^d(X_k/k)$ by Corollary~\pref{C:Sym-Gamma}.
The last statement follows by~\cite[Prop.~5.17]{picavet_decent_rings}.
\end{proof}

\begin{proposition}\label{P:Gamma-of-smooth}
Let $S$ and $X$ be algebraic spaces. If $\map{f}{X}{S}$ is smooth of relative
dimension $0$ (resp. $1$, resp. at most $1$) then $\Gamma^d(X/S)\to S$ is
smooth of relative dimension $0$ (resp. $d$, resp. at most $d$).
\end{proposition}
\begin{proof}
As $\Gamma^d(X/S)\to S$ is flat and locally of finite presentation by
Proposition~\pref{P:Gamma^d-properties}, it is enough to show that its
geometric fibers are regular~\cite[Thm.~17.5.1]{egaIV}. Thus we can assume that
$S=\Spec(k)$ where $k$ is algebraically closed.
Let $y\in\Gamma^d(X/k)$. Then by
Proposition~\pref{P:formal-local-rings-of-Gamma}, the formal local ring
$\widehat{\sO}_{\Gamma^d(X/k),y}$ is the completion at a point of the scheme
$\prod_{i=1}^n\Gamma^{d_i}(\widehat{X}_{x_i}/k)$ where $x_1,x_2,\dots,x_n$ are
points of $X$ and $d=d_1+d_2+\dots+d_n$. If $f$ has relative dimension $0$ at
$x_i$ then $\sO_{X,x_i}=k$ and if $f$ has relative dimension $1$ at $x_i$ then
$\widehat{\sO}_{X,x_i}=k[[t]]$, cf.~\cite[Prop.~17.5.3]{egaIV}. The proposition
now easily follows if we can show that
$\Gamma^e\bigl(\Spec(k[t])/\Spec(k)\bigr)$ is smooth of relative dimension~$e$.
But $\Gamma^e_k(k[t])=\TS^e_k(k[t])=k[s_1,s_2,\dots,s_e]$ where
$s_1,s_2,\dots,s_e$ are the elementary symmetric functions.
\end{proof}

\begin{remark}
If $X/S$ is smooth of relative dimension $\geq 2$ then $\Gamma^d(X/S)$ is not
smooth for $d\geq 2$. This can be seen by an easy tangent space calculation. If
$X/S$ is smooth of relative dimension $2$ then on the other hand $\Hilb^d(X/S)$
is smooth and gives a resolution of $\Gamma^d(X/S)$~\cite[Cor.~2.6 and
Thm.~2.9]{fogarty_algfam_surf_I}. Moreover $\Hilb^d(X/S)\to \Gamma^d(X/S)$ is
a blow-up in this case~\cite{haiman_blowup,skjelnes_ekedahl_good_component}.
\end{remark}

\begin{proposition}\label{P:pushfwd-properties}
If $\map{f}{X}{Y}$ has one of the following properties, then so has
$\map{f^d/\SG{d}}{\Sym^d(X/S)}{\Sym^d(Y/S)}$:
%
%
\begin{enumerate}
\item quasi-compact
\item closed
\item open
\item universally closed
\item universally open
\item open immersion
\item affine
\item quasi-affine
\item integral
\end{enumerate}
If $f$ has one of the above properties or one of the following
\begin{enumerate}\setcounter{enumi}{9}
\item closed immersion
\item immersion
\end{enumerate}
then so has $\map{f_*}{\Gamma^d(X/S)}{\Gamma^d(Y/S)}$.
\end{proposition}
\begin{proof}
Use that $\Psi_X$ and $\map{q}{(X/S)^d}{\Sym^d(X/S)}$ are universally closed,
universally open, quasi-compact and surjective for (i)-(v). Property (vi) is
well-known.  For (vii) reduced to $Y/S$ affine using
Proposition~\pref{P:Gamma-neighborhoods} and then use that $\Gamma^d(X/S)$ and
$\Sym^d(X/S)$ are affine if $X/S$ is affine. The combination of (i), (vi) and
(vii) gives (viii). Finally (ix) follows from (vii) and (iv). The last two
properties for $f_*$ follow from Proposition~\pref{P:Gamma-immersion}.
\end{proof}

\begin{remark}
If $f$ has one of the properties (x) or (xi), then $f^d/\SG{d}$ need not have
that property. If $f$ has one of the properties
\begin{enumerate}
\item finite
\item locally of finite type
\item locally of finite presentation
\item unramified
\item flat
\item \etale{}
\end{enumerate}
then neither $f^d/\SG{d}$ nor $f_*$ need to have that property.
\end{remark}

\begin{corollary}\label{C:addition-is-integral-and-univ-open}
The addition morphism
$\map{+}{\Gamma^d(X/S)\times_S\Gamma^e(X/S)}{\Gamma^{d+e}(X/S)}$ is integral
and universally open.
\end{corollary}
\begin{proof}
The morphism $X\amalg X\to X$ is finite and \etale{} and hence integral and
universally open. Thus $\Gamma^{d+e}(X\amalg X/S)\to\Gamma^{d+e}(X/S)$ is
integral and universally open by Proposition~\pref{P:pushfwd-properties}. As
the addition morphism is the composition of the open and closed immersion
$\Gamma^d(X/S)\times_S\Gamma^e(X/S)\inj\Gamma^{d+e}(X\amalg X/S)$ and the
push-forward along $X\amalg X\to X$ the corollary follows.
\end{proof}

\end{subsection}

\end{section}


\appendix
\begin{section}{Appendix}

\begin{subsection}{The (AF) condition}\label{SS:AF}
The (AF) condition has frequently been used as a natural setting for a wide
range of problems. It guarantees the existence of finite
quotients~\cite[Exp.~V]{sga1}, push-outs~\cite{ferrand_conducteur} and
the Hilbert scheme of points~\cite{rydh_hilbert}. Moreover, under the
(AF) condition, \etale{} cohomology can be calculated using
\v{C}ech cohomology~\cite[Cor.~4.2]{artin_Hensel-join},\ \cite{schroer_bigger-Brauer}. 

\begin{definition}
We say that a scheme $X/S$ is \AF{} if it satisfies the following condition.
\renewcommand{\theequation}{AF}
\begin{equation}\label{E:AF}
\parbox{10 cm}{Every finite set of points $x_1,x_2,\dots,x_n\in X$
over the same point $s\in S$ is contained in an open subset $U\subseteq X$
such that $U\to S$ is quasi-affine.}
\end{equation}
\end{definition}

\begin{remark}
Let $X/S$ and $Y/S$ be \AF{}-schemes. Then $X\times_S Y/S$ is an
\AF{}-scheme. If $S'\to S$ is any morphism, then $X\times_S S'/S'$ is an
\AF{}-scheme. This is obvious as the class of quasi-affine morphisms is
stable under products and base change. It is also clear that the \eqref{E:AF}
condition is local on $S$ and that the subset $U$ in the condition can be
chosen such that $U$ is an affine scheme. Moreover, if $S$ is quasi-separated,
then we can replace the condition that $U\to S$ is quasi-affine with the
condition that $U$ is affine.
\end{remark}

\begin{proposition}
Let $X$ be an $S$-scheme. If $X$ has an ample invertible sheaf $\sO_X(1)$
relative to $S$ then $X/S$ is an \AF{}-scheme. In particular, it is so if $X/S$
is (quasi-)affine or (quasi-)projective.
\end{proposition}
\begin{proof}
Follows immediately from~\cite[Cor.~4.5.4]{egaII} since we can assume that
$S=\Spec(A)$ is affine.
\end{proof}

\begin{proposition}
Let $X/S$ be an \AF{}-scheme. Then $X/S$ is separated.
\end{proposition}
\begin{proof}
Let $z$ be a point in the closure of $\Delta_{X/S}(X)$, where
$\injmap{\Delta_{X/S}}{X}{X\times_S X}$ is the diagonal morphism, and let
$x_1,x_2\in X$ be its two projections. Choose an affine neighborhood $U$
containing $x_1$ and $x_2$. Then $\injmap{\Delta_{U/S}}{U}{U\times_S U}$ is
closed and $\Delta_{U/S}$ is the pull-back of $\Delta_{X/S}$ along the open
immersion $U\times_S U\subset X\times_S X$. Taking closure commutes with
restricting to open subsets and thus $z\in U\subset X$. This shows that
$\Delta_{X/S}(X)$ is closed and hence that $X/S$ is separated.
\end{proof}

The following conjecture was proved by
Kleiman~\cite{kleiman_num_theory_ampleness}.

\begin{theorem}[Chevalley's conjecture]
Let $X/k$ be a proper regular algebraic scheme. Then $X$ is projective if and
only if $X/k$ is an \AF{}-scheme.
\end{theorem}

It is however not true that a proper singular scheme always is projective if it
is \AF{}. In fact, there are singular, proper but non-projective
\AF{}-surfaces~\cite{horrocks_nonproj_surface}.
\end{subsection}


\begin{subsection}{A theorem on integral morphisms}\label{SS:tff-result}

\begin{definition}\label{D:tff}
We say that a morphism $\map{f}{X}{Y}$ has \emph{topologically finite fibers}
if the underlying topological space of every fiber is a finite set.
We say that $f$ has \emph{universally topologically finite fibers}
if the base change of $f$ by any morphism $Y'\to Y$ has topologically finite
fibers, equivalently the underlying topological space of every fiber is a
finite set and the residue field extensions has finite separable degree.
\end{definition}

The purpose of this section is to prove the following theorem:

\begin{theorem}\label{T:image_of_integral_wtff}
Let $\map{f}{X}{Y}$ and $\map{g}{Y}{S}$ be morphisms of algebraic spaces.
If $g\circ f$ is integral with topologically finite fibers and $g$ is separated
then the ``schematic'' image $Y'$ of $f$ exists and $Y'\to S$ is integral with
topologically finite fibers.
\end{theorem}

Let us first note that this is easy to proof when $g$ is locally of finite
type:

\begin{proposition}\label{P:image_of_finite}
Let $X$ and $Y$ be schemes locally of finite type and separated over the base
scheme $S$. Let $\map{f}{X}{Y}$ and $\map{g}{Y}{S}$ be $S$-morphisms. If
$g\circ f$ is finite then the schematic image $Y'$ of $f$ exists and $Y'\to S$
is finite.
\end{proposition}
\begin{proof}
As $g\circ f$ is separated, $f$ is separated. As $g\circ f$ is quasi-compact
and universally closed and $g$ is separated, $f$ is quasi-compact
and universally closed. Thus the image $Y'$ exists
\cite[Prop.~6.10.5]{egaI_NE} and $X\to Y'$ is surjective. As $g\circ f$ is
universally closed and $X\to Y'$ is surjective it follows that
$Y'\to S$ is universally closed. Further it is immediately seen that $Y'\to S$
has discrete fibers. Thus $Y'\to S$ is quasi-finite, universally closed and
separated. By Deligne's theorem~\cite[Cor.~18.12.4]{egaIV} this implies that
$Y'\to S$ is finite.
\end{proof}

\begin{remark}
It is easy to generalize Proposition~\pref{P:image_of_finite} to the case where
$X$ and $Y$ are \emph{algebraic spaces}.
In~\cite[Thm.~6.15]{knutson_alg_spaces} Deligne's theorem is proven for
algebraic spaces under a finite presentation hypothesis. The full version of
Deligne's theorem for algebraic spaces is given in~\cite[Thm.~A.2]{laumon}.
\end{remark}

\begin{remark}
Now instead assume as in Theorem~\pref{T:image_of_integral_wtff} that $X$ and
$Y$ are arbitrary schemes and $g\circ f$ is integral with topologically finite
fibers. The first part of the proof of Proposition~\pref{P:image_of_finite}
then shows as before that the schematic image $Y'$ exists and $Y'\to S$ is
separated and universally closed. It is further easily seen that every fiber
$Y'_s$ is a discrete finite topological space.

Under the hypothesis that $Y/S$ is an \AF{}-scheme it easily follows that
$Y'\to S$ is integral. In fact, then $Y'/S$ is \AF{} and any neighborhood of
$Y'_s$ in $Y'$ contains an affine neighborhood of $Y'_s$. Thus $Y'\to S$ is
affine by~\cite[Lem.~18.12.7.1]{egaIV} and therefore integral
by~\cite[Prop.~18.12.8]{egaIV}.

In general, note that $Y'_s$ is affine and hence integral over $k(s)$ as a
morphism is integral if and only if it is universally closed and affine,
cf.~\cite[Prop.~18.12.8]{egaIV}. Theorem~\pref{T:image_of_integral_wtff} thus
follows by the following conjecture of Grothendieck (for schemes):
\end{remark}

\begin{conjecture}[{\cite[Rem.~18.12.9]{egaIV}}]
If $X\to S$ is a separated, universally closed morphism of algebraic spaces,
such that $X_s$ is integral, then $X\to S$ is integral.
\end{conjecture}

This conjecture will be proved in~\cite{rydh_noetherian-approx}. In the
remainder of this appendix, we will give an independent proof of
Theorem~\pref{T:image_of_integral_wtff} without using Grothendieck's
conjecture. We first establish the following preliminary results.
\begin{enumerate}
\item If $X\to Y$ is integral, $X$ a semi-local scheme and $Y$ henselian then
$X$ is henselian, cf.\ Proposition~\pref{P:int_hensel}.
\item Affineness is descended by (not necessarily quasi-compact) flat morphisms
if we a priori know that the morphism in question is quasi-compact and
quasi-separated, cf.\ Proposition~\pref{P:flat_descends_affineness}.
\item A criterion for an algebraic space to be a scheme,
cf.\ Lemma~\pref{L:local_alg_space_is_scheme}.
\end{enumerate}

\begin{proposition}\label{P:int_hensel}
If $A$ is semi-local and henselian and $B$ is an integral semi-local
$A$-algebra, then $B$ is henselian. In particular $B$ is a finite direct
product of local henselian rings.
\end{proposition}
\begin{proof}
Follows immediately from \cite[Ch.~XI, \S2, Prop.~2]{raynaud_hensel_rings}.
\end{proof}


\begin{proposition}\label{P:flat_descends_affineness}
Let $\map{f}{X}{Y}$ and $\map{g}{Y'}{Y}$ be morphisms of schemes with $g$
faithfully flat. Let $\map{f'}{X'}{Y'}$ be the base-change of $f$ along
$g$. Then
\begin{enumerate}
\item $f$ is a homeomorphism if $f$ is quasi-compact and $f'$
is a homeomorphism.
\item $f$ is an isomorphism if and only if $f$ is quasi-compact and $f'$
is an isomorphism.
\item $f$ is affine if and only if $f$ is quasi-compact and quasi-separated and
$f'$ is affine.
\end{enumerate}
\end{proposition}
\begin{proof}
The conditions in (ii) and (iii) are clearly necessary. Assume that
$f'$ is a homeomorphism (resp. an isomorphism, resp. affine).
Let $Y''=\coprod_{y\in Y}\Spec(\sO_{Y,y})$
and choose for every $y\in Y$ a point $y'\in g^{-1}(y)$. If we let
$Y'''=\coprod_{y\in Y}\Spec(\sO_{Y',y'})$ then $f'''$ is a homeomorphism
(resp. an isomorphism, resp. affine) and 
we can factor $Y'''\to Y'\to Y$ through the natural faithfully flat and
quasi-compact morphism $Y'''\to Y''$. As the statement of the proposition
is true when $g$ is quasi-compact by
\cite[Prop.~2.6.2 (iv), Prop.~2.7.1 (viii), (xiii)]{egaIV} it follows that
$f''$ is a homeomorphism (resp. an isomorphism, resp. affine). Replacing $Y'$
with $Y''$ we can thus assume that $Y'=\coprod_{y\in Y}\Spec(\sO_{Y,y})$.

(i) In order to show that $f$ is a homeomorphism it is enough to show that
$f$ is open since it is clearly bijective. As $f$ is generizing,
see \cite[Def.~3.9.2]{egaI_NE}, it follows by \cite[Thm.~7.3.1]{egaI_NE}
that $f$ is open if and only it is open in the constructible topology
\cite[7.2.11]{egaI_NE}. But as $f$ is quasi-compact and bijective
it follows from \cite[Prop.~7.2.12 (iv)]{egaI_NE} that $f$ is a homeomorphism
in the constructible topology and in particular open.

(ii) From (i) it follows that $f$ is a homeomorphism and since $f'$ is an
isomorphism, we have that $f$ is an isomorphism on the stalks. This shows that
$f$ is an isomorphism.

(iii) Taking direct images along quasi-compact and quasi-separated morphisms
commutes with flat pull-back by \cite[Lem.~2.3.1]{egaIV}. Thus we have a
cartesian diagram:
$$\xymatrix{
{X'}\ar[r]\ar[d] & {\Spec(f'_{*}\sO_{X'})}\ar[r]\ar[d] & {Y'}\ar[d] \\
{X}\ar[r]\ar@{}[ur]|{\square} & {\Spec(f_{*}\sO_{X})}\ar[r]
   & {Y}\ar@{}[ul]|{\square}
}$$
Since $\map{f'}{X'}{Y'}$ is affine we have that $X'\to\Spec(f'_{*}\sO_{X'})$ is
an isomorphism and it is enough to show that $X\to\Spec(f_{*}\sO_{X})$ is an
isomorphism. This follows from (ii).
\end{proof}

\begin{definition}
We say that an algebraic space $X$ is \emph{local} if there exist a point
$x\in X$ such that every closed subset $Z\subseteq X$ contains $x$.
\end{definition}

\begin{remark}
If $X$ is a local algebraic space then there is exactly one closed point
$x\in X$. If $X$ is a local scheme then $X$ is the spectrum of a local ring and
in particular affine.
\end{remark}

\begin{lemma}\label{L:disj_union}
Let $\map{f}{X}{Y}$ be a closed surjective morphism of algebraic spaces.  Let
$y\in Y$ be a closed point such that $f^{-1}(y)$ is discrete and such that for
any $x\in f^{-1}(y)$ we can write $X=X'_x\amalg X''$ where $X'_x$ is local and
contains $x$. Then $Y=Y'\amalg Y''$ where $Y'$ is local and contains
$y$. Furthermore $f^{-1}(Y')=\coprod_{x\in f^{-1}(y)} X'_x$.
\end{lemma}
\begin{proof}
For every $x\in f^{-1}(y)$ let $X'_x\subseteq X$ be a local subspace containing
$x$ and choose $X''$ such that $X=\left(\coprod_{x\in f^{-1}(y)} X'_x\right)
\amalg X''$.  Let $Y'$ be the subset of $Y$ consisting of every generization of
$y$.  As $f(X'')$ is closed and does not contain $y$, it does not intersect
$Y'$. On the other hand $f(X'_x)$ is contained in $Y'$. Since $f$ is surjective
this shows that $f(X'')=Y\mysetminus Y'$ and $f(\bigcup X'_x)=Y'$. Thus $Y'$
and $Y''=Y\mysetminus Y'$ are both open and closed.
\end{proof}

\begin{lemma}\label{L:local_alg_space_is_scheme}
Let $X=\coprod_{\alpha\in \I} X_\alpha$ and $Y$ be algebraic spaces such that
$X_\alpha$ is local with closed point $x_\alpha$ and $Y$ is local with closed
point $y$. Let $\map{f}{X}{Y}$ be a universally closed schematically dominant
morphism such that $f^{-1}(y)=\{x_\alpha\;:\;\alpha\in\I\}$. If $X_\alpha$ is
a henselian scheme for every $\alpha\in\I$ then $Y$ is affine.
\end{lemma}
\begin{proof}
There is an \etale{} quasi-compact separated surjective morphism
$\map{g}{Y'}{Y}$ such that $Y'$ is a scheme and such that there is a point
$y'\in g^{-1}(y)$ with $k(y')=k(y)$. Let $X'=X\times_Y Y'$ with projections
$\map{h}{X'}{X}$ and $\map{f'}{X'}{Y'}$.  Similarly we let
$X'_\alpha=X_\alpha\times_Y Y'$ and we have that
$X'=\coprod_{\alpha\in\I}X'_\alpha$. As $k(y')=k(y)$ we have that
$f'^{-1}(y')=\{x'_\alpha\}$ such that $x'_\alpha\in X'_\alpha$ and
$h(x'_\alpha)=x_\alpha$.

Since $X_\alpha$ is henselian and $h$ is quasi-finite
and separated it follows by \cite[Thm.~18.5.11 c)]{egaIV} that
$\Spec(\sO_{X'_\alpha,x'_\alpha})\to X_\alpha$ is finite and that
$\Spec(\sO_{X'_\alpha,x'_\alpha})\subseteq X'$ is open and closed.
Further as $X_\alpha$
is henselian, $k(x'_\alpha)=k(x_\alpha)$ and $X'_\alpha\to X_\alpha$ is
\etale{} it follows that $\Spec(\sO_{X'_\alpha,x'_\alpha})\to X_\alpha$ is an
isomorphism. By Lemma~\pref{L:disj_union} we then have a decomposition
$Y'=Y'_1\amalg Y'_2$ where $Y'_1$ is local and
$f'^{-1}(Y'_1)=\coprod_\alpha \Spec(\sO_{X'_\alpha,x'_\alpha}) \iso X$. Thus
we can, replacing $Y'$ with $Y'_1$, assume that $Y'$ is a local scheme and
$X'\iso X$.

Let $Y''=Y'\times_Y Y'$, which is a quasi-affine scheme, and
$X''=X\times_Y Y''=X'\times_X X'\iso X$. Lemma~\pref{L:disj_union} shows as
before that $Y''$ is local and hence affine. Let
$Y'=\Spec(A')$, $Y''=\Spec(A'')$, $X'=\Spec(B')$ and $X''=\Spec(B'')$ where
$B''=B'$. As $A'\inj A''$ is faithfully flat it follows that $A''/A'$ is a
flat $A'$-algebra. Further $A'\to B'$ is injective since $X\to Y$ is
schematically dominant. Thus $A''/A'\inj (A''/A')\otimes_{A'} B'=B''/B'=0$
which shows that $A''=A'$. This shows that $Y$ is the quotient of the \etale{}
equivalence relation
$\xymatrix@M=1pt{{\Spec(A'')} \ar@<.5ex>[r] \ar@<-.5ex>[r] & \Spec(A')}$ where
the two morphisms are the identity. Thus $Y=\Spec(A')$ is a local
\emph{scheme}.
\end{proof}

\begin{proof}[Proof of Theorem~\pref{T:image_of_integral_wtff}]
As $g\circ f$ is separated, $f$ is separated. As $g\circ f$ is quasi-compact
and universally closed and $g$ is separated, $f$ is quasi-compact
and universally closed. Thus the image $Y'$ exists
\cite[Prop.~6.10.5]{egaI_NE} and \cite[Prop.~4.6]{knutson_alg_spaces} and
$X\to Y'$ is surjective. As $g\circ f$ is
universally closed and $X\to Y'$ is surjective it follows that
$Y'\to S$ is universally closed. Further it is obvious that $Y'\to S$
has topologically finite fibers.

Since the question is local over $S$, we can assume that $S$ is affine. Then
$X$ is affine and we will show that $Y'\to S$ is \emph{affine}.
It then follows that $Y'\to S$ is integral since
$\sO_S\to g_*\sO_{Y'}\inj g_* f_*\sO_X$ is integral.

Using Proposition~\pref{P:flat_descends_affineness} we are allowed to replace
$S$ with the henselization $\Spec(\hensel{\sO_{S,s}})$ at an arbitrary point
$s$ and thus assume that $S$ is local and henselian. Then by
Proposition~\pref{P:int_hensel} $X$ is henselian and a disjoint union of local
schemes.

Let $x_1,x_2,\dots,x_n$ be the closed points of $X$ and
$X=X_1\amalg X_2\amalg\dots\amalg X_n$ the corresponding partition into
local henselian schemes. Then by Lemma~\pref{L:disj_union}
$Y=Y_1\amalg Y_2\amalg\dots\amalg Y_m$ where $Y_k$ is a local space with
closed point $y_k\in f(x_j)$ for some $j$ depending on $k$. Further
Lemma~\pref{L:local_alg_space_is_scheme} shows that $Y_k$ is a local
\emph{scheme} and hence affine.
\end{proof}

\end{subsection}

\end{section}


\bibliography{famzerocycles}
\bibliographystyle{dary}

\end{document}